\documentclass[12pt]{amsart}
\newcommand{\deleted}[1]{}
\newcommand{\delete}[1]{}
\newcommand{\mynotes}[1]{}
\newcommand\notes[1]{}
\usepackage{txfonts}
\usepackage{amssymb}
\usepackage{eucal}
\usepackage{graphicx}
\usepackage{amsmath}
\usepackage{amscd}
\usepackage[all]{xy}           

\usepackage{amsfonts,latexsym}
\usepackage{xspace}
\usepackage{epsfig}
\usepackage{float}
\usepackage{color}
\usepackage{fancybox}
\usepackage{colordvi}
\usepackage{multicol}
\usepackage{tikz}

\usepackage{wasysym}
\usepackage[active]{srcltx} 
\usepackage{xspace}
\usepackage[colorlinks,final,backref=page,hyperindex,hypertex]{hyperref}

\newcommand\changed[1]{#1}
\changed{}

\newtheorem{theorem}{Theorem}[section]
\newtheorem{lemma}[theorem]{Lemma}
\newtheorem{coro}[theorem]{Corollary}

\newtheorem{conjecture}[theorem]{Conjecture}
\newtheorem{prop}[theorem]{Proposition}
\theoremstyle{definition}
\newtheorem{defn}[theorem]{Definition}
\newtheorem{remark}[theorem]{Remark}
\newtheorem{exam}[theorem]{Example}
\newtheorem{claim}[theorem]{Claim}
\deleted{\newtheorem{theorem}{Theorem}[section]
\newtheorem{prop-def}{Proposition-Definition}[section]
\newtheorem{coro-def}{Corollary-Definition}[section]

}

\newcommand{\nc}{\newcommand}
\nc{\tred}[1]{\textcolor{red}{#1}} \nc{\tblue}[1]{\textcolor{blue}{#1}} \nc{\tgreen}[1]{\textcolor{green}{#1}} \nc{\tpurple}[1]{\textcolor{purple}{#1}} \nc{\btred}[1]{\textcolor{red}{\bf #1}} \nc{\btblue}[1]{\textcolor{blue}{\bf #1}} \nc{\btgreen}[1]{\textcolor{green}{\bf #1}} \nc{\btpurple}[1]{\textcolor{purple}{\bf #1}}

\renewcommand{\Bbb}{\mathbb}

\newcommand{\efootnote}[1]{}

\newcommand\wyscco[1]{}
\renewcommand{\textbf}[1]{}

\nc{\mlabel}[1]{\label{#1}}  
\nc{\mcite}[1]{\cite{#1}}  
\nc{\mref}[1]{\ref{#1}}  
\nc{\mbibitem}[1]{\bibitem{#1}} 

\renewcommand\geq{\geqslant}
\renewcommand\leq{\leqslant}
\renewcommand\preceq{\preccurlyeq}

\newcommand{\paren}[1]{$($#1$)$}
\renewcommand\bar[1]{\overline{#1}}
\renewcommand\tilde[1]{\widetilde{#1}}

\nc\kdot{\bfk\,}
\nc\simple{simple\xspace}

\newcommand\mapmZss{\mapm{Z}^{\star_1,\star_2}}
\nc{\rbw}{\mathfrak{R}} \nc{\brp}{\mathrm{brp}} \nc{\lead}{\mathrm{Lead}} \nc{\Id}{\mathrm{Id}} \nc{\Irr}{\mathrm{Irr}} \nc{\vx}{\sigma} \nc{\vy}{\tau} \nc{\dvx}{\sigma^{(1)}} \nc{\dvy}{\tau^{(1)}} \nc{\done}{\vep} \nc{\citep}[1]{\cite{#1}} \nc{\wt}{\mathrm{wt}} \nc{\bre}[1]{|#1|} \nc{\mapmonoid}{\frakM} \nc{\disjoint}{\frakM'}
\nc{\ncpoly}[1]{\langle #1\rangle}  
\nc{\mapm}[1]{\lfloor\!|{#1}|\!\rfloor}
\nc{\diff}[1]{{}^\NC\{ #1 \}} \nc{\disj}[1]{\{{#1}\}'} \nc{\mdisj}[1]{\frakM'(#1)} \nc{\brho}{\bar{\rho}} \nc{\om}{\bar{\frakm}} \nc{\frakn}{\mathfrak n} \nc{\ddeg}[1]{^{(#1)}} \nc{\opset}{X} \nc{\genset}{{Z}} \nc{\NC}{\mathrm{{NC}}} \nc{\leaf}{\mathrm{leaf}} \nc{\twig}{\mathrm{twig}} \nc{\fe}{\mathrm{fl}} \nc{\munderline}[1]{#1} \nc{\bo}{o} \nc{\dep}{\mathrm{depth}} \nc{\ofe}{\mathrm{ofl}} \nc{\dfe}{\mathrm{dfe}} \nc{\fex}{\mathrm{fex}} \nc{\dl}{\mathrm{dlex}} \nc{\db}{\mathrm{db}} \nc{\lex}{\mathrm{lex}} \nc{\clex}{\mathrm{clex}} \nc{\dgp}{\mathrm{dgp}} \nc{\dgx}{\mathrm{dgx}} \nc{\br}{\mathrm{br}} \nc{\obd}{\mathrm{odb}} \nc{\ob}{\mathrm{ob}}
\nc{\loc}{location\xspace}
\nc{\occ}{occurrence\xspace}
\nc{\occs}{occurrences\xspace}
\nc{\pla}{placement\xspace}
\nc{\plas}{placements\xspace}

\nc{\bin}[2]{ (_{\stackrel{\scs{#1}}{\scs{#2}}})}  
\nc{\binc}[2]{ \left (\!\! \begin{array}{c} \scs{#1}\\
    \scs{#2} \end{array}\!\! \right )}  
\nc{\bincc}[2]{  \left ( {\scs{#1} \atop
    \vspace{-1cm}\scs{#2}} \right )}  
\nc{\bs}{\bar{S}} \nc{\cosum}{\sqsubset} \nc{\la}{\longrightarrow} \nc{\rar}{\rightarrow} \nc{\dar}{\downarrow} \nc{\dprod}{**} \nc{\dap}[1]{\downarrow \rlap{$\scriptstyle{#1}$}} \nc{\md}{\mathrm{dth}} \nc{\uap}[1]{\uparrow \rlap{$\scriptstyle{#1}$}} \nc{\defeq}{\stackrel{\rm def}{=}} \nc{\disp}[1]{\displaystyle{#1}} \nc{\dotcup}{\ \displaystyle{\bigcup^\bullet}\ } \nc{\gzeta}{\bar{\zeta}} \nc{\hcm}{\ \hat{,}\ } \nc{\hts}{\hat{\otimes}} \nc{\barot}{{\otimes}} \nc{\free}[1]{\widetilde{#1}} \nc{\uni}[1]{\tilde{#1}} \nc{\hcirc}{\hat{\circ}} \nc{\leng}{\ell} \nc{\lleft}{[} \nc{\lright}{]} \nc{\lc}{\lfloor} \nc{\rc}{\rfloor}
\nc{\lb}{[} 
\nc{\rb}{]} 
\nc{\curlyl}{\left \{ \begin{array}{c} {} \\ {} \end{array}
    \right.  \!\!\!\!\!\!\!}
\nc{\curlyr}{ \!\!\!\!\!\!\!
    \left. \begin{array}{c} {} \\ {} \end{array}
    \right \} }
\nc{\longmid}{\left | \begin{array}{c} {} \\ {} \end{array}
    \right. \!\!\!\!\!\!\!}
\nc{\onetree}{\bullet} \nc{\ora}[1]{\stackrel{#1}{\rar}}
\nc{\ola}[1]{\stackrel{#1}{\la}}
\nc{\ot}{\otimes} \nc{\mot}{{{\boxtimes\,}}} \nc{\otm}{\overline{\boxtimes}} \nc{\sprod}{\bullet} \nc{\scs}[1]{\scriptstyle{#1}} \nc{\mrm}[1]{{\rm #1}} \nc{\msum}{\sum\limits}
\nc{\margin}[1]{\marginpar{\rm #1}}   
\nc{\dirlim}{\displaystyle{\lim_{\longrightarrow}}\,} \nc{\invlim}{\displaystyle{\lim_{\longleftarrow}}\,} \nc{\mvp}{\vspace{0.3cm}} \nc{\tk}{^{(k)}} \nc{\tp}{^\prime} \nc{\ttp}{^{\prime\prime}} \nc{\svp}{\vspace{2cm}} \nc{\vp}{\vspace{8cm}} \nc{\proofbegin}{\noindent{\bf Proof: }}
\nc{\proofend}{$\blacksquare$ \vspace{0.3cm}}
\nc{\modg}[1]{\!<\!\!{#1}\!\!>}
\nc{\intg}[1]{F_C(#1)} \nc{\lmodg}{\!<\!\!} \nc{\rmodg}{\!\!>\!} \nc{\cpi}{\widehat{\Pi}}
\nc{\sha}{{\mbox{\cyr X}}}  
\nc{\shap}{{\mbox{\cyrs X}}} 
\nc{\shpr}{\diamond}    
\nc{\shp}{\ast} \nc{\shplus}{\shpr^+}
\nc{\shprc}{\shpr_c}    
\nc{\msh}{\ast} \nc{\zprod}{m_0} \nc{\oprod}{m_1} \nc{\vep}{\varepsilon} \nc{\labs}{\mid\!} \nc{\rabs}{\!\mid}
\nc{\astarrow}{\overset{\raisebox{-2pt}{{\scriptsize $\ast$}}}{\rightarrow}}
\nc{\lastarrow}{\overset{\raisebox{-2pt}{{\scriptsize $\ast$}}}{\leftarrow}}
\nc{\mastarrow}[1]{\overset{\raisebox{-2pt}{{\scriptsize $#1$}}}{\rightarrow}}
\nc{\quvarrow}[3]{#1 \overset{q,u,v}{\longrightarrow}_{#3} #2}
\nc{\quvkto}[1]{f_{#1} \overset{q_{#1}, u_{#1}, v_{#1}}{\longrightarrow}_\phi g_{#1}}
\nc{\tvarrow}[3]{#1 \overset{(t,v)}{\longrightarrow}_{#3} #2}
\nc{\Supp}{{\rm Supp}}
\nc{\mpu}{u^{\ast}}
\nc{\mpv}{v^{\ast}}
\nc{\mpw}{w^{\ast}}
\nc{\mpx}{x^{\ast}}
\nc{\dps}{\dotplus}

\nc{\dth}{d} \nc{\mmbox}[1]{\mbox{\ #1\ }} \nc{\fp}{\mrm{FP}} \nc{\rchar}{\mrm{char}} \nc{\Fil}{\mrm{Fil}} \nc{\Mor}{Mor\xspace} \nc{\gmzvs}{gMZV\xspace} \nc{\gmzv}{gMZV\xspace} \nc{\mzv}{MZV\xspace} \nc{\mzvs}{MZVs\xspace} \nc{\Hom}{\mrm{Hom}} \nc{\id}{\mrm{id}} \nc{\im}{\mrm{im}} \nc{\incl}{\mrm{incl}} \nc{\map}{\mrm{Map}} \nc{\mchar}{\rm char} \nc{\nz}{\rm NZ} \nc{\supp}{\mathrm Supp}
\nc{\mo}{\mathbf o}
\nc{\pl}{\mathfrak{p}}

\nc{\Alg}{\mathbf{Alg}} \nc{\Bax}{\mathbf{Bax}} \nc{\bff}{\mathbf f} \nc{\bfk}{{\bf k}} \nc{\bfone}{{\bf 1}} \nc{\bfx}{\mathbf x} \nc{\bfy}{\mathbf y}
\nc{\base}[1]{\bfone^{\otimes ({#1}+1)}} 
\nc{\Cat}{\mathbf{Cat}} \delete{}
\nc{\detail}{\marginpar{\bf More detail}
    \noindent{\bf Need more detail!}
    \svp}
\nc{\Int}{\mathbf{Int}} \nc{\Mon}{\mathbf{Mon}}
\nc{\rbtm}{{shuffle }} \nc{\rbto}{{Rota-Baxter }} \nc{\remarks}{\noindent{\bf Remarks: }} \nc{\Rings}{\mathbf{Rings}} \nc{\Sets}{\mathbf{Sets}}
\nc{\vwpt}{{Let $V$ be a free $\bfk$-module with a $\bfk$-basis $W$ and let $\Pi$ be a \simple term-rewriting system on $V$ with respect to $W$.}\xspace}

\nc{\BA}{{\Bbb A}} \nc{\CC}{{\Bbb C}} \nc{\DD}{{\Bbb D}} \nc{\EE}{{\Bbb E}} \nc{\FF}{{\Bbb F}} \nc{\GG}{{\Bbb G}} \nc{\HH}{{\Bbb H}} \nc{\LL}{{\Bbb L}} \nc{\NN}{{\Bbb N}} \nc{\KK}{{\Bbb K}} \nc{\QQ}{{\Bbb Q}} \nc{\RR}{{\Bbb R}} \nc{\TT}{{\Bbb T}} \nc{\VV}{{\Bbb V}} \nc{\ZZ}{{\Bbb Z}}

\nc{\cala}{{\mathcal A}} \nc{\calc}{{\mathcal C}} \nc{\cald}{{\mathcal D}} \nc{\cale}{{\mathcal E}} \nc{\calf}{{\mathcal F}} \nc{\calg}{{\mathcal G}} \nc{\calh}{{\mathcal H}} \nc{\cali}{{\mathcal I}} \nc{\call}{{\mathcal L}} \nc{\calm}{{\mathcal M}} \nc{\caln}{{\mathcal N}} \nc{\calo}{{\mathcal O}} \nc{\calp}{{\mathcal P}} \nc{\calr}{{\mathcal R}} \nc{\cals}{{\mathcal S}} \nc{\calt}{{\mathcal T}} \nc{\calw}{{\mathcal W}}
\nc{\calv}{{\mathcal V}}
\nc{\calk}{{\mathcal K}} \nc{\calx}{{\mathcal X}} \nc{\CA}{\mathcal{A}}

\nc{\fraka}{{\mathfrak a}} \nc{\frakA}{{\mathfrak A}} \nc{\frakb}{{\mathfrak b}} \nc{\frakB}{{\mathfrak B}} \nc{\frakD}{{\mathfrak D}} \nc{\frakH}{{\mathfrak H}} \nc{\frakM}{{\mathfrak M}} \nc{\bfrakM}{\overline{\frakM}} \nc{\frakm}{{\mathfrak m}} \nc{\frkP}{{\mathfrak P}}
\nc{\frakN}{{\mathfrak N}} \nc{\frakp}{{\mathfrak p}}
\nc{\frakQ}{{\mathfrak Q}}\nc{\frakR}{{\mathfrak R}} \nc{\frakS}{{\mathfrak S}}
\nc{\frakx}{{\mathfrak x}} \nc{\ox}{\bar{\frakx}} \nc{\frakX}{{\mathfrak X}} \nc{\fraky}{{\mathfrak y}}
\nc\dop{\delta}
\nc{\Reduce}{{\rm Red}}

\font\cyr=wncyr10 \font\cyrs=wncyr7
\nc{\redt}[1]{\textcolor{red}{#1}}

\nc{\li}[1]{\textcolor{red}{Li:#1}} 
\nc{\lio}[1]{}
\nc{\sz}[1]{\textcolor{green}{sz:#1}}
\nc{\szo}[1]{}
\nc{\xg}[1]{\textcolor{purple}{xg:#1}}
\nc{\ws}[1]{\textcolor{blue}{{#1}}} 
\nc{\wsc}[1]{\textcolor{blue}{ws:#1}} 
\nc{\wsco}[1]{}
\nc{\wsn}[1]{\textcolor{magenta}{#1}} 

\renewcommand{\theenumi}{{\it\alph{enumi}}}

\begin{document}
\title{Rota-Baxter type operators, rewriting systems and
Gr\"obner-Shirshov bases}

\author{Xing Gao}
\address{School of Mathematics and Statistics,
Key Laboratory of Applied Mathematics and Complex Systems,
Lanzhou University, Lanzhou, Gansu, 730000, P.R. China}
\email{gaoxing@lzu.edu.cn}

\author{Li Guo}
\address{
    Department of Mathematics and Computer Science,
         Rutgers University,
         Newark, NJ 07102, USA}
\email{liguo@rutgers.edu}

\author{William Y. Sit}
\address{Dept. of Math., The City College of The City
University of New York, New York, NY 10031, USA} \email{wyscc@sci.ccny.cuny.edu}

\author{Shanghua Zheng}
\address{Department of Mathematics,
    Lanzhou University,
    Lanzhou, Gansu 730000, China}
\email{zheng2712801@163.com}

\hyphenpenalty=8000
\date{\today}

\begin{abstract}
In this paper we apply the methods of rewriting systems and Gr\"obner-Shirshov bases to give a unified approach to a class of linear operators on associative algebras. These operators resemble the classic Rota-Baxter operator, and they are called {\it Rota-Baxter type operators}. We characterize a Rota-Baxter type operator by the convergency of a rewriting system associated to the operator. By associating such an operator to a Gr\"obner-Shirshov basis, we obtain a canonical basis for the free algebras in the category of associative algebras with that operator. This construction include as special cases several previous ones for free objects in similar categories, such as those of Rota-Baxter algebras and Nijenhuis algebras.
\end{abstract}

\delete{
\begin{keyword}
Rota's Problem; rewriting systems, Gr\"obner-Shirshov bases; operators; classification; Rota-Baxter type operators.
\end{keyword}
}

\maketitle

\tableofcontents

\hyphenpenalty=8000 \setcounter{section}{0}


\section{Introduction}
Many years ago, G.-C.~Rota~\mcite{Ro2} posed the question of finding all the algebraic identities that could be satisfied by {some} linear operator defined on {some} associative algebra.
He wrote:
\begin{quote}
In a series of papers, I have tried to show that other linear operators satisfying algebraic identities may be of equal importance in studying certain algebraic phenomena, and I have posed the problem of finding all possible algebraic identities that can be satisfied by a linear operator on an algebra. Simple computations show that the possibility are very few, and the problem of classifying all such identities is very probably completely solvable. 
\end{quote}

Rota was most interested in the following operators arising from analysis, probability and combinatorics:
\begin{eqnarray*}
 \text{Endomorphism operator} &\quad&
d(xy)=d(x)d(y), \\
 \text{Differential operator} &\quad&
d(xy)=d(x)y+xd(y), \\
 \text{Average operator} &\quad&
P(x)P(y)=P(xP(y)),  \\
\text{Inverse average operator} &\quad& P(x)P(y)=P(P(x)y),
\\
 \text{(Rota-)Baxter operator} &\quad&
 P(x)P(y)=P(xP(y)+P(x)y+\lambda xy),\\
{\text{of weight\ } \lambda}&\quad &\quad \text{ where } \lambda \text{ is a fixed constant},\\
\text{Reynolds operator}
  &\quad& P(x)P(y)=P(xP(y)+P(x)y-P(x)P(y)). \\
\end{eqnarray*}

The importance of the endomorphism operator is well-known for the role automorphisms (bijective endomorphisms) play in Galois theory. The differential operator is essential in analysis and its algebraic generalizations led to the development of differential algebra \citep{Kol,SP}, difference algebra~\citep{Co,AL}, and quantum differential operators~\citep{LR}. The other operators are also important. For example, the Rota-Baxter operator, which originated from probability study~\citep{Ba}, is closely related to the classical Yang-Baxter equation, as well as to operads, to combinatorics and, through the Hopf algebra framework of Connes and Kreimer, to the renormalization of quantum field theory \citep{Ag,AGKO,BBGN,Bai,C-K1,EGK,EGM,EG1,Gub,G-K1,G-Z}.

In recent years, new linear operators have emerged from algebraic studies, combinatorics, and physics~\mcite{CGM,G-K3,Le}. Examples are:
\begin{eqnarray*}
\text{Differential operator of weight} ~ \lambda &\quad&
 d(xy)=d(x)y+x d(y)+\lambda d(x)d(y), \\
&\quad& \quad\text{ where } \lambda \text{ is a fixed constant}, \\
\text{Nijenhuis operator} &\quad& P(x)P(y)=P(xP(y)+P(x)y-P(xy)),\\
\text{Leroux's TD operator} &\quad&
P(x)P(y)=P(xP(y)+P(x)y-xP(1)y).\\
\end{eqnarray*}
These operators in the above two lists can be grouped into two classes. The first two operators in the first list and the first operator in the second satisfy an identity of the form $d(xy)=N(x,y),$ where $N(x,y)$ is some algebraic expression involving $x, y$, and the operator $d$. They belong to the class of {\bf differential type operators} (where $N(x,y)$ is required to satisfy some extra conditions), so called because of their resemblance to the differential operator. The remaining operators satisfy an identity of the form $P(x)P(y)=P(B(x,y))$ where $B(x,y)$ is some algebraic expression involving $x, y$, and the operator $P$.  These belong to the class of {\bf Rota-Baxter type operators} (where $B(x,y)$ is required to satisfy some extra conditions), so called because of their resemblance to the Rota-Baxter operator.

It is interesting to observe that the above operators of differential type share similar properties. Their free objects are constructed in the same way and their studies in general follow parallel paths. The same can be said {of Rota-Baxter} type operators. After the free objects of Rota-Baxter algebras were constructed in~\mcite{EG1,Gub}, similar constructions have been obtained for free objects {for Nijenhuis algebras}~\mcite{LG} and {for TD algebras}~\mcite{ZhC}. Likewise, the constructions of free commutative Nijenhuis algebras and free commutative TD algebras in~\mcite{EL} are similar to the construction {for free} commutative Rota-Baxter algebras in~\mcite{G-K1}. Other instances of similar constructions can be found in~\mcite{Agg,Ca,ZG}. Furthermore, these operators share similar applications: for example, {for the double structures in mathematical physics (especially in} {the Lie algebra context)}~\mcite{Bai2,BGN,STS} {and for the splitting of associativity in} {mathematics}~\mcite{Ag,BBGN,Lo}. It will be helpful to study these two classes of operators under one theory. On the one hand, we will be able to treat all operators of each type uniformly; for instance in the construction of their free objects. On the other hand, we may discover other operators in these two classes, eventually give a complete list of these operators, and make some progress towards solving Rota's problem.

Following this approach, a systematic investigation on differential type operators is carried out in~\mcite{GSZ} by studying the operated polynomial identities they satisfy in the framework of operated algebras~\mcite{Gop}. These identities are then characterized by means of their rewriting systems~\mcite{BN} and associated Gr\"obner-Shirshov bases~\mcite{BCC,BCL,BCQ}.

A conjectured list of Rota-Baxter type operators {is} provided in~\mcite{GSZ} based on symbolic computation done in~\mcite{Sit}. The study of Rota-Baxter type operators, however, is more challenging than their differential counterpart as can be expected already by comparing integral calculus with differential calculus. Nevertheless the method of Gr\"obner-Shirshov bases has been successfully applied to the study of Rota-Baxter algebras, differential Rota-Baxter algebras and integro-differential algebras~\mcite{BCD,BCQ,GGZ}. We show in this paper that the methods of rewriting {systems} and Gr\"obner-Shirshov bases apply more generally to Rota-Baxter type algebras as well. As consequences we obtain free objects in these operated algebra categories and verify that the operators in the above-mentioned conjectured list are indeed of Rota-Baxter type.

In Section~\mref{sec:back}, we associate, to each operated polynomial identity $\phi(x,y)=0$ of a certain form, a family of rewriting systems on free operated algebras, and define a linear operator satisfying that identity to be of Rota-Baxter type if the rewriting systems have some additional properties. We also restate the conjectured list of 14 Rota-Baxter type algebras announced in~\mcite{GSZ}. In Section~\ref{sec:rbcr}, we show that a linear operator is of Rota-Baxter type if and only if the rewriting systems associated with the identity it satisfies are convergent. In Section~\mref{sec:GSu}, we introduce the notion of a monomial order on free operated algebras that are compatible with the rewriting systems, which enables us to characterize Rota-Baxter type algebras in terms of Gr\"obner-Shirshov bases. We show that the linear operator is of Rota-Baxter type precisely when the set of operated polynomials derived from $\phi$ is a Gr\"obner-Shirshov basis. When this is the case, we give an explicit construction of a free object in the category of operated algebras satisfying the identity $\phi=0$.  In Section~\mref{sec:evi}, we establish a monomial order needed in Section~\mref{sec:GSu} and verify that the identities in the conjectured list indeed define Rota-Baxter type operators and algebras. Thus, we have achieved a uniform construction of the free objects for all the 14 categories of operated algebras whose defining identities are listed in the conjecture. Our construction generalizes and includes as special cases the known constructions for various operated algebras~\mcite{BCD,BCQ,Ca,EG1,LG}.

Our characterization of Rota-Baxter type operators and identities in terms of Gr\"obner-Shirshov bases and convergent rewriting systems reveals the power of this general approach. It would be interesting to further apply rewriting system and  Gr\"obner-Shirshov bases techniques to study these operators, with the resolution of Rota's classification problem in mind.

\noindent
{\bf Convention. } Throughout this paper, we fix a commutative unitary ring $\bfk$. By an algebra we mean an associative (but not necessarily commutative) unitary $\bfk$-algebra, unless the contrary is specified. Following common terminology, a non-unitary algebra means one that may not have an identity element.

\section{Rota-Baxter type operators and rewriting systems}
\mlabel{sec:back}

In this section, we recall the construction of free operated algebras that gives operated polynomial identity algebras. We also obtain results on term-rewriting systems for free $\bfk$-modules. These concepts and results provide us with a framework to define Rota-Baxter type operators for algebras and to give a conjectured list of these operators together with the identity each must satisfy. They also prepare us for our main tasks in later sections.

We begin by reviewing some background on operated algebras.

\subsection{Free operated algebras}

The construction of free operated algebras was given in~\mcite{Gop,GSZ}. See also~\mcite{BCQ}. We reproduce that construction here to review the notation.

\begin{defn}
{\rm An {\bf operated monoid} (resp. {\bf operated $\bfk$-algebra}, resp. {\bf operated $\bfk$-module}) is a monoid (resp. $\bfk$-algebra, resp. $\bfk$-module) $U$ together with a map (resp. $\bfk$-linear map, resp. $\bfk$-linear map) $P: U\to U$. A morphism from an operated monoid\, (resp. $\bfk$-algebra, resp. $\bfk$-module) $U$ to an operated monoid (resp. $\bfk$-algebra, resp. $\bfk$-module) $V$ is a monoid (resp. $\bfk$-algebra, resp. $\bfk$-module) homomorphism $f :U\to V$ such that $f \circ P= P \circ f$. } \mlabel{de:mapset}
\end{defn}

Let $Y$ be a set, let $M(Y)$ be the free monoid on $Y$ with identity $1$, and let $S(Y)$ be the free semigroup on $Y$. Let $\lc Y\rc:=\{\lc y\rc \,|\, y\in Y\}$ denote a set indexed by $Y$, but disjoint from $Y$.

Let $X$ be a given set. We will construct the free operated monoid over  $X$ as the limit of a directed system $$\{\,\iota_{n}: \mapmonoid_n\to \mapmonoid_{n+1}\, \}_{n=0}^\infty$$ of free monoids $\mapmonoid_n:=\mapmonoid_n(X)$, where the transition morphisms $\iota_{n}$ will be natural embeddings. For this purpose, let $\mapmonoid_0=M(X)$, and let
$$ \mapmonoid_1:=M(X\cup \lc \mapmonoid_0\rc).$$ Let $\iota_{0}$
be the natural embedding $\iota_{0}:\mapmonoid_0 \hookrightarrow     \mapmonoid_1$. 
Assume by induction that for some $n\geq 0$, we have defined, for $0\leq i\leq n+1$, the free monoids $\mapmonoid_i$  with the properties that for $0 \leq i\leq n$, we have $\mapmonoid_{i+1}=M(X\cup \lc\mapmonoid_{i}\rc )$
and natural embeddings  $ \iota_i: \mapmonoid_i \to \mapmonoid_{i+1}$.  Let
\begin{equation}
 \mapmonoid_{n+2}:=M(X\cup \lc\mapmonoid_{n+1}\rc ).
 \mlabel{eq:frakm}
 \end{equation}
The identity map on $X$ and the embedding $\iota_{n}$ together induce an injection
\begin{equation}
\iota_{n+1}: X\cup \lc\mapmonoid_{n}\rc \hookrightarrow
    X\cup \lc \mapmonoid_{n+1} \rc,
\mlabel{eq:transet}
\end{equation}
which, by the functoriality of $M$, extends to an embedding (still denoted by $\iota_{n+1}$) of free monoids 
\begin{equation}
 \iota_{n+1}: \mapmonoid_{n+1} = M(X\cup \lc\mapmonoid_{n}\rc)\hookrightarrow
    M(X\cup \lc \mapmonoid_{n+1}\rc) = \mapmonoid_{n+2}. \mlabel{eq:tranm}
\end{equation}

This completes our inductive definition of the {directed} system. Let
$$ \mapm{X}:=\frakM(X):=\bigcup_{n\geq 0}\frakM_n=\dirlim
\frakM_n$$ be the direct limit of the system. Elements of $\frakM_n\backslash \frakM_{n-1}$ are said to have {\bf depth} $n$. We note that $\frakM(X)$ is a monoid, and by taking direct limit on both sides of $\frakM_n = M(X\cup \lc \frakM_{n-1}\rc)$, we obtain
\begin{equation}\frakM(X)=M(X\cup \lc \frakM(X)\rc).
\mlabel{eq:omid}\end{equation}

Let $\kdot\frakM(X)$ be the (free) $\bfk$-module with basis
$\frakM(X)$.
 Since the basis is a monoid, the
multiplication on $\frakM(X)$ can be extended via linearity to turn the
$\bfk$-module $\kdot\frakM(X)$ into a $\bfk$-algebra, which
we denote by $\bfk\frakM(X)$. Similarly, we can extend the operator
$\lc\ \rc: \frakM(X) \to \frakM(X)$, which takes $w \in \frakM(X)$ to
$\lc w\rc$, to an operator $P$ on $\bfk\frakM(X)$ by $\bfk$-linearity
and turn the $\bfk$-algebra $\bfk\frakM(X)$ into an operated
$\bfk$-algebra, which we shall denote by $\bfk\mapm{X}$ (or by abuse,
$\bfk\frakM(X)$, since as sets, $\frakM(X)=\mapm{X}$). If $X$ is a finite set, we may also just list its elements, as in $\bfk\mapm{x,y}$ when $X = \{x, y\}$.

\begin{lemma}{\bf \mcite{Gop}}
Let $i_X:X \to \frakM(X)$ and $j_X: \frakM(X) \to \bfk\mapm{X}$ be the natural embeddings. Then, with structures as above,
\begin{enumerate}
\item
the triple $(\frakM(X),\lc\ \rc, i_X)$ is the free operated monoid on $X$; and
\mlabel{it:mapsetm}
\item
the triple $(\bfk\mapm{X},P, j_X\circ i_X)$ is the free operated unitary $\bfk$-algebra
on $X$. \mlabel{it:mapalgsg}
\end{enumerate}
\mlabel{pp:freetm}
\end{lemma}

{\it For the rest of this paper}, we will use the infix notation
$\lc r \rc$ interchangeably with $P(r)$ for any $r \in R$ where $R$
is an operated algebra with operator $P$; for example, when $R =
\bfk\mapm{X}$.

\begin{defn} Elements of $\frakM(X)$ are called {\bf bracketed words} or
{\bf bracketed monomials in $X$}. An element $\phi \in \bfk\mapm{X}$ will be
called an {\bf operated} or {\bf bracketed polynomial in $X$ with coefficients in
$\bfk$}, and we will implicitly assume that $\phi \notin \bfk$,
unless otherwise noted. When there is no danger of confusion, we often omit the adjective ``bracketed.''  \mlabel{def:bwbp}\end{defn}

The following notions will be needed for Sections \ref{sec:GSu} and
\ref{sec:evi}.

\begin{defn}
Let $u\in \frakM(X)$, $u \ne 1$. By Eq.\,(\mref{eq:omid}), we may
 write $u$ as a product $v_1\cdots v_k$
uniquely for some $k$ with $v_i\in X\cup \lc \frakM(X)\rc$ for
$1\leq i\leq k$. We call $k$ the {\bf breadth} of $u$ and denote it
by $|u|$. If $u = 1 \in \frakM(X)$, we define $|u| = 0$.
Alternatively, by combining adjacent
factors of $u=v_1 \dots v_k$ that belong to $X$ into a monomial
belonging to $M(X)$ and by inserting $1 \in M(X)$ between two adjacent factors
of the form $\lc \frakx \rc$ where $\frakx \in \frakM(X)$, we may
write $u$ uniquely in the canonical form
\begin{equation}
u=u_0\lc \mpu_1\rc u_1 \cdots \lc \mpu_r \rc u_r, {\rm\ where\ } u_0,\cdots, u_r\in
M(X) {\rm\ and\ } \mpu_1,\cdots, \mpu_r \in \frakM(X). \mlabel{eq:decom}
\end{equation}
We define the {\bf $P$-breadth} of $u$ to be $r$ and denote it by
$|u|_P$. Note that $|u|_P = 0$ if and only if $u =u_0 \in M(X)$. We
further define the {\bf operator degree} $\deg_P(u)$ of a monomial
$u$ in $\bfk\mapm{\opset}$ to be the total number of occurrences of
the operator $\lc\ \rc$ in the monomial $u$.\mlabel{def:breadths}
\end{defn}

\subsection{Operated polynomial identity algebras}

Let $k \geq 1$ and $X=\{x_1,x_2,\cdots,x_k\}$, let $\phi:=\phi_{\lc\ \rc}(x_1,\cdots,x_k)\in \bfk\mapm{X}$. We call $x_1,\cdots,x_k$ the {\bf argument variables} and $\lc\ \rc$ the {\bf operator variable}. When $\lc\ \rc$ does not appear in $\phi$, then $\phi=\phi(x_1,\cdots,x_k)$ is just a polynomial and its evaluation can be defined as usual by specializing the argument variables $x_1,\cdots,x_k$. We next define its evaluation in general. Let $R$ be an operated algebra with operator $P$, and let $r = (r_1, \cdots, r_k) \in R^k$. By the universal property of the free operated algebra $\bfk\mapm{X}$, the map $f_r: \{x_1,\cdots,x_k\}\to R$ that sends $x_i$ to $r_i$ induces a unique morphism $\free{f_r}:\bfk\mapm{x_1,\cdots,x_k}\to R$ of operated algebras that extends $f_r$. Define the {\bf evaluation map} $\phi_{(R,P)}:R^k\to R$ by: \footnote{Later in this paper, we may simplify the notation $\phi_{(R, P)}$ to $\phi_R$, or emphasize the multiplication $\ast$ of $R$ and replace $(R, P)$ by the triple $(R,\ast, P)$. On the other hand, if there is no possibility of confusion of what $R$ and $P$ should be, we may write $\phi(r_1, \dots, r_k)$ for $\phi_{(R, P)}(r_1, \dots, r_k)$.\mlabel{fn:algebra}}
\begin{equation}
\phi_{(R,P)}(r_1,\cdots,r_k) := \free{f_r}(\phi_{\lc\,\rc}(x_1,\cdots,x_k)), \quad r = (r_1, \cdots, r_k) \in R^k. \mlabel{eq:phibar}
\end{equation}

We call $\phi_{(R, P)}(r_1,\cdots,r_k)$ the {\bf evaluation of the operated polynomial $\phi_{\lc\ \rc}(x_1,\cdots,x_k)$} at the {\bf point $(r_1,\cdots,r_k)$} with {\bf operator $P$.} When $\lc \ \rc$ does not appear in $\phi$, this reduces to the usual notion of evaluation of a polynomial at the point $(r_1,\cdots,r_k)$.

\begin{defn}
Let $\phi\in \bfk\mapm{x_1,\cdots,x_k}$. We say that an operated algebra $R$ with operator $P$ is a {\bf $\phi$-algebra} and that $P$ is a {\bf $\phi$-operator}, if ${\phi}_{(R,P)}(r_1,\dots,r_k)=0$ for all $r_1,\dots,r_k\in R$. An {\bf operated polynomial identity algebra} is any $\phi$-algebra for some $\phi$. If $R$ is a $\phi$-algebra, we will say loosely that $\phi=0$ (or by abuse, $\phi$) is an {\bf operated polynomial identity} ({\bf OPI}) satisfied by $R$. \mlabel{de:pio}
\end{defn}

\begin{defn}
Given any set $\genset$, and a subset $S \subset \bfk\mapm{Z}$, the {\bf operated ideal $\Id(S)$ of ~$\bfk\mapm{Z}$ generated by $S$} is the smallest operated ideal containing $S$.
\mlabel{de:repgen}
\end{defn}

Given any set $\genset$, let $I_\phi(\genset)$ be the operated ideal of $\bfk\mapm{Z}$ generated by the set
\begin{equation}
S_\phi(Z) := \left\{\,{\phi}_{(\bfk\mapm{Z},\lc\,\rc)}(u_1,\cdots,u_k) \mid u_1,\cdots,u_k\in \bfk\mapm{Z}\,\right\}.\mlabel{eq:genphi}\end{equation} Then the quotient algebra $\bfk_\phi\mapm{Z} := \bfk\mapm{Z}/I_\phi(\genset)$ has a natural structure of a $\phi$-algebra.

\begin{prop} \cite[Theorem~3.5.6]{BN} Let $\phi\in
\bfk\mapm{x_1,\cdots,x_k}$. Given any set $\genset$, the quotient operated algebra $\bfk_\phi\mapm{Z}$ is the free $\phi$-algebra on $\genset$. \mlabel{pp:frpio}
\end{prop}

The following definitions are adapted from~\mcite{BCQ,GSZ} and supersede those in~\mcite{GSZ}, where the definitions were only preliminary.

\begin{defn}
Let $Z$ be a set, let $\star$ be a symbol not in $Z$, and let $Z^\star = Z \cup \{\star\}$. By a {\bf $\star$-bracketed word} (respectively, {\bf $\star$-bracketed polynomial}) on $Z$, we mean any word in $\mapm{Z^\star}=\frakM({Z^\star})$ (respectively, polynomial in $\kdot\frakM(Z^\star)$) with exactly one\footnote{counting  multiplicities; thus $q = \star^2$ and $q = \lc \star \rc^2$ are not $\star$-bracketed words.} occurrence of $\star$,  The set of all $\star$-bracketed words (respectively, $\star$-bracketed expressions) on $Z$ is denoted by $\mapm{Z}^\star$ or $\frakM^{\star}(Z)$ (respectively, $\bfk^\star\mapm{Z}$). \mlabel{def:starbw}
\end{defn}

Let $q\in \mapm{Z}^\star$ and $u \in  \frakM({Z})$. We will use $q|_{u}$ or $q|_{\star \mapsto u}$ to denote the bracketed word on $Z$ obtained by replacing the symbol $\star$ in $q$ by $u$. Next, we extend by linearity this notion to elements $s=\sum_i c_i u_i\in \bfk\frakM{(Z)}$, where $c_i\in\bfk$ and $u_i\in \frakM{(Z)}$, that is, we define in this case $q|_s$ to be the bracketed expression:
\begin{equation*}
 q|_{s}:=\sum_i c_i q|_{u_i}\,.
\end{equation*}  Finally, we extend again by linearity this notation
to any $q \in \bfk^\star\mapm{Z}$.
Note that in either of these generalized settings, $q|_s$ is usually not a
bracketed word but a bracketed polynomial.

With the above notation, we can now describe the operated ideal $\Id(S)$ generated by a subset $S \subseteq \bfk\mapm{Z}$. It is given \mcite{BCQ,GSZ} by
\begin{equation}\Id(S) = \left\{ \sum_{i=1}^k c_i q_i|_{s_i}\,\Big|\ k\geq 1
{\rm\ and\ } c_i\in \bfk,
q_i\in \frakM^{\star}(Z), s_i\in S {\rm\ for\ } 1\leq i\leq k\right\}.
\mlabel{eq:repgen}\end{equation}
Note that neither the $q_i$'s nor the $s_i$'s ($1 \leq i \leq k$) appearing in the above summation expression need be distinct.

\begin{defn}A bracketed word $u \in \frakM(Z)$ is a {\bf subword} of another
bracketed word $w \in \frakM(Z)$ if $w  = q|_u$ for some $q \in
\frakM^\star(Z)$, where the specific occurrence of $u$ in $w$ is
defined by $q|_u$ (that is, by the $\star$ in $q$). To make this
more precise, the set of character positions (when a bracketed word is viewed
as a string of characters) occupied by the subword $u$ in the word
$w$ under the substitution $q|_u$ is called the {\bf placement of
$u$ in $w$ by $q$}. We denote this placement by the pair $(u,q)$. A subword $u$ may appear at multiple locations (and hence have distinct placements using distinct $q$'s) in a
bracketed word $w$. \mlabel{de:subword}
\end{defn}

\begin{exam}
Let $Z=\{ x, y\,\}$. Consider placements in the monomial $w = \lc x y x y\rc$.
\begin{enumerate}
\item The subword $u:=x$ appears at two locations in $w$. Their placements are $(u,q_1)$ and $(u,q_2)$ where $q_1=\star\lc x\rc$ and $q_2=x\lc \star\rc$.
\item   The placement of $x$ in $w$ by
$q_1 = \lc xy\star y \rc$ is included as a (proper) subset of the
placement of $xy$ in $w$ by $q_2 = \lc xy \star \rc$. We say $(x, q_1)$ and $(x,q_2)$ are {\it nested} in $w$.
\item The two placements for the subword $x$ are disjoint, as are the two for $xy$. We say in each case, the two placements are {\it separated in $w$}.
\item Each of the two placements for $xy$ overlaps partially the unique
placement of $yx$ in $w$. We say they are {\it intersecting}. The two placements of $xx$ in $xxx$
are also intersecting.
\end{enumerate}
The notions illustrated by these examples will be formally defined later under Definition \mref{defn:bwrel}.
\mlabel{ex:subwords}
\end{exam}

\subsection{Term-rewriting on free $\bfk$-modules}
\mlabel{sec:TRS}

In this subsection, we recall some basic definitions and develop new results for term-rewriting systems when they are specialized for free $\bfk$-modules with a given basis and satisfy a simple condition. In the next subsection, we apply them to the Rota-Baxter term rewriting systems on operated $\bfk$-algebras.

\begin{defn} Let $V$ be a free $\bfk$-module with a given $\bfk$-basis $W$.  For $f \in V$, when $f$ is expressed as a unique linear combination of $w \in W$ with coefficients in $\bfk$, the {\bf support} $\Supp(f)$ of $f$ is the set consisting of $w \in W$ appearing in $f$ (with non-zero coefficients).
Let $f,g\in V$. We use $f \dps g$ to indicate the relation that $\Supp(f) \cap \Supp(g) = \emptyset$. If this is the case, we say $f + g$ is a {\bf direct sum} of $f$ and $g$, and by abuse,\footnote{Whether $\dps$ refers to the relation or the direct sum will always be clear from the context.} we use $f\dps g$ also for the sum $f+ g$. \mlabel{def:dps}
\end{defn}

Note $\Supp(0) = \emptyset$ and hence $f \dps 0$ for any $f \in V$. We record the following obvious properties of $\dotplus$.
\begin{lemma} Let $V$ be a free $\bfk$-module with a $\bfk$-basis $W$. Let $a,b,c \in \bfk$, let $w \in W$, and let $f, g, h \in V$. Then
\begin{enumerate}
\item The basis element $w$ is not in $\Supp(f)$ if and only if $w+f=w\dps f$. \mlabel{item:add1}
\item
If $f \dotplus g$, then $af\dotplus bg$, and if furthermore $f \dotplus h$, then $af\dotplus (bg + ch)$.
\mlabel{item:add2}
\end{enumerate}
\mlabel{lem:add}
\end{lemma}

\begin{defn} Let $V$ be a free $\bfk$-module with a $\bfk$-basis $W$. For $f \in V$ and $w \in \Supp(f)$, let the coefficient of $w$ in $f$ be $c_w$. We define the {\bf $w$-complement of $f$} to be $R_w(f) := f - c_w w \in V$, so that $f = c_w w \dps R_w(f)$. \mlabel{item:wcomp}
\end{defn}

\begin{defn} Let $V$ be a free $\bfk$-module with a $\bfk$-basis $W$.  A {\bf term-rewriting system $\Pi$ on $V$ with basis $W$} is a binary relation $\Pi \subseteq W \times V$.
\begin{enumerate}
\item We say the rewriting system $\Pi$ is {\bf \simple} if $t \dps v$ for all $(t,v)\in \Pi$.

\item The image $\pi_1(\Pi)$ of $\Pi$ under the first projection map $\pi_1: W \times V \to W$ will be denoted by $T$ or $T(\Pi)$. A {\bf reducible term} (under $\Pi$) is any element $t \in T$.

\item For $(t, v) \in \Pi \subseteq T \times V$, we write $t \to_\Pi v$ and view this as a rewriting rule on $V$, that is, if $f \in V$, $t \in \Supp(f)$ and $c_t \in \bfk$ is the coefficient of $t$ in $f$, then we may apply the rule to $f$ by replacing $t$ with $v$, resulting in a new element $g:= c_t v + R_t(f) \in V$
 and say $f$ {\bf reduces to}, or {\bf rewrites to}, $g$ {\bf in one-step} and indicate any such one-step rewriting by $f \to_\Pi g$, or in more detail, by $\tvarrow{f}{g}{\Pi}$.  \mlabel{item:Trule}

\item The reflexive transitive closure of $\rightarrow_\Pi$ (as a binary relation on $V$) will be denoted by $\astarrow_\Pi$ and we say {\bf $f$ reduces to $g$ with respect to $\Pi$} if $f \astarrow_\Pi g$. \mlabel{item:rtcl}

\item We say {\bf $f \in V$ reduces to $g \in V$ with respect to $\Pi$ in $n$ steps} ($n \geq 1$) and denote this by $f \mastarrow{n}_\Pi g$ if there exist $f_0, f_1, \dots, f_n \in V$ such that $f_i \neq f_{i+1}$ for $i=0, \dots, n-1$ and $f= f_0 \rightarrow_\Pi f_1 \rightarrow_\Pi \cdots \rightarrow_\Pi f_n = g$. An element $f \in V$ is {\bf reducible}\footnote{In certain context, there may be several term-rewriting systems under discussion, in which case, we use {\it $\Pi$-reducible} for reducible. Similar modification will be used for other terms defined below.} if $f \mastarrow{n}_\Pi g$ for some $g \in V$ and $n \geq 1$, otherwise we say $f$ is {\bf irreducible} or {\bf in normal form}. We extend this notation by convention to $f \mastarrow{0} g$, which means $f = g$, and includes, although not necessarily, the case when $f$ is irreducible.
    \mlabel{item:rednstep}

\item Two elements  $f, g \in V$ are {\bf joinable} if there exist $p \in V$ such that $f \astarrow_\Pi p$ and $g \astarrow_\Pi p$; we denote this by $f \downarrow_\Pi g$. If $f, g$ are joinable, the {\bf joinable distance $d^\vee_\Pi(f,g)$ between $f$ and $g$} is the minimum of $m+n$ over all possible $p$ and reductions to $p$, that is,
\begin{equation}
d^\vee_\Pi(f,g):=\min\left\{m+n\,\mid\, \exists p \in V {\rm\ such\ that\ } f \mastarrow{m}_\Pi p,\  g \mastarrow{n}_\Pi p,\  m \geq 0, {\rm\ and\ } n \geq 0\ \right\}.
\mlabel{eq:dist}
\end{equation}
\end{enumerate}
\mlabel{def:ARSbasics}\end{defn}

\begin{defn}
A (term) rewriting system $\Pi$ on $V$ is called
 \begin{enumerate}
 \item {\bf normalizing} if every $f \in V$ reduces to a (not necessarily unique) element in normal form;
 \item {\bf terminating} if there is no infinite chain of one-step reductions $f_0 \rightarrow_\Pi f_1 \rightarrow_\Pi f_2 \cdots $;
\item {\bf confluent} (resp. {\bf locally confluent}) if every fork (resp. local fork) is joinable; and
\item {\bf convergent} if it is both terminating and confluent.
\end{enumerate}
\mlabel{def:ARS}
\end{defn}

\begin{remark} The property that a term-rewriting system is simple is a very weak condition. For example, $\Pi = \{(x,y),(y,x)\}$ for the $\bfk$-submodule $V$ with basis $W=\{x,y\}$ is simple. Every element of $V$ is reducible, none has a normal form, and $\Pi$ is neither normalizing nor terminating, but is confluent. For later applications, further restrictions will be imposed so that, for example, the rewriting system is consistent with the algebraic structure. See Definition~\mref{def:redsys}.
\label{rmk:twpi}
\end{remark}

A well-known result on rewriting systems is Newman's Lemma \cite[Lemma 2.7.2]{BN}.
\begin{lemma} \paren{Newman} A terminating rewriting system is confluent if and only if it is locally confluent. \mlabel{lem:newman}\end{lemma}

\begin{prop} \vwpt For any $f, g \in V$, consider the following properties:
\begin{enumerate}
\item $f \to_\Pi g$. \mlabel{item:eq1}
\item $(f-g) \to_\Pi 0$. \mlabel{item:eq2}
\item $(f - g) \astarrow_\Pi 0$. $($equivalently,  $(f - g) \downarrow_\Pi 0$$)$. \mlabel{item:eq3}
\item $f \downarrow_\Pi g$.\mlabel{item:eq4}
\end{enumerate}
Then (\mref{item:eq1}) $\implies$ (\mref{item:eq2}) $\implies$ (\mref{item:eq3}) $\implies$ (\mref{item:eq4}).
\mlabel{pp:red}
\end{prop}

\begin{proof} (\mref{item:eq1}) $\implies$ (\mref{item:eq2}): Let $f = ct \dps R_t(f)$ and $g = cv +R_t(f)$, where $(t,v) \in \Pi$. Then since $t \dps v$, we have $f - g = ct \dps (-c v) \to_\Pi cv - cv = 0.$
\smallskip

\noindent
(\mref{item:eq2}) $\implies$ (\mref{item:eq3}): and the equivalence in (\mref{item:eq3}) are obvious.
\smallskip

\noindent
(\mref{item:eq3}) $\implies$ (\mref{item:eq4}): By hypothesis, $(f -g) \mastarrow{n}_\Pi 0$ for some $n \geq 0$. We prove (\mref{item:eq3}) $\implies$ (\mref{item:eq4}) by induction on $n$. If $n = 0$, then $f=g$ and hence $f \downarrow_\Pi g$. Suppose $n \geq 1$. Then $(f-g) \to_\Pi h$ for some $h \in V$,  $h \ne f-g$, and $h \overset{n-1}{\longrightarrow}_\Pi 0$.  More specifically, there exist $(t,v) \in \Pi$ and $c_t \in \bfk$, $c_t \ne 0$, such that
$f-g = c_t t \dps R_t(f-g)$ and $h = c_t v + R_t(f-g)$.  Now we may write $f = a t \dps R_t(f)$ and $g = b t \dps R_t(g)$, where at most one of $a, b$ may be zero. Then $f - g = (a - b)t \dps (R_t(f)- R_t(g))$ by Lemma \mref{lem:add}(\mref{item:add2}), and hence $c_t = a - b$ and $R_t(f-g) = R_t(f) - R_t(g)$. Then rewriting
\begin{eqnarray*}
h& =&c_t t + R_t(f-g) - c_t (t - v)\\
&=&(f-g) - (a - b) (t - v)\\
&=& (f - a(t - v)) - (g - b(t - v)),
\end{eqnarray*}
and noting that $h \overset{n-1}{\longrightarrow}_\Pi 0$, we have by induction that $(f - a(t - v)) \downarrow_\Pi (g - b(t - v))$. Now $$f = a t \dps R_t(f) \astarrow_\Pi a v + R_t(f) = f - a(t - v)$$ and similarly $g \astarrow_\Pi g - b(t - v)$, so $f \downarrow_\Pi g$.
\end{proof}

There are examples to show that the implications in Proposition \mref{pp:red} are all one-way, thus providing a strict hierarchy of binary relations on $V$.

We next give a more general result then the implication (a) $\implies$ (d) in Proposition~\mref{pp:red}.
\begin{lemma} \vwpt  Let $f, g, h \in V$. If $f \to_\Pi g$, then
 $(f+h) \downarrow_\Pi (g+h)$. \mlabel{lem:join1}
\end{lemma}
\begin{proof} By Definition
\mref{def:ARSbasics}(\mref{item:Trule}), $f\to_\Pi g$ means that there exist $(t,v) \in \Pi$ and $0\neq c \in\bfk$ such that $f=c t \dotplus R_t(f)$ and $g=c v + R_t(f)$.  Let $h = b t \dps R_t(h)$, where $b \in \bfk$ ($b$ may be zero). Since $\Pi$ is \simple, we have $t \dps v$. Since $t \dps R_t(f)$,  we have $t \dps g$ by Lemma~\mref{lem:add}. Applying Lemma~\mref{lem:add} again, we get
\begin{equation}
g+h=bt \dps (g+R_t(h))\astarrow_\Pi bv+(g+R_t(h))=(b+c)v+R_t(f)+R_t(h).
\mlabel{eq:add}
\end{equation}
On the other hand, we also have
\begin{equation}
f+h=(b+c)t \dps (R_t(f)+R_t(h))\astarrow_\Pi (b+c)v +R_t(f)+R_t(h).
\mlabel{eq:redeq1}
\end{equation}
By Eqs. (\mref{eq:add}) and (\mref{eq:redeq1}), $(f+h) \downarrow_\Pi (g+h)$.
\end{proof}

The following is a key result for applications in later sections.
\begin{theorem} \vwpt  Consider the following properties on $\Pi$:
\begin{enumerate}
\item $\Pi$ is confluent, that is, for any $f,g,h \in V$, $$(f \astarrow_\Pi g,\ f \astarrow_\Pi h) \implies g \downarrow_\Pi h.$$ \mlabel{item:confluence}
\item For all $f, g, h \in V$, $$f \downarrow_\Pi g,\  g \downarrow_\Pi h \implies f \downarrow_\Pi h.$$ \mlabel{item:transitivity}
\item For all $f, g, f', g' \in V$, $$f \downarrow_\Pi g,\  f' \downarrow_\Pi g' \implies (f + f') \downarrow_\Pi (g+g').$$ \mlabel{item:2-additivity}
\item For all $r \geq 1$ and $f_1, \dots, f_r, g_1, \dots, g_r \in V$, $$f_i \downarrow_\Pi g_i  \quad (1 \leq i \leq r) {\rm\ and\ } \sum_{i=1}^r g_i = 0 \implies \left(\sum_{i=1}^r f_i \right)  \astarrow_\Pi 0.$$ \mlabel{co:red0}\mlabel{item:add0transpose}
\end{enumerate}
Then (\mref{item:confluence}) $\implies$ (\mref{item:transitivity}) $\implies$ (\mref{item:2-additivity}) $\implies$ (\mref{item:add0transpose}).
\mlabel{thm:downarrow}
\end{theorem}

Much more can be said about a simple term-rewriting system. For example the above properties are equivalent. \footnote{Further discussions are left out for limit of space, but can be included if the referee or editor prefers.}
\begin{proof} (\mref{item:confluence}) $\implies$  (\mref{item:transitivity}):
Since $f \downarrow_\Pi g$, there exists $f'\in V$ such that $f\astarrow_\Pi f'$ and $ g\astarrow_\Pi f'$. Similarly, since $g \downarrow_\Pi h$, there exists $h'\in V$ such that $g\astarrow_\Pi h'$ and $h\astarrow_\Pi h'$. Then $(g\astarrow_\Pi f', g\astarrow_\Pi h')$ is a fork and since $\Pi$ is confluent, $f' \downarrow_\Pi h'$. Therefore $f \downarrow_\Pi h$.

\smallskip
\noindent(\mref{item:transitivity}) $\implies$  (\mref{item:2-additivity}): We first consider the special case when $f'=g'$ both of which are denoted by $h'$.  Let $d = d^\vee_\Pi(f,g)$, and let $m, n \in \NN$ be such that $m+n = d$ and by minimality there exist distinct $f_0, f_1, \dots, f_m \in V$ and distinct $g_0, g_1, \dots, g_n \in V$ with $$f = f_0 \to_\Pi f_1\to_\Pi \cdots \to_\Pi f_m, \quad g = g_0 \to_\Pi g_1\to_\Pi \cdots \to_\Pi g_n,$$ and $f_m=g_n$. If $d=0$,  then $f=g$ and clearly $(f+h') \downarrow_\Pi (g+h')$. If $d = 1$, then either $f\to_\Pi g$ or $g \to_\Pi f$ and these cases follow from Lemma \mref{lem:join1}. Suppose now $d=s+1$ where $s\geq 1$, and suppose by induction that for all $\tilde{f}, \tilde{g} \in V$,
$$\tilde{f} \downarrow_\Pi \tilde{g}, \ d^\vee_\Pi(\tilde{f},\tilde{g}) \leq s \implies (\tilde{f}+h')\downarrow_\Pi (\tilde{g}+h').$$   Since $d \geq 2$, either $m\geq 1$ or $n\geq 1$ (or both). Without loss of generality, we assume $m \geq 1$. Then $f_1 \downarrow_\Pi g$ and $d^\vee_\Pi(f_1,g)\leq s$. By the induction hypothesis, $(f_1+h') \downarrow_\Pi (g+h')$. It follows by Lemma \mref{lem:join1} that $(f+h') \downarrow_\Pi (f_1+h')$, and by assumption (\mref{item:transitivity}) that $(f+h') \downarrow_\Pi (g+h')$. This completes the induction for the special case.

Applying the special case with $h' = f'$, we get $(f+f') \downarrow_\Pi (g+f')$, while applying the special case with $h' = g$, we get $(g+f') \downarrow_\Pi (g+g')$. Thus, by assumption (\mref{item:transitivity}), we have $(f+f') \downarrow_\Pi (g+g')$.

\smallskip\noindent(\mref{item:2-additivity}) $\implies$  (\mref{item:add0transpose}): An inductive argument shows that,
for all $r \geq 1$ and $f_1, \dots, f_r, g_1, \dots, g_r \in V$, $$f_i \downarrow_\Pi g_i  \quad (1 \leq i \leq r) \implies \left(\sum_{i=1}^r f_i \right)\  \big\downarrow_\Pi\  \left(\sum_{i=1}^r g_i \right).$$
Indeed the case $r=1$ is trivial and the case $r=2$ holds by assumption (\mref{item:2-additivity}). For $r > 2$, by induction, we may assume that  $(\sum_{i=1}^{r-1} f_i) \downarrow_\Pi (\sum_{i=1}^{r-1} g_i)$ and by the case $r=2$,
$(\sum_{i=1}^{r-1} f_i) + f_r \downarrow_\Pi (\sum_{i=1}^{r-1} g_i)+ g_r$.

Then (\mref{item:add0transpose}) follows since $\left (\sum_{i=1}^r f_i\right) \downarrow_\Pi 0$ implies $\left (\sum_{i=1}^r f_i\right ) \astarrow_\Pi 0.$
\end{proof}

We now introduce a finer concept of confluence. 
\begin{defn} A {\bf local term-fork} is a fork $(c t \to_\Pi c v_1, c t \to_\Pi c v_2)$ where $(t,v_1), (t, v_2) \in \Pi$ and $c \in \bfk$, $c \ne 0$. The rewriting system $\Pi$ is {\bf locally term-confluent} if for every local term-fork $(c t \to_\Pi c v_1, c t \to_\Pi c v_2)$, we have $c (v_1 - v_2) \astarrow_\Pi 0$.\footnote{By Proposition \mref{pp:red}, this is a stronger condition than $cv_1 \downarrow_\Pi cv_2$. On the other hand, this is like Buchberger's $S$-polynomials reducing to zero for Gr\"obner basis.}
\end{defn}

\begin{lemma} Let $V$ be a free $\bfk$-module with a $\bfk$-basis $W$ and let $\Pi$ be a \simple term-rewriting system on $V$. Suppose we have a well-order $\preceq$ on $W$ with the property that, for all $(t,v)\in \Pi$, we have $v\prec t$ in the sense that $w\prec t$ (that is, $w\preceq t$ but $w\neq t$) for all $w\in \Supp(v)$. If $\Pi$ is locally term-confluent, then it is locally confluent. \mlabel{lem:ltcon}
\end{lemma}

\begin{proof} Let $(g \to_\Pi f, g \to_\Pi h)$ be a local fork in $V$. Then there exist $(t_1, v_1), (t_2, v_2) \in \Pi$ such that
$g \overset{(t_1, v_1)}{\longrightarrow}_\Pi f$ and $g \overset{(t_2, v_2)}{\longrightarrow}_\Pi  h$.

To prove $f \downarrow_\Pi h$, first suppose $t_1 \ne t_2$. Without loss of generality, we may suppose $t_1 \succ t_2$. Then we may write $g = c_1 t_1 \dps (c_2 t_2 \dps r) = c_2 t_2 \dps (c_1 t_1 \dps r)$ for some $r \in V$, $c_1, c_2 \in \bfk$ and $c_1 \ne 0$, $c_2 \ne 0$. Then $f = c_1 v_1 + (c_2 t_2 \dps r)$ and $h = c_2 v_2 + (c_1 t_1 \dps r)$. Hence $f - h = c_1(v_1 - t_1) + c_2(t_2 - v_2)$. Since $t_1 \succ w_1$ for any $w_1\in\Supp(v_1)$ and $t_1 \succ t_2 \succ w_2$ for any $w_2\in\Supp(v_2)$, $f - h \overset{(t_1, v_1)}{\longrightarrow}_\Pi c_2(t_2 - v_2) \overset{(t_2,v_2)}{\longrightarrow}_\Pi 0$. By Proposition \mref{pp:red}, we have $f \downarrow_\Pi h$.

Next, we suppose $t_1 = t_2$. Writing $t:= t_1=t_2$ and $g = c t \dps R_t(g)$ for some $c \in \bfk$ and $c\neq 0$, we have $f = c v_1 + R_t(g)$ and $h = c v_2 + R_t(g)$. By hypothesis, the local term-fork $(ct \to_\Pi cv_1, ct \to_\Pi cv_2)$   implies that $f - h = c v_1 - cv_2 \astarrow_\Pi 0$. By Proposition \mref{pp:red}, $f \downarrow_\Pi h$.
\end{proof}

\subsection{Rota-Baxter term-rewriting}
\mlabel{ss:RBTR}
We now apply the general results from the last subsection to the rewriting process from a Rota-Baxter type OPI.
Except for those of differential type, which have been considered in~\mcite{GSZ}, and the Reynolds operator, the class of Rota-Baxter type OPI will include the operated identities that interested Rota~\mcite{Ro2} and were listed in the introduction. Later in the paper, we will use rewriting systems and Gr\"obner-Shirshov bases to give an explicit construction of the free $\phi$-algebra for some Rota-Baxter type OPI $\phi$.

We apply the general setup in Section~\mref{sec:TRS} to a Rota-Baxter type OPI.

\begin{defn}
Let $\phi(x,y)\in \bfk\mapm{x,y}$ be an OPI of the form $\lc x\rc \lc y\rc - \lc B(x,y)\rc $, where $B(x,y) \in \bfk\mapm{x,y}$.
\begin{enumerate}
\item A {\bf Rota-Baxter term} or {\bf a $\phi$-term} is a bracketed monomial $m \in \mapm{Z}$ of the form $q|_{\lc u \rc \lc v \rc}$, where $u, v \in \mapm{Z}$ and $q \in \frakM^\star(Z)$. Such a triple $(q,u,v)$ is called a {\bf representation of $m$}. The set of all representations of $m \in \mapm{Z}$ will be denoted by $\Phi_m$.
If this is the case, we shall call $(q,u,v)$ a {\bf triple}. For $f \in \bfk\mapm{Z}$, if $w = q|_{\lc u \rc \lc v \rc} \in \Supp(f)$ and the coefficient of $w$ in $f$ is $c \in \bfk$ with $c\neq 0$, then the $w$-complement $R_w(f):=f - cw$ of $f$ will also be denoted by $R_{q,u,v}(f)$ and we call it the {\bf $(q,u,v)$-complement of $f$}.

\item The {\bf Rota-Baxter $\phi$-rewriting system (RB$\phi$RS)} is the set $\Pi_\phi(Z)$ of rewriting rules in the sense of Definition~\mref{def:ARSbasics}, when we take $W:=\mapm{Z}, V=\bfk\mapm{Z}$ and
\begin{equation}
\Pi_\phi:=\Pi_\phi(Z):=\left\{\left .\, (q|_{\lc u\rc \lc v\rc}, q|_{\lc B(u,v)\rc})\,\right |\, q \in \frakM^\star(Z), u, v \in \frakM(Z) \right \} \subseteq \mapm{Z}\times \bfk\mapm{Z}.
\mlabel{eq:T0}
\end{equation}
\mlabel{def:redrules}

\item
We say $f\rightarrow_{\Pi_\phi} g$ (in words, {\bf $f$ reduces, or {\bf rewrites}, to $g$ with respect to $\Pi_{\phi}$ in one step}) if
there are elements $q\in \frakM^\star(Z)$, $c \in \bfk$ ($c \ne 0$), and $u, v \in \frakM(Z)$ such that
\begin{enumerate}
\item
$q|_{\lc u\rc \lc v\rc} \in \Supp(f)$, and $f = c q|_{\lc u\rc \lc v\rc} \dps R_{q,u,v}(f)$; and
\item
$g=c q|_{\lc B(u,v)\rc} + R_{q,u,v}(f)$.
\end{enumerate}
In other words, $f\rightarrow_{\Pi_\phi} g$ if for some $u, v \in \frakM(Z)$, $g$ is obtained from $f$ by replacing {\it exactly once} a subword $\lc u\rc \lc v\rc$ in {\it one} monomial $t \in \Supp(f)$ by $\lc B(u,v)\rc$. When we want to emphasize the parameters involved in this reduction, we write $\quvarrow{f}{g}{\Pi_\phi}$. \mlabel{item:red1step}
\end{enumerate}
\mlabel{def:redsys}
\end{defn}

Then Eq.~(\mref{eq:T0}) can be expressed as
\begin{equation}
\Pi_\phi:= \Pi_\phi(Z) := \left \{\left.\, \quvarrow{q|_{\lc u\rc \lc v\rc}}{q|_{\lc B(u,v)\rc}}{\Pi} \,\right| q \in \frakM^\star(Z), u, v \in \frakM(Z)\right \}. \mlabel{eq:T}
\end{equation}
In the following, we shall also denote $\to_{\Pi_\phi}$ (resp. $\quvarrow{}{}{\Pi_\phi}$, $\astarrow_{\Pi_\phi}$) simply by $\to_\phi$ (resp. $\quvarrow{}{}{\phi}$, $\astarrow_\phi$). We also abbreviate Eq. (\mref{eq:T}) by
\begin{equation}
\Pi_\phi:= \Pi_\phi(Z) := \left \{\left .\,\lc u\rc \lc v\rc \rightarrow_\phi \lc B(u,v)\rc \,\right| u, v \in \frakM(Z) \right \}. \mlabel{eq:Sphi}
\end{equation}

\begin{defn} Let $W$ be a subset of $\mapm{Z}$ and let $V$ be the free $\bfk$-submodule of $\bfk\mapm{Z}$ with basis $W$. We say the rewriting system $\to_\phi$ on $\bfk\mapm{Z}$ {\bf restricts to} a rewriting system $\to_\Pi$ on $V$ with basis $W$ if for all $f \in V$, $u, v \in \mapm{Z}$ and $q \in \frakM^\star(Z)$ such that $q|_{\lc u \rc\lc v \rc}$ is in $\Supp(f)$, we have $\quvarrow{f}{g}{\phi}$  implies $g \in V$.
\mlabel{de:rwrest}
\end{defn}

\begin{remark} A sufficient condition that $\to_\phi$ restricts to $\to_\Pi$ is when $q|_{\lc B(u,v)\rc} \in V$ whenever $q|_{\lc u \rc\lc v \rc} \in W$. \mlabel{rem:VT} \end{remark}

\begin{lemma}
For all $u, v\in \frakM(Z)$, $E \in \bfk\mapm{Z}$, and $q\in \frakM^{\star}(Z)$, we have  $q|_{\lc u\rc \lc v\rc} \dps q|_{\lc E \rc}.$
\mlabel{lem:cyc}
\end{lemma}
\begin{proof}
First note that left and right multiplication on $\frakM^{\star}(Z)$ by an element of $\frakM(Z)$ is injective, as is the map sending $q\in \frakM^{\star}(Z)$ to $\lc q\rc$.

By Lemma~\mref{lem:add}, we only need to prove the lemma for monomials $E$ for which we apply induction on the depth of $q$. If the depth of $q$ is 0, then $q=q_1\star q_2$ for $q_1,q_2\in M(Z)$. Suppose there is $c\in \bfk$ such that $q|_{\lc u\rc \lc v\rc}=cq|_{\lc E\rc}$. Then we obtain $\lc u\rc \lc v\rc =c\lc E\rc$. This is a contradiction since the two sides have different breadths. Thus we have $q|_{\lc u\rc \lc v\rc} \dps q|_{\lc E\rc}.$
Assume that the lemma has been proved for $q$ with depth less or equal to $n\geq 1$ and consider $q$ with depth $n+1$. Then we have $q=q_1\lc q'\rc q_2$ with $q_1, q_2\in \frakM(Z)$ and $q'\in \frakM^{\star}(Z)$ with depth $n$. Thus from $q|_{\lc u\rc \lc v\rc}=cq|_{\lc E\rc}$ with $c\in \bfk$, we obtain
$$ q_1 \lc q'|_{\lc u\rc \lc v\rc}\rc q_2 =c q_1 \lc q'|_{\lc E\rc}\rc q_2.$$
From this we obtain
$q'|_{\lc u\rc \lc v\rc} =c q'|_{\lc E\rc}.$
This contradicts the induction hypothesis. Thus we have $q|_{\lc u\rc \lc v\rc} \dps q|_{\lc E\rc}$, completing the induction.
\end{proof}

\begin{coro} The RB$\phi$RS $\Pi_\phi$ is a simple rewriting system in the sense of Definition~\mref{def:ARSbasics}.
\mlabel{cor:1nonrec}
\end{coro}

 We recall here some basic notions of rewriting systems that we specialize to $\Pi_\phi(Z)$ for any set $Z$ (including the case $Z = X =\{x,y\}$).

\begin{defn}
We say a bracketed polynomial $f \in \bfk\mapm{Z}$ is {\bf $\phi$-irreducible} (that is, irreducible with respect to $\Pi_\phi(Z)$) or is {\bf in (Rota-Baxter) normal form} \paren{RBNF} if no monomial of $f$ has $\lc u\rc \lc v \rc$ as a subword for any two monomials $u,v \in \mapm{Z}$; otherwise, we say $f$ is {\bf $\phi$-reducible}. Equivalently, $f$ is $\phi$-reducible if there exists $g \in \bfk\mapm{Z}$, $g \ne f$, such that $f \rightarrow_\phi g$. A bracketed polynomial $g \in \bfk\mapm{Z}$ is said to be {\bf a normal $\phi$-form} for $f$ if $g$ is in RBNF and $f \astarrow_\phi g$.\mlabel{def:RBNF}
\end{defn}

In particular, if $f$ is a {\it monomial}, then it is in RBNF if and only if there do not exist $q\in \frakM^\star(Z)$ and $u, v \in \mapm{Z}$ such that $f=q|_{\lc u\rc \lc v\rc}$. Let $\rbw(Z)$ denote the set of {\it monomials} of $\mapm{Z}$ in RBNF.

For a set $X$, the monomials in $X$ in RBNF are called Rota-Baxter words (RBW) in $X$ in~\mcite{EG1}. They are so named since they form a canonical basis of the free Rota-Baxter algebra on $X$. We will see later that they also form a canonical basis for some other Rota-Baxter type algebras.
As can be seen by removing all appearances of the superfluous monoid unit 1 from a representation given by Eq.\,(\mref{eq:decom}), every monomial $\frakx \in \frakM(X)$ in RBNF that is not the monoid unit 1 has a unique decomposition of the form
\begin{equation}
\frakx = \frakx_1\cdots \frakx_{k}, \mlabel{eq:stde}
\end{equation}
where the $\frakx_i$ for $1\leq i\leq k$ alternate to belong to either $S(X)$ or $\lc \rbw(X)\rc$.

\begin{defn} An expression $B \in \bfk\mapm{X}$ is {\bf totally linear in $X$} if every variables $x \in X$ appears exactly once in every monomial of $B$, when counted with multiplicity in repeated multiplications. \mlabel{def:tlinear}
\end{defn}

\begin{exam} Let $X = \{x, y\}$. The expression $x\lc y\rc + \lc x \rc \lc y \rc + x y$ is totally linear in $X$, but the monomials $\lc x \rc$, $x^2 \lc y \rc$ and $x \lc y \rc^2$ are not. \mlabel{ex:tlinear}
\end{exam}

The following definition is extracted from key properties of Rota-Baxter operators.

\begin{defn}
{\rm An expression $\phi \in \bfk\mapm{x,y}$ (more correctly, the OPI $\phi = 0$) is a {\bf Rota-Baxter type} OPI if $\phi$ has the form $\lc x\rc\lc y\rc-\lc B(x,y)\rc$ for some $B(x,y)\in \bfk\mapm{x,y}$ and if the following four conditions are satisfied:
\begin{enumerate}
\setlength{\itemsep}{3pt}
\item
$B(x,y)$ is {\bf totally linear} in $x, y$;
\mlabel{it:rb0}
\item
$B(x,y)$ is in RBNF; \mlabel{it:rb1}
\item For every set $Z$,
the rewriting system $\Pi_\phi(Z)$ in Eq.~(\mref{eq:Sphi}) is terminating;
\mlabel{it:rb2}
\item
For every set $Z$ and for all $u, v, w\in\frakM(Z)$, the expression $B(B(u,v),w)-B(u,B(v,w))$ is $\phi$-reducible to zero. \mlabel{it:rb3}
\end{enumerate}
If $\phi:=\lc x\rc\lc y\rc-\lc B(x,y)\rc$ is of Rota-Baxter type, then we say the expression $B(x,y)$ and the defining operator $P=\lc\ \rc$ of a $\phi$-algebra $R$ are {\bf of Rota-Baxter type}, too. By a {\bf Rota-Baxter type algebra}, we mean some $\phi$-algebra $R$ where $\phi$ is some expression in $\bfk\mapm{x,y}$ of Rota-Baxter type. \mlabel{de:rbtype} }
\end{defn}

\begin{exam} \mcite{GSZ}
Let $B(x,y):=x\lc y\rc$. Then $\phi = 0$ is the OPI defining the average operator and it is of Rota-Baxter type. As will be shown in Theorem~\mref{thm:exam}, the identities defining a Rota-Baxter operator and that defining a Nijenhuis operator are OPIs of Rota-Baxter type. \mlabel{ex:rbt0} \end{exam}

\begin{exam}
The expression $B(x,y):=y\lc x\rc$ is not of Rota-Baxter type. This is because in $\bfk\mapm{u,v,w}$, the operated polynomial
\begin{eqnarray*}B(B(u,v),w) -B(u,B(v,w))  &=& w \lc B(u,v)\rc -
B(v,w)\lc
u\rc\\&=&w\lc v \lc u\rc\rc  - w\lc v\rc \lc u\rc\\
&\,\,\,\,\,\,\rightarrow_\phi  & w\lc v \lc u\rc\rc  - w\lc u \lc v\rc \rc
\end{eqnarray*}
is in RBNF but is non-zero, and there is no other sequence of reduction for $w\lc v \lc u\rc\rc - w\lc u \lc v\rc\rc$.
\mlabel{ex:rbt1}\end{exam}

\begin{remark}Condition (\mref{it:rb0}) in
Definition~\mref{de:rbtype} is imposed since we are considering linear operators. Conditions (\mref{it:rb1}) and (\mref{it:rb2}) are needed to avoid obvious infinite rewriting under $\Pi_\phi(Z)$ though their relationship is still quite mysterious. Condition (\mref{it:rb3}) is to ensure compatibility with the associative law for products of the form $\lc a \rc \lc b \rc \lc c \rc$ where $a, b, c \in R$ for  a $\phi$-algebra $R$.
\end{remark}

The next proposition shows that a $\phi$-algebra is also a $\psi$-algebra if $\psi \astarrow_\phi 0$. In the next two propositions, for clarity, we spell out the algebra and $\phi$-algebra structure explicitly when needed. For example, if $B(x,y) \in \bfk\mapm{x,y}$, then $B_{(R,\ast,P)}$ refers to the set map from $(R,\ast,P)^2 \rightarrow (R,\ast,P)$ (see Footnote \mref{fn:algebra}).

\begin{prop} Let $R = (R, \ast, P)$ be a $\phi$-algebra, where $\phi:=\lc x\rc\lc y\rc-\lc B(x,y)\rc$. Then for any set $Z$, any finite number of distinct symbols $z_1, \cdots, z_k \in \genset$, and any operated polynomial $\psi=\psi(z_1, \cdots, z_k)$ in $\bfk\mapm{Z}$ such that $\psi \astarrow_\phi 0$, the $\phi$-algebra $R$ is also a $\psi$-algebra. \mlabel{prop:psi}
\end{prop}
\begin{proof} It suffices to show that if $\psi \rightarrow_\phi  \psi'$, then $\psi_R(r_1, \cdots, r_k)= \psi'_R(r_1, \cdots, r_k)$ for all $r_1, \cdots, r_k \in R$. Let $\psi \rightarrow_\phi  \psi'$. Then there exist $q \in \frakM^{\star}(Z)$, $c \in \bfk$ ($c \ne 0$), and $u, v \in \frakM(Z)$ such that
\begin{enumerate}
\item $q|_{\lc u\rc \lc v\rc}$ is a monomial of $\psi$, which has $c$ as its coefficient.

\item $\psi'=\psi-cq|_{(\lc u\rc \lc v\rc- \lc B(u,v)\rc)}$.
\end{enumerate}
By increasing $k$ if necessary, we may assume $u=u(z_1, \cdots, z_k)$ and $ v=v(z_1, \cdots, z_k)$ are in $\mapm{z_1, \cdots, z_k}$ and $q \in \frakM^\star(z_1, \cdots, z_k)$. Then for any $r_1, \cdots, r_k \in R$, the elements $a = u_R(r_1, \cdots, r_k)$ and $b = v_R(r_1, \cdots, r_k)$ are in $R$. Since $R$ is a $\phi$-algebra, $\lc a\rc \lc b\rc- \lc B_R(a,b)\rc = \phi_R(a,b) =0$, and hence $$\psi'_R(r_1, \cdots, r_k) = \psi_R(r_1, \cdots, r_k) - c q_R(r_1, \dots, r_k)|_{(\lc a\rc \lc b\rc- \lc B(a,b)\rc)} = \psi_R(r_1, \cdots, r_k).$$
This completes the proof.
\end{proof}

For a Rota-Baxter algebra $R$, with multiplication $*$ and Rota-Baxter operator $P$, it is common to endow $R$ with another multiplication in terms of the defining operator identity. This double algebra structure plays important roles in the splitting of associativity in algebras such as the dendriform algebra and more generally successors of operads\mcite{BBGN,Lo2,Lo}, and in integrable systems in the Lie algebra context~\mcite{Bai2,BGN,STS}. We describe this double structure below for the more general Rota-Baxter type algebras (for the case of Rota-Baxter operator, see~\cite[\S~1.1.17]{Gub}).

\begin{prop}
Let $\phi \in \bfk\mapm{x,y}$ be of Rota-Baxter type and suppose $\phi = \lc x \rc \lc y \rc - \lc B(x,y)\rc$. Let $(R,\ast,Q)$ be a $\phi$-algebra. Define a second multiplication $\ast_\phi$ by
\begin{equation}
r_1\ast_\phi r_2:= B_{(R,\ast, Q)}(r_1,r_2), \quad \text{for all } r_1, r_2\in R. \notag
\end{equation}
Then
\begin{enumerate}
\item
The pair $(R,\ast_\phi)$ is a nonunitary $\bfk$-algebra. \mlabel{it:double1}
\item If $B(x,y)$ does not involve $\lc 1 \rc \in \bfk\mapm{x,y}$, then
the triple $(R,\ast_\phi,Q)$ is a nonunitary $\phi$-algebra. \mlabel{it:double2}
\end{enumerate}
\mlabel{pp:double}
\end{prop}

\begin{proof} For clarity, we now use $P$ to denote the operator $\lc \, \rc$ for $\bfk\mapm{x,y}$, so for example $P(1) = \lc 1 \rc$, $P^i$ is the $i$-fold iteration of $P$, and $P^0$ is the identity operator. We observe that since $B(x,y)\in \bfk\mapm{x,y}$ is totally linear in $x, y$ and is in RBNF, we can write
\begin{equation}
B(x,y)=\sum_{j\in J} a_j B_j(x,y),
\mlabel{eq:decomp}
\end{equation}
where $J$ is a finite set, for $j \in J$, $B_j(x,y)$ are distinct, totally-linear monomials in RBNF and do not involving $P(1)$ in $\frakM(x,y)$, and $a_j \in \bfk$ and $a_j \ne 0$.  Hence $B_j(x,y)$ has one of two forms

\begin{equation}{\rm either\ } B_j(x,y) = P^{k_j}(P^{m_j}(x)P^{n_j}(y)) \quad{\rm\ or\ }\quad B_j(x,y) = P^{k_j}(P^{n_j}(y)P^{m_j}(x)) \mlabel{eq:mRBNF}\end{equation}
with integers $k_j, m_j, n_j\geq 0$ and  $m_j n_j=0$.\footnote{The careful reader will note that in the proof of (\mref{it:double2}), we never make use of the property that $m_j n_j = 0$. Thus the operated polynomial identity $\phi_{(R, \ast_\phi, Q)}(r_1, r_2) = 0$ holds under a much weaker assumption, requiring only that $B(x,y)$ be totally linear but not necessarily in RBNF. However, considering Examples \mref{ex:rbt0} and \mref{ex:rbt1}, we want to emphasize the importance that $B(x,y)$ be of Rota-Baxter type to begin with for (\mref{it:double1}) to hold, that is, for $\ast_\phi$ to be associative.\mlabel{fn:mn}}

(\mref{it:double1}) Applying Lemma \mref{pp:freetm} with $X$ set to $Z := \{u,v,w\}$, let $(R', \ast', P')$ be the free operated algebra $\bfk\mapm{Z} = \bfk\mapm{u,v,w}$ on the set $Z$. The operated polynomial
\begin{equation}\psi:= B_{(R', \ast', P')}(B_{(R', \ast', P')}(u,v),w) - B_{(R', \ast', P')}(u, B_{(R', \ast', P')}(v,w))
\mlabel{eq:bj}\end{equation}
of $R'$ is $\phi$-reducible to zero. Since $(R,*,Q)$ is a $\phi$-algebra, the associativity of $\ast_\phi$ in $(R, \ast_\phi)$ holds if and only if $(R, \ast, Q)$ is a $\psi$-algebra, which is the case by Proposition \mref{prop:psi}. Similarly, consider the operated polynomials
\begin{eqnarray}
\psi_1 &:=& B_{(R', \ast', P')}(u,v+w)-B_{(R', \ast', P')}(u,v) - B_{(R', \ast', P')}(u,w), \mlabel{eq:distr1}\\
\psi_2 &:=& B_{(R', \ast', P')}(v+w,u)-B_{(R', \ast', P')}(v,u) - B_{(R', \ast', P')}(w,u). \mlabel{eq:distr2}
\end{eqnarray}
By Eq.\,(\mref{eq:bj}), the linearity of $P'$, and the distributive laws of $\ast'$, we get $\psi_1 = \psi_2 = 0$ and hence $(R, \ast_\phi)$ satisfies
the left and right distributive laws. Thus $(R,\ast_\phi)$ is a nonunitary $\bfk$-algebra.

\smallskip
\noindent (\mref{it:double2}) To prove that $(R, \ast_\phi, Q)$ is a $\phi$-algebra, we must show that $\phi_{(R, \ast_\phi, Q)}(r_1, r_2) = 0$ for all $r_1, r_2 \in R$. We partition the index set $J$ from Eq.\,(\mref{eq:decomp}) accordingly into two disjoint sets $J_1$, $J_2$, where $B_j \in J_1$ has the first form from Eq.\,(\mref{eq:mRBNF}), and $B_j \in J_2$ has the second form.

For any $r_1, r_2 \in R$, we have
\begin{eqnarray*}
Q(r_1)\ast_\phi Q(r_2) &=& B_{(R,\ast,Q)}(Q(r_1),Q(r_2))\\
&=& \sum_{j\in J_1} a_j Q^{k_j}\Bigl(Q^{m_j+1}(r_1)\ast Q^{n_j+1}(r_2)\Bigr)
+ \sum_{j \in J_2} a_j Q^{k_j}\Bigl(Q^{n_j+1}(r_2)\ast Q^{m_j+1}(r_1)\Bigr)\\
&=&\sum_{j\in J_1} a_j Q^{k_j}\Bigl(Q\bigl( Q^{m_j}(r_1)\bigr)\ast Q\bigl( Q^{n_j}(r_2)\bigr)\Bigr) +
\sum_{j\in J_2} a_j Q^{k_j}\Bigl(Q\bigl(Q^{n_j}(r_2)\bigr)\ast Q\bigl( Q^{m_j}(r_1)\bigr)\Bigr)\\
&=&\sum_{j\in J_1} a_j Q^{k_j}\Bigl(Q\bigl( B_{(R,\ast,Q)}(Q^{m_j}(r_1), Q^{n_j}(r_2))\bigr)\Bigr) +
\sum_{j\in J_2} a_j Q^{k_j}\Bigl(Q\bigl( B_{(R,\ast,Q)}(Q^{n_j}(r_2), Q^{m_j}(r_1))\bigr)\Bigr)\\
&=&\sum_{j\in J_1} a_j Q^{k_j+1}\Bigl(Q^{m_j}(r_1)\ast_\phi Q^{n_j}(r_2)\Bigr) +
\sum_{j\in J_2} a_j Q^{k_j+1}\Bigl(Q^{n_j}(r_2)\ast_\phi Q^{m_j}(r_1)\Bigr) \\
&=&Q\Bigl(\sum_{j\in J_1} a_j Q^{k_j}\bigl(Q^{m_j}(r_1) \ast_\phi Q^{n_j}(r_2)\bigr)
+\sum_{j\in J_2} a_j Q^{k_j}\bigl(Q^{n_j}(r_2) \ast_\phi Q^{m_j}(r_1)\bigr)
\Bigr)\\
&=&Q\Bigl( B_{(R,\ast_\phi,Q)}(r_1,r_2)\Bigr).
\end{eqnarray*}
Thus $\phi_{(R, \ast_\phi, Q)}(r_1, r_2) = 0$ for all $r_1, r_2 \in R$ and $(R, \ast_\phi, Q)$ is a non-unitary $\phi$-algebra (of Rota-Baxter type).
\end{proof}

We recall the following conjecture on Rota-Baxter type operators as a case of Rota's problem.

\begin{conjecture}
{\bf (Classification of Rota-Baxter Type Operators)\mcite{GSZ}} For any $c, \lambda \in \bfk$, the operated polynomial $\phi:= \lc x\rc \lc y\rc - \lc B(x,y)\rc$, where $B(x,y)$ is taken from the list below, is of Rota-Baxter type. Moreover, any OPI $\phi$ of Rota-Baxter type is necessarily  defined as above by a $B(x,y)$ from among this list \paren{new types are underlined}.
\smallskip
\begin{enumerate}\setlength\itemsep{3pt}
\item $x\lc y\rc \quad$ \paren{average operator}, 
\item $\lc x\rc y \quad$ \paren{inverse average operator}, 
\item $\underline{x\lc y\rc +y\lc x\rc }$, 
\item $\underline{\lc x\rc y+\lc y\rc x}$, 
\item $x\lc y\rc +\lc x\rc y -\lc xy\rc \quad$ \paren{
Nijenhuis operator}, 
\item $x\lc y\rc +\lc x\rc y + \lambda xy \quad$
\paren{Rota-Baxter operator of weight $\lambda$}, 
\item $\underline{x\lc y\rc -x\lc 1\rc y + \lambda xy}$, 
\item $\underline{\lc x\rc y - x\lc 1\rc y + \lambda xy}$, 
\item $\underline{x\lc y\rc  + \lc x\rc y -x\lc 1\rc y+\lambda xy}
\quad$ \paren{generalized
Leroux TD operator with weight $\lambda$}, 
\item $\underline{x\lc y\rc +\lc x\rc y - xy\lc 1\rc -x\lc 1\rc y+
\lambda xy}$, 
\item $\underline{x\lc y\rc +\lc x\rc y -x\lc 1\rc y-\lc xy\rc +
\lambda xy}$, 
\item $\underline{x\lc y\rc +\lc x\rc y-x\lc 1\rc y-\lc 1\rc xy+
\lambda xy}$, 
\item $\underline{c x\lc 1\rc y + \lambda xy}\quad$ \paren{generalized
endomorphisms}, 
\item $\underline{c y\lc 1\rc x + \lambda yx} \quad$ \paren{generalized
antimorphisms}. 
\end{enumerate}
\mlabel{con:rbt}
\end{conjecture}

\section{Rota-Baxter type operators and convergent rewriting
systems} \mlabel{sec:rbcr} In this section, we shall establish the close relationship between a Rota-Baxter type OPI $\phi$ and the convergence of its rewriting systems $\Pi_\phi(Z)$ on $\bfk\mapm{Z}$ for sets $Z$ in the presence of a monomial order that is compatible with $\Pi_\phi$. It would be interesting to explore how the latter condition can be weakened or removed.

\begin{defn} Let $Z$ be a set. For distinct symbols $\star_1,\star_2$ not in $Z$, let $Z^{\star_1,\star_2}: = Z \cup \{\star_1, \star_2\}$. We define a {\bf $(\star_1, \star_2)$-bracketed word on $Z$} to be a bracketed word in $\mapm{Z^{\star_1,\star_2}} = \frakM(Z^{\star_1,\star_2})$ with exactly one occurrence of $\star_1$ and exactly one occurrence of $\star_2$, each counted with multiplicity. The set of $(\star_1, \star_2)$-bracketed words on $Z$ is denoted by either $\mapmZss$ or $\frakM^{\star_1,\star_2}(Z)$. For $q\in \mapmZss$ and $u_1,u_2\in \bfk\mapm{Z}$, we define
\begin{equation}
q|_{u_1,\ u_2}= q|_{\star_1 \mapsto u_1,\star_2 \mapsto u_2},
\mlabel{eq:2star1}
\end{equation}
to be the bracketed word in $\frakM(Z)$ obtained by replacing the symbol $\star_1$ in $q$ by $u_1$ and replacing the symbol $\star_2$ in $q$ by $u_2$, simultaneously.  A {\bf $(u_1, u_2)$-bracketed word on $Z$} is a word of the form Eq.\,(\mref{eq:2star1}) for some $q \in \mapmZss$.
\end{defn}

A $(u_1, u_2)$-bracketed word on $Z$ can also be recursively defined by
\begin{equation}
q|_{u_1,u_2}:=(q^{\star_1}|_{u_1})|_{u_2}, \mlabel{eq:2star2}
\end{equation}
where $q^{\star_1}$ is $q$ when $q$ is regarded as a
$\star_1$-bracketed word on the set $Z\cup\{{\star_2}\}$. Then $q^{\star_1}|_{u_1}$ is in $\frakM^{\star_2}({Z})$ and hence Eq. (\mref{eq:2star2}) is well-defined. Similarly, treating $q$ first as a $\star_2$-bracketed word $q^{\star_2}$ on the set $Z\cup\{{\star_1}\}$, we have
\begin{equation}
q|_{u_1,u_2}:=(q^{\star_2}|_{u_2})|_{u_1}. \mlabel{eq:2star3}
\end{equation}

We describe the relative location of two bracketed subwords, or more precisely, their placements (Definition~\mref{de:subword}), in a bracketed word. See Example \mref{ex:subwords} for motivation and \cite{ZheG} for details.

\begin{defn}
Let $w, u_1, u_2\in \mapm{Z}$ and $q_1, q_2\in \frakM^\star(Z)$ be such that \begin{equation}q_1|_{u_1}=w=q_2|_{u_2}.\mlabel{eq:plas}\end{equation}
The two \plas $(u_1,q_1)$ and $(u_2,q_2)$ are said to be
\begin{enumerate}
\item
{\bf separated} if there exist an element $q$ in $\frakM^{\star_1,\star_2}(Z)$ and $a,b \in \mapm{Z}$ such that $q_1|_{\star_1}=q|_{\star_1,\,b}$, $q_2|_{\star_2}=q|_{a,\,\star_2}$ and $w=q|_{a,\,b}$;
\mlabel{item:bsep}
\item
{\bf nested} if there exists an element $q$ in $\frakM^{\star}(Z)$ such that either $q_2=q_1|_q$ or $q_1=q_2|_q$;
\mlabel{item:bnes}
\item
 {\bf intersecting} if there exist an element $q$ in $\frakM^{\star}(Z)$ and elements $a, b, c$ in $\frakM(Z)\backslash\{1\}$ such that $w=q|_{abc}$ and either
\begin{enumerate}
\item
$ q_1=q|_{\star c}, q_2=q|_{a\star}$; or\mlabel{item:left2}
\item
$q_1=q|_{a\star}, q_2=q|_{\star c}$.
\mlabel{item:right2}
\end{enumerate}
\mlabel{item:bint}
\end{enumerate}
\mlabel{defn:bwrel}
\end{defn}

\begin{remark} The defining conditions in Definition \mref{defn:bwrel} apparently are properties of $q_1$ and $q_2$ that have nothing to do with $u_1$ and $u_2$, but actually they constrain the placements of $u_1$ and $u_2$ to be separated, nested or intersecting for any two subwords $u_1, u_2$ satisfying Eq.\,(\mref{eq:plas}). The conditions are thus stronger than the requirement for a particular pair of $u_1, u_2$. To illustrate this, we view $q_1, q_2$ as strings and write $q_1 = \ell_1 \star r_1$ and $q_2 = \ell_2 \star r_2$.

\begin{enumerate}
\item Condition (\mref{item:bsep}) implies $q|_{\star_1,\,b} = \ell_1 \star_1 r_1$ and $q|_{a,\,\star_2} = \ell_2 \star_2 r_2$. Hence $w = \ell_1 a r_1 = \ell_2 b r_2$. By Eq.\,(\mref{eq:plas}), we must have $u_1 = a$ and $u_2 = b$. Since $q$ is of the form either $\ell \star_1 m \star_2 r$ or $\ell \star_2 m \star_1 r$, the placements $u_1$ and $u_2$ in $w$ are separated.
Note, however, that even if $u_1$ is a subword of  $u_2$ (including equality), their placements may still be separated.

\item Suppose $q_2 = q_1|_q$ in Condition (\mref{item:bnes}) is satisfied and we write $q = \ell \star r$. Then $\ell_2 \star r_2=\ell_1 \ell \star r r_1$ so that $\ell_2 = \ell_1 \ell$ and $r_2 = r r_1$. This means $q_1$ and $q_2$ share an initial string $\ell_1$ and an ending string $r_1$; replacing the inner $\star$-word $q$ by $\star$ collapses $q_2$ to $q_1$. From Eq.\,(\mref{eq:plas}), $\ell_1 u_1 r_1 = w = \ell_2 u_2 r_2 = \ell_1 \ell u_2 r r_1$ and hence
\begin{equation}
u_1=q|_{u_2}.
\mlabel{eq:rel}
\end{equation}
Thus $u_2$ is a bracketed subword of $u_1$. Note that a special case of being nested is when $(u_1, q_1) = (u_2, q_2)$ with $q=\star$.

\item
 Suppose (i) under Condition (\mref{item:bint}) holds. Then we have $w=q_1|_{u_1}=(q|_{\star c})|_{u_1} =q|_{u_1c}$ and $w=q_2|_{u_2}=(q|_{a\star})|_{u_2} =q|_{au_2}$. Then we have $u_1c=au_2=abc$. Thus $u_1=ab$, $u_2=bc$ and they have a nontrivial intersection $b$. Note that the existence of $q$ satisfying all the conditions is crucial. Example: Let $b \ne 1$, $y = abc$, $w=xabybcz$, $u_1=ab$, $u_2=bc$, $q_1 = x\star y bc z$, and $q_2 = xaby\star z$. Then $q_1|_{u_1}=w=q_2|_{u_2}$.  The placements $(u_1, q_1)$ and $(u_2, q_2)$ do not overlap, even though $abc=y$ is a subword of $w$---it occurs at the wrong place.
\mlabel{it:rel3}
\end{enumerate}
\mlabel{rk:rel}
\end{remark}

\begin{exam}
Let $w=x\lc x\rc x, u_1=x\lc x\rc$ and $u_2=x$. Then $u_1$ is a bracketed subword of $w$ with \pla $(u_1,q_1)$ where $q_1=\star x$. Also $u_2$ is a bracketed subword of $w$ in three locations with \plas $(u_2,q_{21})$, $(u_2,q_{22})$, and $(u_2, q_{23})$, where $q_{21}=\star \lc x\rc x$, $q_{22}=x\lc x\rc \star$, and $q_{23} = x \lc \star \rc x$. Then $(u_1,q_1)$ and $(u_2,q_{21})$ are nested, as are $(u_1, q_1)$ and $(u_2, q_{23})$, while $(u_1,q_1)$ and $(u_2,q_{22})$ are separated. Further, denoting $u_3:=\lc x\rc x$, then $(u_3,q_3)$ with $q_3=x \star$ is a \pla of $u_3$ in $w$ and the \plas $(u_1,q_1)$ and $(u_3,q_3)$ are intersecting, while $(u_2, q_{22})$ and $(u_3, q_3)$ are nested.
\mlabel{ex:placement}
\end{exam}

\vspace{0.2in}
\begin{theorem} ~\cite[Theorem 4.11]{ZheG}
Let $w$ be a bracketed word in $\frakM(X)$. For any two \plas $(u_1, q_1)$ and $(u_2,q_2)$ in $w$, exactly one of the following is true:
\begin{enumerate}
\item
$(u_1, q_1)$ and $(u_2,q_2)$ are separated;
\item
$(u_1, q_1)$ and $(u_2,q_2)$ are nested;
\item
$(u_1, q_1)$ and $(u_2,q_2)$ are intersecting.
\end{enumerate}
\mlabel{thm:thrrel}
\end{theorem}

Let a set $Z$ be given. We next give the definition of monomial order on $\frakM(Z)$.

\begin{defn} A {\bf monomial order on $\frakM(Z)$} is a
well order $\leq := \leq_Z$ on $\frakM(Z)$ such that
\begin{equation} u < v
\Longrightarrow q|_u < q|_v, \quad \text{\ for\ all\ } u, v \in \frakM(Z) \text{\ and\ all\ } q \in \frakM^{\star}(Z) .\mlabel{eq:morder}
\end{equation}
Here, as usual, we denote $u < v$ if $u\leq v$ but $u\neq v$.
\end{defn}

Since $\leq$ is a well order, it follows from Eq.\,(\mref{eq:morder}) that $1 \leq u$ and $u < \lc u \rc$ for all $u \in \frakM(Z)$.

\begin{defn} Let $\leq$ be a monomial order on $\mapm{Z}$, $f \in \bfk\mapm{Z}$ and $S \subset \bfk\mapm{Z}$. Let $\phi(x,y):=\lc x\rc \lc y\rc -\lc B(x,y)\rc \in \bfk\mapm{x,y}$.
\begin{enumerate}
\item
The {\bf leading bracketed word (monomial) of $f$} is the (unique) largest monomial $\bar{f}$ appearing in $f$. The {\bf leading coefficient of $f$} is the coefficient of $\bar{f}$ in $f$, which we denote by $c(f)$. If $c(f) = 1$, we say $f$ is {\bf monic with respect to the monomial order $\leq$\,}.
We define the {\bf remainder} $R(f)$ of $f$ by \begin{equation}
R(f):= f - c(f)\bar{f}
\mlabel{eq:lead}
\end{equation}
so that $f=c(f) \bar{f}\dps R(f)$. 
\item Suppose $f$ is $\phi$-reducible. We define the {\bf leading $\phi$-reducible monomial of $f$} to be the monomial $L(f)$ maximal with respect to $\leq$ among monomials $m$ appearing in $f$ that are $\phi$-reducible, that is, $$L(f):=\max\{m\,|\, m \text{ is a monomial of $f$ and $m \notin \frakR(Z)$}\}.$$ \mlabel{item:leadred}

\item Suppose $s$ is monic for all $s \in S$. We define the {\bf rewriting system associated with $S$} to be the set of rewriting rules given by
\begin{equation}
\Pi_S(Z) := \{ \bar{s} \to_S -R(s) \mid s \in S\}
\mlabel{eq:rwS}
\end{equation}
and we denote the reflexive transitive closure of $\to_S$ by $\astarrow_S$. The set $\Irr(S)$ of {\bf irreducibles with respect to $S$ and $\leq$} is defined by
$$ \Irr(S):=\Irr^{Z,\leq}(S) = \frakM(Z)\backslash \left\{\,q|_{\overline{s}} \,|\,
q \in \frakM^\star(Z), s\in S\right\}.$$
An element $f \in \bfk\mapm{Z}$ is {\bf irreducible with respect to $S$} if $f \in \kdot \Irr(S)$.

\item We say $\phi$, or the rewriting system $\Pi_\phi(Z)$ defined by {\rm Eq.~(\mref{eq:Sphi})}, is {\bf compatible with $\leq$} if $\overline{\lc B(u,v)\rc} < \lc u\rc\lc v\rc$ (equivalently, $\overline{\phi(u,v)}=\lc u\rc \lc v\rc$) for all $u, v\in \frakM(Z)$.\mlabel{de:mcompat}

\end{enumerate}\mlabel{def:irrS}
\end{defn}

\begin{remark} When $\phi$ is compatible with $\leq$, and $S := S_\phi(Z)$ (as defined by Eq.\,(\mref{eq:genphi})), the relation $\rightarrow_\phi $ (resp. its reflexive transitive closure $\astarrow_\phi$) is the relation $\to_{S}$ (resp. $\astarrow_S$). If $s \in S$ is given by $s := \lc u \rc \lc v\rc - \lc B(u,v)\rc$, where $u, v \in \mapm{Z}$, then the remainder $R(s)$ is the $(q,u,v)$-complement $R_{q,u,v}(s)$ of $s$. However, in general, if $f$ has a monomial of the form $q|_{\lc u \rc \lc v \rc}$, the remainder $R(f)$ need not be the same as the $(q,u,v)$-complement $R_{q,u,v}(f)$ of $f$, unless $\bar{f} = q|_{\lc u \rc \lc v \rc}$. The set $\Irr(S)$ is precisely the set of monomials that are in RBNF and $f \in \bfk\mapm{Z}$ is $\phi$-reducible if and only if $f$ is $S$-reducible. \mlabel{rmk:compat}
\end{remark}

\begin{lemma} Let $Z$ be a set. Suppose $\phi:= \lc x \rc \lc y \rc - \lc B(x,y) \rc$ is compatible with a monomial order $\leq$ on $\frakM(Z)$. Suppose $g, g' \in \bfk\mapm{Z}$ are both $\phi$-reducible and for $q \in \frakM^\star(Z)$ and $u, v \in \mapm{Z}$, $\quvarrow{g}{g'}{\phi}$. Then $L(g') \leq L(g)$, where equality holds if and only if $L(g)\ne q|_{\lc u\rc\lc v \rc}$. \mlabel{lem:Lred}
\end{lemma}

\begin{proof} We may write
$$g = (c_1 m_1 + \cdots + c_n m_n) \dps h$$
for some integer $n \geq 1$; $c_1, \dots, c_n \in \bfk$ are all non-zero;  $m_1 > \cdots > m_n$ are monomials in $\frakM(Z) \backslash \frakR(Z)$; and $h \in \bfk\mapm{Z}$ is in RBNF. By Definition \mref{def:irrS} (\mref{item:leadred}), $L(g) = m_1$. Let $i$, $(1 \leq i \leq n)$, be such that $m_i = q|_{\lc u \rc \lc v \rc}$. Then
\begin{eqnarray*}g'& =&c_i q|_{\lc B(u,v) \rc} \dps R_{q,u,v}(g)\\
&=& c_1 m_1 + \cdots + c_{i-1} m_{i-1} + c_i q|_{\lc B(u,v) \rc } +
c_{i+1} m_{i+1} + \cdots + c_n m_n + h
\end{eqnarray*}
Now $m_1 \geq m_i = q|_{\lc u\rc \lc v \rc} > q|_{\overline{\lc B(u,v) \rc}}$ since $\phi$ is compatible with $\leq$. Thus $L(g') = m_1 = L(g)$ if $i \ne 1$, and $L(g') < m_1$ if $i = 1$.
\end{proof}

\begin{theorem}
Suppose $Z, B(x,y)$ and $\phi$ are as in {\rm Lemma \mref{lem:Lred}}. Then the rewriting system $\Pi_\phi(Z)$ is terminating. \mlabel{thm:roc1}
\end{theorem}

\begin{proof} Let
$$\calc=\left\{g\in\bfk\mapm{Z}\,|\, \text{there is an infinite $\phi$-reduction chain} \,g:=g_0\rightarrow_\phi g_1\rightarrow_\phi\cdots\right\}.$$
We only need to prove that $\calc=\emptyset$.
Suppose that $\calc\neq \emptyset$. Since $g$ is $\phi$-reducible for all $g \in \calc$, and $\leq$ is a well order on $\mapm{Z}$, the set $\call:= \{ L(g) \mid g \in \calc \}$, where $L(g)$ is the leading $\phi$-reducible monomial in $g$, is non-empty and has a least element $w_0$. We fix a $g \in \calc$ with $L(g)=w_0$ and fix an infinite $\phi$-reduction chain $g:=g_0\overset{q_0,u_0,v_0}{\longrightarrow}_\phi g_1\overset{q_1,u_1, v_1}{\longrightarrow}_\phi\cdots$. Then we have $g_i\in\calc$ and hence $\phi$-reducible for all $i\geq 1$.  Let $w_i = L(g_i)$. By Lemma \mref{lem:Lred}, we have $w_0 \geq w_1 \geq \dots$.  Since every $g_i$ is in $\calc$, and $w_0$ is the least element in $\call$, we must have $w_0 = w_i$ for all $i$. By Lemma \mref{lem:Lred}, none of the $w_i$ is involved in $\phi$-reduction of the fixed sequence starting with $g$. Let $f_i = g_i - b_i w_i$, where   $b_i$ is the coefficient of $w_i$ in $g_i$. Then we have the infinite reduction sequence $f_0 \overset{q_0,u_0,v_0}{\longrightarrow}_\phi f_1\overset{q_1,u_1, v_1}{\longrightarrow}_\phi\cdots$ and $L(f_0) < L(g)$. This is a contradiction, showing that $\calc = \emptyset$. This completes the proof.
\end{proof}

Now we apply Theorem~\mref{thm:downarrow} to our situation.

\begin{lemma}
Let $\phi(x,y):=\lc x\rc \lc y\rc -\lc B(x,y)\rc \in \bfk\mapm{x,y}$ with
$B(x,y)$ in {\rm RBNF} and totally linear in $x, y$. Let $Z$ be a set and let $\Pi_\phi := \Pi_\phi(Z)$ be the rewriting system in {\rm Eq.}~\paren{\mref{eq:Sphi}}. Let $\leq$ be a monomial order on $\frakM(Z)$ that is compatible with $\Pi_\phi$. Let $Y$ be a subset of $\frakM(Z)$. Suppose that $\Pi_\phi$ restricts to a rewriting system $\Pi_{\phi,Y}$ on $\bfk Y\subseteq \bfk \mapm{Z}$ in the sense of Definition~\mref{de:rwrest}, that is, for any $q\in \frakM^\star(Z)$ and $s\in S_\phi(Z)$, if $q|_{\bar{s}}$ is in $Y$, then $q|_{\bar{s}-s}$ is in $\bfk Y$.   Suppose that $\Pi_{\phi, Y}$ is confluent.
Let $q_i\in \frakM^\star(Z)$ and $s_i\in S_\phi(Z), 1\leq i\leq n,$ be such that $q_i|_{\bar{s_i}}$ is in $Y$ and let $c_i\in \bfk$. If $c_1q_1|_{\bar{s_1}-s_1}+c_2q_2|_{\bar{s_2}-s_2}+\cdots
+c_nq_n|_{\bar{s}_n-s_n}=0$, then $c_1q_1|_{\bar{s_1}}+c_2q_2|_{\bar{s_2}}
+\cdots
+c_nq_n|_{\bar{s_n}}\astarrow_\phi 0.$
\mlabel{lem:redzero}
\end{lemma}

\begin{proof} Since $\Pi_\phi$ is compatible with the monomial order $\leq $, we have $\bar{s}=\lc u\rc \lc v\rc$ and $\bar{s}-s=\lc B(u,v)\rc$ for all $u,v\in\frakM(Z)$. By Lemma~\mref{lem:cyc}, we have $q|_{\bar{s}}\dps q|_{\bar{s}-s}$, and so the rewriting system $\Pi_{\phi, Y}$ is \simple.
In Theorem~\mref{thm:downarrow}, take $W$ to be $Y$. Then we obtain $c_1q_1|_{\bar{s_1}}+c_2q_2|_{\bar{s_2}}
+\cdots+c_nq_n|_{\bar{s_n}}\astarrow_
{\Pi_{\phi,Y}} 0.$ Thus the lemma follows.
\end{proof}

\begin{theorem}
Let $\bfk$ be a field.
Let $\phi(x,y):=\lc x\rc \lc y\rc -\lc B(x,y)\rc \in \bfk\mapm{x,y}$ with
$B(x,y)$ in {\rm RBNF} and totally linear in $x, y$. Let $Z$ be a set and let $\Pi_\phi := \Pi_\phi(Z)$ be the rewriting system in {\rm Eq.}~\paren{\mref{eq:Sphi}}. Let $\leq$ be a monomial order on $\frakM(Z)$ that is compatible with $\Pi_\phi(Z)$. Then the following conditions are equivalent.
\begin{enumerate}
\item
For all $u, v, w \in \mapm{Z}$, the expression $B(B(u,v),w) - B(u,B(v,w))$ is $\phi$-reducible to zero.\mlabel{it:mainr}
\item $\Pi_\phi$ is convergent. \mlabel{it:mainc}
\end{enumerate}
\mlabel{thm:rbr}
\end{theorem}

From the theorem we immediately obtain:
\begin{coro}
Let $\phi(x,y):=\lc x\rc \lc y\rc-\lc B(x,y)\rc\in \bfk\mapm{x,y}$ with $B(x,y)$ in {\rm RBNF} and totally linear in $x, y$. Let $\leq$ be a monomial order on $\frakM(Z)$ that is compatible with $\Pi_\phi(Z)$. Then $P=\lc\, \rc$ is a Rota-Baxter type operator if and only if\, $\Pi_\phi(Z)$ is convergent for every set $Z$.
\mlabel{co:rbcon}
\end{coro}
\begin{proof}
($\Longrightarrow$)
If $P=\lc\,\rc$ is a Rota-Baxter type operator, then by Definition~\mref{de:rbtype}(\mref{it:rb3}),
the expression $B(B(u,v),w) - B(u,B(v,w))$ is $\phi$-reducible to zero for all $u,v,w\in \mapm{Z}$. Then by Theorem~\mref{thm:rbr}, $\Pi_\phi$ is convergent.
\smallskip

\noindent
($\Longleftarrow$) If $\Pi_\phi$ is convergent, then by Theorem~\mref{thm:rbr}, the expression $B(B(u,v),w) - B(u,B(v,w))$ is $\phi$-reducible to zero for all $u,v,w\in \mapm{Z}$.  Since the rewriting system $\Pi_\phi(Z)$ is compatible with the monomial order $\leq$, by Theorem~\mref{thm:roc1}, $\Pi_\phi$ is terminating.  Together with the conditions that $B(x,y)$ is in {\rm RBNF} and is totally linear in $x, y$, we see that $P=\lc\,\rc$ is a Rota-Baxter operator by definition.
\end{proof}

We now give the proof of Theorem~\mref{thm:rbr}.
\begin{proof} \noindent (\mref{it:mainr}) $\Longrightarrow$ (\mref{it:mainc}).
Recall from Eqs.\,(\mref{eq:phibar}) and (\mref{eq:genphi}) that for $u,v \in \frakM(Z)$, $$\phi(u,v):=\lc u\rc \lc v\rc -\lc B(u,v)\rc, \quad {\rm and} \quad
S_\phi(Z):=\left\{\,\phi(u,v)\, \mid\  u,v \in \frakM(Z)\,\right\}.$$ Since the rewriting system $\Pi_\phi(Z)$ is compatible with the monomial order $\leq$, we have $\bar{s}=\lc u\rc \lc v\rc$ for any $s\in S_\phi(Z)$. Then, for $f \in \frakM(Z)$ and $g \in \bfk\mapm{Z}$,  $f\rightarrow_\phi  g$ means that there are $q\in \frakM^\star(Z)$ and $s\in S_\phi(Z)$ such that $f=q|_{\bar{s}}$ and $g=q|_{\bar{s}-s}$.

By Theorem~\mref{thm:roc1}, the rewriting system $\Pi_\phi$ is terminating. By Lemma~\mref{lem:newman}, to prove that $\Pi_\phi$ is confluent and hence convergent, we just need to prove that $\Pi_\phi$ is locally confluent. By Lemma~\mref{lem:ltcon}, we only need to prove that $\Pi_\phi$ is locally term-confluent, i.e. for any local term fork ($ct\to_\phi cv_1, ct\to_\phi cv_2$) where $(t,v_1), (t,v_2)\in \Pi_\phi$ and $c\in \bfk$, $c\neq 0$, we have $c(v_1-v_2)\astarrow_\phi 0$. Since $\bfk$ is a field, a local term fork ($ct\to_\phi cv_1, ct\to_\phi cv_2$) gives a local term fork ($t\to_\phi v_1, t\to_\phi v_2$). Since $v_1-v_2\astarrow_\phi 0$ implies $c(v_1-v_2)\astarrow_\phi 0$, we see that we only need to prove that,
\begin{quote}
for any local term fork ($t\to_\phi v_1, t\to_\phi v_2$) where $(t,v_1), (t,v_2)\in \Pi_\phi$,
we have $v_1-v_2\astarrow_\phi 0$.
\end{quote}
We will prove this statement by contradiction. Suppose there are local term forks ($t\to_\phi v_1, t\to_\phi v_2$) such that $v_1-v_2\not\astarrow_\phi 0$.  Then the set
$$\frakN:=\left \{\left . f\in \frakM(Z) \begin{array}{l} {} \\{} \end{array}\right |\, \text{ there is a fork }
f \rightarrow_\phi  g_1, f \rightarrow_\phi  g_2 \text{ with } g_1, g_2\in \bfk\mapm{Z} \text{ and } g_1-g_2\not\astarrow_\phi 0\right\}$$
is not empty.
Since $\leq$ is a well order on $\frakM(Z)$, we can take $f$ to be the least element in $\frakN$ with respect to $\leq$.
Thus there are $q_1, q_2\in \frakM^\star(Z)$ and $s_1, s_2\in S_\phi(Z)$ such that $$q_1|_{\bar{s_1}}=f=q_2|_{\bar{s_2}}, \ g_1=q_1|_{\bar{s_1}-s_1},\ g_2=q_2|_{\bar{s_2}-s_2}$$
and $g_1-g_2\not\astarrow_\phi 0$.
Since $q_1|_{\bar{s_1}}=f=q_2|_{\bar{s_2}}$, $\bar{s_1}$ and $\bar{s_2}$ occur in $f$ as bracketed subwords in the forms of \plas $(\bar{s}_1,q_1)$ and $(\bar{s}_2,q_2)$ in $f$. By Theorem~\mref{thm:thrrel},  these \plas $f$ have three possible relative locations.
Accordingly, we will prove that $g_1- g_2\astarrow_\phi 0$ in each of these three cases, yielding the desired contradiction. Before carrying out the proof, we fix some notations.

Let $$Y:=\{ g\in \frakM(Z)\ | g<f\}.$$
By the minimality of $f$ in $\frakN$, for $t\in Y$, all local term forks $(t\to_\phi v_1, t\to_\phi v_2)$ satisfy $v_1-v_2\astarrow_\phi 0$. Note that since the monomial order $\leq$ is compatible with $\Pi_\phi$, $v_1$ and $v_2$ are in $\bfk Y$. In particular, $Y$ is not empty. Then the rewriting system $\Pi_\phi$, which is simple by Corollary~\mref{cor:1nonrec}, restricts to a rewriting system $\Pi_{\phi,Y}$ on $\bfk Y$ which is simple and locally term-confluent. By Lemma~\mref{lem:ltcon}, $\Pi_{\phi,Y}$ on the space $\bfk Y$ is confluent.

Since $s_1, s_2$ are in $S_\phi(Z)$, using Eq.~(\mref{eq:decomp}),  there exist $u, v, r, t \in \frakM(Z)$ such that
\begin{equation}
\begin{array}{lll}
 s_1 &=& \phi(u,v)=\lc u\rc \lc v\rc -\lc B(u,v)\rc=\lc u\rc \lc v\rc-\sum\limits_{i=1}^ka_i \lc B_i(u,v)\rc,\\
s_2 &=& \phi(r,t)\ =\lc r\rc \lc t\rc \,-\, \lc B(r,t)\rc=\lc r\rc \lc t\rc-\sum\limits_{i=1}^ka_i \lc B_i(r,t)\rc.
\end{array}
\mlabel{eq:phimon}
\end{equation}

\smallskip

\noindent {\bf Case I.}
Suppose that the \plas $(\bar{s_1},q_1)$ and $(\bar{s_2},q_2)$ are separated. Then by Definition~\mref{defn:bwrel}, there is $q\in \frakM^{\star_1,\star_2}(Z)$ such that
$$q|_{\bar{s_1},\
\bar{s_2}} = f = q_1|_{\bar{s_1}}=q_2|_{\bar{s_2}}.
$$
By $q_1|_{\star_1}=q|_{\star_1,\,\bar{s_2}}
$ and $q_2|_{\star_2}=q|_{\bar{s_1},\,\star_2},
$
we have
$$g_1=q_1|_{\bar{s_1}-s_1}=q|_{\bar{s_1}-
s_1,\bar{s_2}} \quad {\rm\ and\ } \quad g_2=q_2|_{\bar{s_2}-s_2}
=q|_{\bar{s_1},\bar{s_2}-s_2}.$$
By Eqs.~(\mref{eq:phimon}), we have
$$g_1=\sum_{i=1}^ka_iq|_{\lc B_i(u,v)\rc,\,\bar{s_2}}\quad {\rm\ and }\quad g_2=\sum_{i=1}^ka_iq|_{\bar{s_1},\,\lc B_i(r,t)\rc},$$
and so
\begin{eqnarray*}
g_1-g_2&=&\sum_{i=1}^ka_iq|_{\lc B_i(u,v)\rc, \bar{s_2}}-\sum_{i=1}^ka_iq|_{\bar{s_1},
\lc B_i(r,t)\rc}\\
&=&\sum_{\ell=1}^{2k}c_\ell q_\ell|_{\bar{u_\ell}},
\end{eqnarray*}
where
$$ q_\ell:=\left \{\begin{array}{ll} q|_{\lc B_\ell(u,v)\rc,\star}, & 1\leq \ell\leq k, \\
q|_{\star,\lc B_{\ell-k}(r,t)\rc}, & k+1\leq \ell\leq 2k, \end{array} \right . \quad c_\ell:=\left\{\begin{array}{ll} a_\ell, &1\leq \ell\leq k, \\
-a_{\ell-k}, & k+1\leq \ell\leq 2k. \end{array} \right .
$$
$$
u_\ell:=\left \{\begin{array}{ll}
s_2, &1\leq \ell\leq k, \\
s_1, & k+1\leq \ell\leq 2k, \end{array} \right . \quad
\bar{u_\ell}:= \left \{\begin{array}{ll} \bar{s_2}, &1\leq \ell\leq k, \\ \bar{s_1}, & k+1\leq \ell\leq 2k. \end{array} \right .
$$
Then we have
\begin{eqnarray*}
\sum_{\ell=1}^{2k}c_\ell q_\ell|_{\bar{u_\ell}-u_\ell}&=&
\sum_{i=1}^ka_iq|_{\lc B_i(u,v)\rc, \bar{s_2}-s_2}-\sum_{i=1}^ka_iq|_{
\bar{s_1}-s_1,
\lc B_i(r,t)\rc}\\
&=&\sum_{i,j=1}^ka_iq|_{\lc B_i(u,v)\rc, \lc B_j(r,t)\rc}-\sum_{i,j=1}^ka_iq|_{
\lc B_j(u,v)\rc,
\lc B_i(r,t)\rc}\\
&=&0.
\end{eqnarray*}
Since the monomial order $\leq$ is compatible with $\Pi_\phi$, we have
$$q_\ell|_{\bar{u_\ell}}<q|_{\lc B_\ell(u,v)\rc,\bar{u_\ell}}<q|_{\bar{s_1}
,\bar{s_2}}=f,\quad 1\leq \ell\leq k$$
and
$$
q_\ell|_{\bar{u_\ell}}<q|_{\bar{u_\ell},\lc B_{\ell-k}(r,t)\rc}<q|_{\bar{s_1}
,\bar{s_2}}=f, \quad\, k+1\leq \ell\leq 2k. $$
Then $q_\ell|_{\bar{u_\ell}}$ are in $Y$ for $1\leq \ell\leq 2k$.  By the compatibility of $\leq$ with $\Pi_\phi$ again,  $q_\ell|_{\bar{u_\ell}}>q_\ell|_{\bar{\bar{
u_\ell}-u_\ell}}$, and so $q_\ell|_{\bar{u_\ell}-u_\ell}$ are in $\bfk Y$ for $1\leq \ell\leq 2k$.
By Lemma~\mref{lem:redzero}, $g_1-g_2=\sum\limits_{\ell=1}^{2k}c_\ell q_\ell|_{\bar{u_\ell}}\astarrow_\phi 0$.  This is a contradiction.

\smallskip

\noindent {\bf Case II.}
Suppose  that the \plas $(\bar{s_1},q_1)$ and $(\bar{s_2},q_2)$ are nested. Without loss of generality, assume that there exists $q \in \frakM^\star(Z)$ such that $q_1|_q = q_2$.
We first consider the case when $q = \star$.
Then $q_1 = q_2$. From $q_1|_{\bar{s_1}}=f=q_2|_{\bar{s_2}}$,
we obtain $\bar{s_1}=\bar{s_2}$.
Since $\bar{s_1}=\lc u\rc \lc v\rc $ and $\bar{s_2}=\lc r\rc \lc t\rc $, we must have, by Eq.\,(\mref{eq:decom}),
$\lc u\rc =\lc r\rc$ and
$\lc v\rc =\lc t\rc$. So we have  $u=r$,  $v=t$, and
$ g_1=q_1|_{\lc B(u,v)\rc}=q_2|_{\lc B(u,v)\rc}=g_2,$ and there is nothing more to prove.

Now suppose that $q \ne \star$ and hence $q_1\neq q_2$ and $\bar{s_1} \ne \bar{s_2}$. Since $\bar{s_1} = \lc u \rc \lc v \rc$ and $\bar{s_2} = \lc r\rc\lc t\rc$, by Eq.~(\mref{eq:rel}),
there exists $q' \in \frakM^\star(Z)$ such that
\renewcommand{\theenumi}{{\it\roman{enumi}}}
\begin{enumerate}
\item
either $u = q'|_{\bar{s_2}}$ and $q=\lc q'\rc\,\lc v\rc$; \mlabel{it:incl1}
\item
or $v = q'|_{\bar{s_2}}$ and $q=\lc u\rc \lc q'\rc$. \mlabel{it:incl2}
\end{enumerate}
In subcase (\mref{it:incl1}), we have
$$
g_1-g_2=q_1|_{\bar{s_1}-s_1} -q_2|_{\bar{s_2}-s_2} =q_1|_{\bar{s_1}-s_1} -(q_1|_q)|_{\bar{s_2}-s_2} =q_1|_{{\bar{s_1}-s_1} - q|_{\bar{s_2}-s_2}}.
$$
Since $\lc 0 \rc = 0$, to show $g_1-g_2 \astarrow_\phi 0$, it suffices to prove that ${\bar{s_1}-s_1} - q|_{\bar{s_2}-s_2} = q|_{s_2} - s_1$ (by Eq.\,(\mref{eq:rel})) is $\phi$-reducible to zero. Applying the conditions given in Subcase (\mref{it:incl1})
and expanding $B(x,y) = \sum\limits_{i=1}^k a_i B_i(x,y)$ as a linear combination of distinct monomials $B_i(x,y)$ with non-zero coefficients $a_i \in \bfk$, we have
\begin{eqnarray*}
q|_{s_2} - s_1
&=&\lc q'|_{\bar{s_2}}\rc\lc v\rc-\lc q'|_{\lc B(r,t)\rc}\rc\lc v\rc - \lc u\rc\lc  v \rc + \lc B(u,v)\rc\\
&=&-\lc q'|_{\lc B(r, t)\rc
}\rc\,\lc v\rc + \lc B(q'|_{\lc r\rc \lc t\rc },v)\rc \\
&=& - \sum_{i=1}^k a_i \lc q'|_{\lc B_i(r, t)\rc
}\rc\,\lc v\rc + \sum_{i=1}^k a_i \lc B_i(q'|_{\lc r\rc \lc t\rc },v)\rc\\
&=&
\sum_{\ell=1}^{2k} c_\ell q_\ell|_{\bar{u_\ell}},
\end{eqnarray*}
where
$$ q_\ell:=\left \{\begin{array}{ll} \star, & 1\leq \ell\leq k, \\
\lc B_{\ell-k}(q',v)\rc, & k+1\leq \ell\leq 2k, \end{array} \right . \quad c_\ell:=\left\{\begin{array}{ll} -a_\ell, &1\leq \ell\leq k, \\ a_{\ell-k}, & k+1\leq \ell\leq 2k, \end{array} \right .
$$
$$
u_\ell:=\left \{\begin{array}{ll} \phi(q'|_{\lc B_\ell(r,t)\rc}, v), &1\leq \ell\leq k, \\ \phi(r,t), & k+1\leq \ell\leq 2k, \end{array} \right . \quad
\bar{u_\ell}:= \left \{\begin{array}{ll} \lc q'|_{\lc B_\ell(r,t)\rc}\rc\lc v\rc, &1\leq \ell\leq k, \\ \lc r\rc \lc t\rc, & k+1\leq \ell\leq 2k. \end{array} \right .
$$
Then we have
\begin{eqnarray*}
\sum_{\ell=1}^{2k} c_\ell q_\ell|_{\bar{u_\ell}-u_\ell}&=&\sum_{i=1}^k-a_i\lc B(q'|_{\lc B_i(r, t)\rc},v) \rc
+\sum_{i=1}^ka_{i}\lc B_i(q'|_{\lc B(r,t)\rc},v)\lc\\
&=&- \sum_{i,j=1}^k a_ja_i\lc B_j(q'|_{\lc B_i(r, t)\rc},v)\rc + \sum_{i,j=1}^k a_i a_j \lc B_i(q'|_{\lc B_j(r,t)\rc },v)\rc\\
&=&0.
\end{eqnarray*}

Note that $q_\ell|_{\bar{u_\ell}} < \bar{s_1} \leq q_1|_{\bar{s_1}}=f$ and hence is in $Y$.  Applying Lemma~\mref{lem:redzero} to $\bfk Y$ with $u_\ell, 1\leq \ell\leq 2k,$ given above, we obtain
$$ q|_{s_2} - s_1=\sum_{\ell=1}^{2k} c_\ell q_\ell|_{\bar{u_\ell}}\astarrow_\phi 0.$$
Therefore $g_1-g_2\astarrow_\phi 0.$
Subcase (\mref{it:incl2}) can be similarly treated.
This is again a contradiction.

\smallskip

\noindent {\bf Case III.} Suppose that the \plas $(\bar{s_1},q_1)$ and $(\bar{s_2},q_2)$ are intersecting. Then by Definition~\mref{defn:bwrel}, $q_1 \ne q_2$, and by Remark~\mref{rk:rel}.(\mref{it:rel3}), without loss of generality, we may assume that the partial overlap occurs at a right segment of $\bar{s_1}$ and a left segment of $\bar{s_2}$. Since $\bar{s_1}=\lc u\rc \lc v\rc $ and $\bar{s_2}=\lc r\rc \lc t\rc $, the common segment must be a proper subword, indeed, the subword $\lc v\rc =\lc r\rc$, and so $v=r$. We have $\lc u\rc \lc v\rc \lc t\rc$ appearing in $f$. Let $q \in \frakM^{\star_1,\star_2}(Z)$ be the $(\star_1,\star_2)$-bracketed word obtained by replacing the occurrence of $\lc u\rc \lc v \rc$  in $f$ by $\star_1$ and the occurrence of  $\lc t \rc$ in $f$ by $\star_2$ (thus, $\star_1$ and $\star_2$ are adjacent symbols in $q$).
More precisely, using the convention from Eqs.\,(\mref{eq:2star1}) and (\mref{eq:2star2}), we have
$$q_1 = q^{\star_2}|_{\lc t \rc} \quad {\rm\ and\ } \quad q_2 = q^{\star_1}|_{\lc u\rc},$$
where in the first equation, we identify $\star$ with $\star_1$ and in the second, $\star$ with $\star_2$.
Let $p \in \mapm{Z}^\star$ be the $\star$-bracketed word obtained by replacing $\star_1\star_2$ in $q$ by $\star$. Then using the convention from Eqs.\,(\mref{eq:2star1}) and (\mref{eq:2star2}), we have
\begin{eqnarray*}
g_1 - g_2 &=& q_1|_{\bar{s_1}-s_1} - q_2|_{\bar{s_2}-s_2}\\
&=&(q^{\star_2}|_{\lc
t\rc})|_{(\bar{s_1}-s_1)} - (q^{\star_1}|_{\lc u\rc})|_{(\bar{s_2}-s_2)}\\
&=&p|_{(\bar{s_1}-s_1)\lc t\rc}- p|_{\lc u\rc(\bar{s_2}-s_2)}\\
&=&p|_{((\bar{s_1}-s_1)\lc t\rc - \lc
u\rc(\bar{s_2}-s_2))}\\
&=&p|_{\lc B(u,v)\rc \lc t\rc - \lc u\rc \lc B(v,t)\rc}\\
&\astarrow_\phi&p|_{\lc B(B(u, v),  t)\rc -  \lc B( u,  B(v, t))\rc},
\end{eqnarray*}
where the last step follows from the rewriting
\begin{eqnarray*}
 \lc B(u,v)\rc\lc t\rc -\lc u\rc \lc B(v,t)\rc
 &=&
\lc B(u,v)\rc\lc t\rc \dps (-\lc u\rc \lc B(v,t)\rc) \\
& \rightarrow_\phi & \lc B(B(u,v),t)\rc \dps (-\lc u\rc \lc B(v,t)\rc)\\
& \rightarrow_\phi & \lc B(B(u,v),t)\rc  - \lc B(u,B(v,t))\rc.
\end{eqnarray*}

By assumption (\mref{it:mainr}), $ B(B(u, v),  t) -  B( u, B(v, t)) \astarrow_\phi 0$. So $g_1-g_2 \astarrow_\phi 0$.  This is a contradiction.
\smallskip

This completes the proof of (\mref{it:mainr}) $\Longrightarrow$ (\mref{it:mainc}).

\smallskip

\noindent (\mref{it:mainc}) $\Longrightarrow$ (\mref{it:mainr}). Suppose that the rewriting system $\Pi_\phi$ is convergent. Then it is confluent. Thus for any $u, v, w\in \frakM(Z)$, the fork
\begin{eqnarray*}
\Bigl(\lc u\rc\lc v\rc \lc w\rc \rightarrow_\phi \lc u\rc \lc B(v,w)\rc &\rightarrow_\phi & \lc B(u,B(v,w))\rc, \\
\lc u\rc\lc v\rc \lc w\rc &\rightarrow_\phi &\lc B(u,v)\rc  \lc
w\rc \rightarrow_\phi  \lc B(B(u,v),w)\rc \Bigr)
\end{eqnarray*}
is joinable. Then by Theorem~\mref{thm:downarrow}, we have $\lc B(B(u,v),w)\rc -\lc B(u,B(v,w))\rc \astarrow_\phi 0$ and $B(B(u,v),w)-B(u,B(v,w)) \astarrow_\phi 0$.
This proves (\mref{it:mainr}).
\end{proof}

\section{Rota-Baxter type operators and Gr\"obner-Shirshov basis}
\mlabel{sec:GSu} We now characterize Rota-Baxter type operators in terms of Gr\"obner-Shirshov bases. The main theorem and its proof are given in Section~\mref{ss:main}. The application of the main theorem to the construction of free objects is provided in Section~\mref{ss:constru}.

\subsection{CD lemma and the main theorem}
\mlabel{ss:main} We provide some background and then state the main theorem on  Gr\"obner-Shirshov bases for Rota-Baxter type operators.
\begin{defn}
Let $\leq$ be a monomial order on $\mapm{Z}$.
Let $f, g \in \bfk\mapm{Z}$ be distinct  monic bracketed polynomials
with respect to the monomial order $\leq$ (See Definition~\mref{def:irrS}).
If there exist $\mu,\nu, w\in \frakM(Z)$ such that $w=\bar{f}\mu=\nu\bar{g}$ with $\max\{\bre{\bar{f}},\bre{\bar{g}}\}<\bre{w}< \bre{\bar{f}}+\bre{\bar{g}}$, we call the operated polynomial
$$(f,g)^{\mu,\nu}_w:=f\mu-\nu g$$
the {\bf intersection composition of $f$ and $g$ with respect to $(\mu,\nu)$}.
If there exist $q\in \frakM^{\star}(Z)$ and
$w \in \frakM(Z)$ such that $w =\bar{f}=q|_{\bar{g}}$,  we call the operated polynomial
$$(f,g)^{q}_w:=f-q|_{g}$$ the {\bf including
composition of $f$ and $g$ with respect to $q$}.
\end{defn}

\begin{defn}
Let $\leq$ be a monomial order on $\mapm{Z}$.
 Let $S$ be a set of monic bracketed polynomials in $\bfk\mapm{Z}$ and let $w \in \frakM(Z)$ be a monomial. An operated polynomial $f \in \bfk\mapm{Z}$ is called {\bf trivial modulo $S$ with bound $w$} or, in short, {\bf trivial modulo $(S, w)$} if it can be expressed as $f = \sum_ic_iq_i|_{s_i}$
with $c_i\in \bfk$,  $q_i \in \frakM^{\star}(Z)$, $s_i\in S$ (so $f \in \Id(S)$ by Eq.\,(\mref{eq:repgen})) and $q_i|_{\overline{s_i}}< w$.
\mlabel{de:trivial}
\end{defn}

\begin{defn}
A set $S\subseteq \bfk\mapm{Z}$ of monic bracketed polynomials is called a {\bf Gr\"{o}bner-Shirshov basis with respect to $\leq$} if, for all pairs $f,g\in S$ with $f\neq g$, every intersection composition of the form $(f,g)_w^{\mu,\nu}$ is trivial modulo $S$  with bound $w$, and every including composition of the form $(f,g)_w^q$ is trivial modulo $S$ with bound $w$.
\mlabel{de:GS}
\end{defn}

The following Composition-Diamond Lemma is the basic fact for our study of Rota-Baxter type operators.

\begin{theorem} \rm $($Composition-Diamond Lemma$)$\mcite{BCQ,GSZ}
Let $S$ be a set of monic bracketed polynomials in $\bfk\mapm{Z}$.   Then the following conditions are equivalent.
\renewcommand{\theenumi}{{\it\alph{enumi}}}
\begin{enumerate}
\item $S $ is a Gr\"{o}bner-Shirshov basis in $\bfk\mapm{Z}$.
\mlabel{it:cd1}
\item For every non-zero $f \in \Id(S)$, $\bar{f}=q|_{\overline{s}}$
for some $q \in \frakM^\star(Z)$ and some $s\in S$. \mlabel{it:cd2}
\item  For every non-zero $f \in \Id(S)$, $f$ can be expressed in {\bf triangular form}, that is, in the form
\begin{equation}
f= c_1q_1|_{s_1}+ c_2q_2|_{s_2}+\cdots+ c_nq_n|_{s_n}, \label{eq:fexp1}
\end{equation}
where for $1 \leq i \leq n$, $c_i\in \bfk$ ($c_i \neq 0$), $s_i\in S$, $q_i\in \frakM^{\star}(Z)$ and $ q_1|_{\overline{s_1}}> q_2|_{\overline{s_2}}
> \cdots> q_n|_{\overline{s_n}}$\,. \mlabel{it:cd3}
\item As $\bfk$-modules, $\bfk\mapm{Z}=\bfk\cdot\Irr(S)\oplus \Id(S)$
and $\Irr(S)$ is a $\bfk$-basis of $\bfk\mapm{Z}/\Id(S)$. \mlabel{it:cd4}
\end{enumerate}
\mlabel{thm:CDL}
\end{theorem}

\begin{exam} Let $\phi \in \bfk\mapm{x,y}$ be an OPI of Rota Baxter type and let $S = S_\phi(Z)$. Then by Proposition~\mref{pp:frpio} $\bfk\mapm{Z}/\Id(S) = \bfk\mapm{Z}/I_\phi(Z)$ is the free $\phi$-algebra $\bfk_\phi\mapm{Z}$. If $S_\phi(Z)$ is a Gr\"obner-Shirshov basis of $\bfk\mapm{Z}$, then $\Irr(S_\phi(Z))$ is a basis of $\bfk_\phi\mapm{Z}$.\end{exam}

We next review more general reduction relations for operated polynomial algebras $\bfk\mapm{Z}$ from \mcite{GSZ}. These relations generalize those for polynomial algebras $\bfk[Z]$ (\cite[Section 8.2]{BN}) and, under suitable conditions, we will show they include $\astarrow_\phi$.

\begin{defn}
Let $s\in \bfk\mapm{Z}$ be monic with leading term $\bar{s}$. We use $s$ to define the following {\bf reduction relation} $\rightarrow_s$: For $g, g'\in \bfk\mapm{Z}$, let $g\rightarrow_s g'$ denote the relation that there are some $c \in \bfk$ ($c \ne 0)$ and some $q\in \frakM^\star(Z)$ such that
\begin{enumerate}
\item
$g=c q|_{\bar{s}}\dps R(g)$, where $R(g):=g-cq|_{\bar{s}}\in\bfk\mapm{Z}$;
\item
$g'=c q|_{\bar{s}-s}+R(g)$.
\end{enumerate}
Equivalently, we  say that $g\rightarrow_s g'$ if there are some $c \in \bfk$ ($c \ne 0)$ and some $q\in \frakM^\star(Z)$ such that
\begin{enumerate}
\item
$q|_{\bar{s}}$ is a monomial of $g$ with coefficient $c$;
\item
$g'=g-cq|_{s}$.
\end{enumerate}\mlabel{def:redrel}
\end{defn}
If $S$ is a set of monic bracketed polynomials, we let $g \rightarrow_S g'$ denote the relation that $g \rightarrow_s g'$ for some $s \in S$. The reflexive transitive closure of $\rightarrow_s$ and $\rightarrow_S$ are denoted by $\astarrow_s$ and $\astarrow_S$, respectively.

\begin{lemma}
Let $\phi(x,y):=\lc x\rc \lc y\rc -\lc B(x,y)\rc \in \bfk\mapm{x,y}$ with $B(x,y)$ in {\rm RBNF} and totally linear in $x, y$. Let $Z$ be a set, let $S = S_\phi(Z)$ and let $\leq$ be a monomial order on $\frakM(Z)$ that is compatible with $\Pi_\phi(Z)$. Let $g\in\bfk\mapm{Z}$ and $w\in \mapm{Z}$ be given with $w>\bar{g}$. If $g\astarrow_\phi0$, then $g$ is trivial modulo $(S,w)$.
\mlabel{lem:roc2}
\end{lemma}

\begin{proof}
By Theorem~\mref{thm:roc1},  the rewriting system $\Pi_\phi$ is terminating. From $g\astarrow_\phi0$, there exist $k\geq 1$ and $g_i\in \bfk\mapm{Z}, 1\leq i\leq k$, such that
$$g=:g_1\to_\phi g_2\to_\phi \cdots\to_\phi g_k:=0.$$
\begin{claim}
We have $\bar{g_{i+1}}\leq \bar{g_i}, 1\leq i\leq k-1$.
\mlabel{cl:les}
\end{claim}
\begin{proof}
For $1\leq i\leq k$, we can rewrite $g_i$ as
\begin{equation}
g_i=(c_{i,1}u_{i,1}+\cdots +c_{i,n_i}u_{i,n_i})\dps g_{i,2},
\mlabel{eq:gsum}
\end{equation}
where for $1\leq j\leq n_i$, $0\neq c_{i,j}\in\bfk$, $u_{i,j}$ is not in RBNF, $u_{i,1}>\cdots>u_{i,n_i}$ with respect to $\leq$ and $g_{i,2}\in\bfk\mapm{Z}$ is in RBNF.
For any reduction $g_i\to_\phi g_{i+1}$, $1\leq i\leq k-1$, there exist $q_i\in\frakM^{\star}(Z)$, $u_i,v_i\in\frakM(Z)$ and $0\neq c_i\in\bfk$ such that
$$g_i=c_iq_i|_{\lc u_i\rc\lc v_i\rc}\dps R(g_i) \quad \text{and} \quad g_{i+1}=c_iq_i|_{\lc B(u_i,v_i)\rc} +R(g_i).$$
Then $q_i|_{\lc u_i\rc \lc v_i\rc}$ is a monomial of $g_i$ and is not in RBNF. Since $g_i=(c_{i,1}u_{i,1}+\cdots +c_{i,n_i}u_{i,n_i})\dps g_{i,2}$, there exists a natural number $1\leq r_i\leq n_i$ such that $u_{i,r_i}=q_i|_{\lc u_i\rc\lc v_i\rc}$ and $c_i=c_{i,r_i}.$ Then $$R(g_i)=(c_{i,1}u_{i,1}+\cdots+c_{i,n_i}u_
{i,n_i}-c_{i,r_i}u_{i,r_i}) + g_{i,2}.$$
So we get
\begin{equation}
g_{i+1}=c_iq_i|_{\lc B(u_i,v_i)\rc}+(c_{i,1}u_{i,1}+\cdots+c_{i,n_i}u_
{i,n_i}-c_{i,r_i}u_{i,r_i}) + g_{i,2}.
\mlabel{eq:gsum2}
\end{equation}
We distinguish two cases, depending on whether or not  $\bar{g_i} =\bar{g_{i,2}}$.

\smallskip
\noindent
{\bf Case 1. Suppose $\bar{g_i}=\bar{g_{i,2}}$:} By Eq.~(\mref{eq:gsum}),  $\bar{g_{i,2}}>u_{i,1}$. By the compatibility of $\phi$ with $\leq$ and the monomial property in Eq.~(\mref{eq:morder}), we have
$$u_{i,r_i}=q_i|_{\lc u_i\rc \lc v_i\rc}>q_i|_{\bar{\lc B(u_i,v_i)\rc}}.$$
Since $u_{i,1}\geq u_{i,r_i}$, we have
$$\bar{g_{i,2}}>u_{i,1}\geq u_{i,r_i}>q_i|_{\bar{\lc B(u_i,v_i)\rc}}.$$
By $u_{i,1}>\cdots>u_{i,n_i}$ and Eq.~(\mref{eq:gsum2}), we get
$\bar{g_{i+1}}=\bar{g_{i,2}}$. Thus $\bar{g_{i+1}}=\bar{g_i}$.

\smallskip
\noindent
{\bf Case 2. Suppose $\bar{g_i}\neq \bar{g_{i,2}}$:} Then $\bar{g_i}> \bar{g_{i,2}}$. By Eq.~(\mref{eq:gsum}), we have $\bar{g_i}=u_{i,1}$ and then $\bar{g_{i,2}}<u_{i,1}$.
By the compatibility of $\phi$ with $\leq$ and the monomial property in Eq.~(\mref{eq:morder}), we have
$$u_{i,1}\geq u_{i,r_i}=q_i|_{\lc u_i\rc\lc v_i\rc}> q_i|_{\bar{\lc B(u_i,v_i)\rc}}\,.$$
There are two subcases to consider.
First assume that $u_{i,1}=u_{i,r_i}$.
Then $u_{i,1}=q_i|_{\lc u_i\rc \lc v_i\rc}$. By Eq.~(\mref{eq:gsum2}), we get
\begin{equation}
g_{i+1}=c_iq_i|_{\lc B(u_i,v_i)\rc }+(c_{i,2}u_{i,2}+\cdots+c_{i,n_i}u_{i,
n_i})+g_{i,2}.
\end{equation}
So $\bar{g_{i+1}}\leq \max\{q_i|_{\bar{\lc B(u_i,v_i)\rc}},\, u_{i,2},\, \bar{g_{i,2}}\}$. Since $q_i|_{\bar{\lc B(u_i,v_i)\rc}}<u_{i,1}, u_{i,2}<u_{i,1}$ and $\bar{g_{i,2}}<u_{i,1}$, we have $\bar{g_{i+1}}<u_{i,1}$. Hence $\bar{g_{i+1}}<\bar{g_i}$.
Next assume $u_{i,1}>u_{i,r_i}$. Then $2\leq r_i\leq n_i$. By Eq.~(\mref{eq:gsum2}), we can rewrite $g_{i+1}$ as
$$g_{i+1}=c_{i,1}u_{i,1}+c_{i,2}u_{i,2}
+\cdots+c_{i,r_i}q_i|_{\lc B(u_i,v_i)\rc}+\cdots+c_{i,n_i}u_{i,n_i}+
g_{i,2}.$$
Since $u_{i,1}>u_{i,r_i}>q_i|_{\bar{\lc B(u_i,v_i)\rc}}$, $u_{i,1}>\bar{g_{i,2}}$ and $u_{i,1}>u_{i,2}>\cdots>u_{i,n_i}$, we have $\bar{g_{i+1}}=u_{i,1}$. Thus $\bar{g_{i+1}}=\bar{g_i}$.

In summary, we have $\bar{g_{i+1}}\leq \bar{g_i}, 1\leq i\leq k-1$.
\end{proof}

Now we continue with the proof of Lemma~\mref{lem:roc2}. By $g_i\to_\phi g_{i+1}$ there exist $q_i\in\frakM^{\star}(Z), s_i:=\lc u_i\rc \lc v_i\rc-\lc B(u_i,v_i)\rc\in S$ and $0\neq c_i$ such that
$g_{i+1}=g_i-c_iq_i|_{s_i}$.
So we have $g_i=g_{i+1}+c_iq_i|_{s_i}$.
From the finite reduction sequence $g=:g_1\to_\phi g_2\to_\phi \cdots\to_\phi g_{k}:=0,$
we get
\begin{eqnarray*}
g_1&=&g_2+c_1q_1|_{s_1}\\
&=&g_3+c_2q_2|_{s_2}+c_1q_1|_{s_1}\\
&\cdots&\\
&=&g_{k}+c_{k-1}q_{k-1}|_{s_{k-1}}+\cdots+c_1q_1
|_{s_1}.\\
\end{eqnarray*}
Since $g_{k}=0$, we have
$$g_1=c_1q_1|_{s_1}+\cdots+c_{k-1}q_{k-1}|_{s_{k-1}}.$$
By Claim~\mref{cl:les}, we have
$$\bar{g_{k}}\leq \bar{g_{k-1}}\leq \cdots\leq \bar{g_1}.$$
Since $c_iq_i|_{s_i}=g_i-g_{i+1}$, we get $q_i|_{\bar{s_i}}\leq \max\{\bar{g_i},\bar{g_{i+1}}\}=\bar{g_i}\leq \bar{g_1}<w$ by our choice of $w$. This means that $g$ is trivial modulo $(S,w)$.
\end{proof}

Together with Corollary~\mref{co:rbcon}, the following theorem characterizes Rota-Baxter operators in terms of convergent rewriting systems and Gr\"obner-Shirshov bases.

\begin{theorem} Let $\phi(x,y):=\lc x\rc \lc y\rc -\lc B(x,y)\rc \in \bfk\mapm{x,y}$ with $B(x,y)$ in {\rm RBNF} and totally linear in $x, y$. Let $Z$ be a set, and let $\leq$ be a monomial order on $\frakM(Z)$ that is compatible with $\Pi_\phi(Z)$.
The following conditions are equivalent.
\begin{enumerate}
\item
The rewriting system $\Pi_\phi(Z)$ is convergent. \mlabel{it:mainc2}
\item
With respect to $\leq$, the set $S:=S_\phi(Z)$
is a Gr\"{o}bner-Shirshov basis in $\bfk\mapm{Z}$. \mlabel{it:maings}
\end{enumerate}
\mlabel{thm:gsrb}
\end{theorem}

\begin{proof}
\noindent (\mref{it:mainc2}) $\Longrightarrow$ (\mref{it:maings}) Suppose that the rewriting system $\Pi_\phi$ is convergent.

Let two elements $f$ and $g$ of $S_\phi(Z)$ be given with $f\neq g$. They are of the form
$$f:=\phi(u,v), \quad g:=\phi(r,s), \quad u,v,r,s\in \frakM(Z).$$
Since $\Pi_\phi$ is compatible with $\leq$, we have $\bar{f}=\lc u\rc \lc v\rc $ and $\bar{g}=\lc r\rc \lc s\rc $.
\smallskip

\noindent {\bf (The case of intersection compositions). }  Suppose that $w=\bar{f}\mu=\nu\bar{g}$ gives an intersection composition, where $\mu,\nu\in \frakM(Z)$. Since $\bre{\bar{f}}=\bre{\bar{g}}=2$, we must have $\bre{w}<\bre{\bar{f}}+\bre{\bar{g}}=4$. Thus $\bre{w}=3$. This means that $\bre{\mu}=\bre{\nu}=1$. Since $f, g$ are monic, we have $\mu=\lc s\rc , \nu=\lc u\rc .$ Thus $w=(\lc u\rc \lc v\rc )\lc s\rc =\lc u\rc (\lc r\rc \lc s\rc ).$  Then we have $v=r$ and
$$\begin{array}{llllll}
(f,g)_w^{\mu,\nu}&=&f\mu-\nu g\\
&=&(\lc u\rc \lc v\rc -\lc B(u,v)\rc)\mu-\nu (\lc r\rc \lc s\rc -
\lc B(r,s)\rc)\\
&=&-(\lc B(u,v)\rc \lc s\rc-\lc u\rc \lc B(v,s)\rc).\\
\end{array}$$
Since $$\lc B(u,v)\rc \lc s\rc_{\phi}\!\!\!\leftarrow\lc u\rc \lc v\rc \lc s\rc \rightarrow_{\phi}\lc u\rc \lc B(v,s)\rc  $$ and $\Pi_\phi$ is convergent, by Theorem~\mref{thm:downarrow},  we have
$$\lc B(u,v)\rc \lc s\rc-\lc u\rc \lc B(v,s)\rc\astarrow_\phi 0$$
Since
$$\overline{\lc B(u,v)\rc \lc s\rc-\lc u\rc \lc B(v,s)\rc} \leq
\max\Big\{~\overline{\lc B(u,v)\rc \lc s\rc}, ~\overline{\lc u\rc \lc B(v,s)\rc}~\Big\}<\lc u\rc\lc v\rc \lc s\rc.$$
By Lemma~\mref{lem:roc2}, $\lc B(u,v)\rc \lc s\rc-\lc u\rc \lc B(v,s)\rc$ is trivial modulo $(S,\lc u\rc \lc v\rc \lc s\rc).$

\smallskip

\noindent {\bf (The case of including compositions).} Suppose that $w:=\overline{f}=q|_{\overline{g}}$. Then we have $w:=\lc u\rc \lc v\rc =q|_{\lc r\rc \lc s\rc }$. If $q=\star$, then $\lc u\rc \lc v\rc=\lc r\rc \lc s\rc$. Thus $u=r$ and $v=s$. Hence $f=\phi(u,v)=g$, a contradiction to the hypothesis that $f\neq g$. Then we get $q\neq \star$.  So the $\star$ in $q$ can come from either $u$ or $v$. Thus $f$ and $g$ could only have the following including compositions.
\begin{enumerate}
\item
If $u=q'|_{\lc r\rc \lc s\rc }$ for some $q'\in \frakM^{\star}(Z)$, then
$$\overline{f}=q|_{\overline{g}}=\lc q'|_{\lc r\rc \lc s\rc }\rc\,
\lc v\rc $$ with $q:=\lc q'\rc\,\lc v\rc $.
\item
If $v=q'|_{\lc r\rc \lc s\rc }$ for some $q'\in \frakM^{\star}(Z)$, then
$$\overline{f}=q|_{\overline{g}}=\lc u\rc \lc q'|_{\lc r\rc \lc s\rc
}\rc$$ with $q:=\lc u\rc \lc q'\rc$.
\end{enumerate}

So we just need to check that in both cases these compositions are trivial modulo $(S,w)$. Consider the first case. Using the notation in Eq.~(\mref{eq:phimon}), this composition is

\allowdisplaybreaks{
\begin{eqnarray*}
(f,g)_w^q&=& f- q|_g \\
&=& \lc u\rc \lc v\rc -\sum_{i=1}^ka_i\lc B_i(u,v)\rc
- \lc q'|_{g}\rc\,\lc v\rc \\
&=& \lc q'|_{\lc r\rc \lc s\rc }\rc\,\lc v\rc  -\sum_{i=1}^k a_i
\lc B_i(q'|_{\lc r\rc \lc s\rc },v)\rc\\
&&-\left( \lc q'|_{\lc r\rc \lc s\rc }\rc\,\lc v\rc  -\sum_{i=1}^k a_i
\lc q'|_{\lc B_i(r,s)\rc}\rc\,\lc v\rc \right)\\
&=&-\sum_{i=1}^k a_i \lc B_i(q'|_{\lc r\rc \lc s\rc },v)\rc +\sum_{i=1}^k a_i
\lc q'|_{\lc B_i(r,s)\rc}\rc\,\lc v\rc \\
&=& -\sum_{i=1}^k a_i \lc B_i(q'|_{\phi(r,s)},v)\rc-\sum_{i=1}^k
a_i \lc B_i(q'|_{\lc B(r,s)\rc},v)\rc\\
&&+\sum_{i=1}^k a_i\phi(q'|_{\lc B_i(r,s)\rc}, v) +\sum_{i=1}^k a_i
\lc B(q'|_{\lc B_i(r,s)\rc}, v)\rc\\
&=& -\sum_{i=1}^k a_i \lc B_i(q'|_{\phi(r,s)},v)\rc+\sum_{i=1}^k a_i
\phi(q'|_{\lc B_i(r,s)\rc}, v)\\
&& -\sum_{i=1}^k a_i  \sum_{j=1}^k a_j\lc B_i(q'|_{\lc B_j(r,s)\rc},v)\rc +\sum_{i=1}^k a_i \sum_{j=1}^k a_j \lc B_j(q'|_{\lc
B_i(r,s)\rc}, v)\rc \\
&=&-\sum_{i=1}^k a_i \lc B_i(q'|_{\phi(r,s)},v)\rc+\sum_{i=1}^k a_i \phi(q'|_{\lc B_i(r,s)\rc}, v)
\end{eqnarray*}}
since the double sums become the same after exchanging $i$ and $j$. Let $q_i=\lc B_i( q', v)\rc\in \frakM^{\star}(Z)$, $1\leq i\leq k$.   Then
$$\sum_{i=1}^k a_i \lc B_i(q'|_{\phi(r,s)},v)\rc=\sum_{i=1}^k
a_i q_i|_{\phi(r,s)}.$$
Further
$$ q_i|_{\overline{\phi(r,s)}}=q_i|_{\lc r\rc\lc s\rc}=\lc B_i(q'|_{\lc r\rc\lc s\rc},v)\rc
=\lc B_i(u,v)\rc < \lc u\rc \lc v\rc =w.$$
Thus the first sum is trivial modulo $(S,w)$.

For the second sum, we have
$$\sum_{i=1}^k a_i
\phi(q'|_{\lc B_i(r,s)\rc}, v)=\sum_{i=1}^k a_i q_i|_{u_i},$$ where $q_i = \star$ and $u_i:=\phi(q'|_{\lc B_i(r,s)\rc}, v) \in S$. Further,
$${q_i|_{\overline{u_i}}}=\overline{u_i}=
\overline{\phi(q'|_{\lc B_i(r,s)\rc}, v)}= \lc{q'|_{\lc B_i(r,s)\rc}}\rc\lc v\rc<\lc q'|_{\lc r\rc \lc s\rc}\rc\lc v\rc=w.$$
Hence the second sum is also trivial modulo $(S,w)$. This proves $(f,g)_w^q$ is trivial modulo $ (S,w)$. The proof of the second case is similar.

\smallskip

\noindent (\mref{it:maings}) $\Longrightarrow$ (\mref{it:mainr})
By Theorem~\mref{thm:roc1}, we conclude that $\Pi_\phi$ is terminating.
By Corollary ~\mref{co:rbcon}, it remains to verify that $B(B(u,v),w)-B(u,B(v,w)) \astarrow_\phi 0$  for $u,v,w\in \frakM(Z)$. We prove this by contradiction. Suppose that there are $u, v, w\in \frakM(Z)$ such that $B(B(u,v),w)-B(u,B(v,w))$ is not $\phi$-reducible to zero. Then we have $B(B(u,v),w)-B(u,B(v,w))\rightarrow_\phi \cdots \rightarrow_\phi  G$ where $G\neq 0$ is in RBNF. Thus $\lc B(B(u,v),w)\rc-\lc B(u,B(v,w))\rc \rightarrow_\phi \cdots \rightarrow_\phi  \lc G\rc$. Note that
\begin{eqnarray*}
&&\lc B(B(u,v),w)\rc -\lc B(u,B(v,w))\rc\\
&=& \lc B(B(u,v),w)\rc - \lc B(u,v)\rc\lc w\rc +\lc B(u,v)\rc \lc
w\rc - \lc u\rc \lc v\rc \lc w\rc \\
&& - \lc B(u,B(v,w))\rc + \lc u\rc \lc B(v,w)\rc - \lc u\rc \lc B(v,w)\rc + \lc u\rc \lc v\rc \lc w\rc
\end{eqnarray*}
is in $\id(S)$. Hence $\lc G\rc$ is in $\Id(S)$. Since $S$ is a Gr\"obner-Shirshov basis, by Theorem~\mref{thm:CDL}, there are $q\in \frakM(Z)$ and $s\in S$ such that $\overline{\lc G\rc}=q|_{\bar{s}}$. This shows that $\lc G\rc $ is not in RBNF. Hence $G$ is not in RBNF, a contradiction. In summary, $\phi$ is of Rota-Baxter type. Thus by Corollary ~\mref{co:rbcon},
(\mref{it:mainc2}) follows.
\smallskip

This completes the proof of Theorem~\mref{thm:gsrb}.
\end{proof}

\subsection{Construction of free $\phi$-algebra}
\mlabel{ss:constru}
We next give the following explicit construction of free objects in the category of algebras with a given Rota-Baxter type operator. As we will see in Theorem~\mref{thm:exam}, this construction applies to all the operators in the list of Conjecture~\mref{con:rbt} and thus provides a uniform exposition compared with the previously separate case-by-case construction method~\mcite{Agg,EG1,Gub,LG,ZhC,ZG}.

Recall from Proposition~\mref{pp:frpio} that $\bfk_\phi\mapm{Z} = \bfk\mapm{Z}/I_\phi(Z)$ is the free $\phi$-algebra on $Z$. Let $\rbw(Z)$ be the set of bracketed words in $\frakM(Z)$ in RBNF. Then $\rbw(Z)$ is closed under the operator $P_r:=\lc\,\rc$. Let $\kdot\rbw(Z)$ be the free $\bfk$-module with basis $\rbw(Z)$ and let the operator  $P_r$ on $\rbw(Z)$ be extended $\bfk$-linearly to $\kdot\rbw(Z)$. Then $(\kdot\rbw(Z), P_r)$ is an operated $\bfk$-module as defined in Definition~\mref{de:mapset}.

\begin{theorem} Let $\phi(x,y):=\lc x\rc \lc y\rc -\lc B(x,y)\rc  \in
\bfk\mapm{x,y}$ be of Rota-Baxter type. Let $Z$ be a set and let $\leq$ be a monomial order on $\frakM(Z)$. Suppose that $\phi$ is compatible with $\leq$. Then:
\begin{enumerate}
\item
The composition of natural  $\bfk$-module morphisms
$$ \kdot\rbw(Z) \overset{\iota}{\to} \kdot\frakM(Z) \equiv \bfk\mapm{Z}  \overset{\eta}{\to} \bfk_\phi\mapm{Z}$$
is an   isomorphism and $\eta \circ \iota(\rbw(Z))$ is a $\bfk$-basis of $\bfk_\phi\mapm{Z}$.
 \mlabel{it:freea}
\item
Let $\alpha:\bfk_\phi\mapm{Z}\to \kdot\rbw(Z)$ be the inverse of $\eta\circ \iota$ from {\rm Part}~{\rm (\mref{it:freea})} and let \Reduce\ be the composition $\bfk\mapm{Z} \overset{\eta}{\rightarrow} \bfk_\phi\mapm{Z} \overset{\alpha}{\rightarrow} \kdot\rbw(Z)$.   Then $(\kdot\rbw(Z), \diamondsuit, P_r)$ is a free $\phi$-algebra, where the multiplication $\diamondsuit$ on $\kdot\rbw(Z)$ is defined on $\rbw(Z)$ as follows and extended by bilinearity. For any bracketed words $u, v \in \rbw(Z):$
\begin{enumerate}
\item
$1 \diamondsuit u =u \diamondsuit 1: = u$, where $1$ is the empty word in $\frakM(Z);$\mlabel{it:diapr0}
\item
$u \diamondsuit v: = uv$ if either $u \in S(Z)$ or $v \in S(Z);$ \mlabel{it:diapr1}
\item
$u \diamondsuit v: =
\lc \Reduce(B(\mpu,\mpv))\rc$ if $u=\lc \mpu\rc$ and $v=\lc \mpv\rc$ are both in $\lc \rbw(Z)\rc;$
\mlabel{it:diapr2}
\item
$u\diamondsuit
v:=u_1\cdots u_{s-1}(u_s \diamondsuit v_1)v_2\cdots v_t$ if the standard decomposition
$u_1\cdots u_{s-1}u_s$ of $u$ has $s > 1$ or the standard decomposition $ v_1v_2\cdots v_t$ of $v$ has $t > 1$.
Here except for $u_s \diamondsuit
v_1$, the rest are concatenations as in the standard decompositions defined in {\rm Eq.~(\mref{eq:stde})}.
\mlabel{it:diapr3}
\end{enumerate}
\mlabel{it:freeb}
\end{enumerate}
\mlabel{thm:free}
\end{theorem}
\begin{proof}
(\mref{it:freea}) By Corollary ~\mref{co:rbcon} and Theorem~\mref{thm:gsrb}, $S=S_\phi(Z)$ is a Gr\"obner-Shirshov basis in $\bfk\mapm{Z}$ with respect to $\leq$. Hence by Theorem~\mref{thm:CDL}, $\Irr(S)$ is a $\bfk$-basis of $\bfk\mapm{Z}/\Id(S)$. Since $\Irr(S)=\rbw(Z)$ and $\Id(S)=I_\phi(Z)$ from their definitions, Part~(\mref{it:freea}) follows.
\smallskip

Before proving Theorem~\mref{thm:free} (\mref{it:freeb}), we give the following lemma.

\begin{lemma}
We keep the same notations as in Theorem~\mref{thm:free}~$($\mref{it:freeb}$)$.
\begin{enumerate}
\item The linear maps $\iota, \eta, \alpha$ and $\Reduce$ are operated $\bfk$-module morphisms.
\mlabel{it:mor}
\item For any $f,g\in\bfk\mapm{Z}$ such that $f\to_\phi g$, we have $\Reduce(f)=\Reduce(g)$;
\mlabel{it:redeq}
\item
For any $f\in\bfk\mapm{Z}$, there exists $f'\in\bfk\mapm{Z}$ in RBNF such that $f\astarrow_\phi f'$, and $\Reduce(f)=f'$.
\mlabel{it:redrb}
\end{enumerate}
\mlabel{lem:redur}
\end{lemma}
\begin{proof}
(\mref{it:mor})
By definition, $\rbw(Z)$ is closed under taking the bracket. So $\kdot\rbw(Z)$ is an operated $\bfk$-module. Then the embedding $\iota: \kdot\rbw(Z) \to \bfk\mapm{Z}$ is an operated $\bfk$-module morphism. Since $\eta: \bfk\mapm{Z} \to \bfk_\phi\mapm{Z}$ is a quotient map of operated $\bfk$-modules and hence an operated $\bfk$-module morphism, the composition $\eta\circ \iota$ is also one. Since $\eta\circ \iota$ is a linear bijection by Theorem~\mref{thm:free}(\mref{it:freea}), its inverse $\alpha$ is also an operated $\bfk$-module morphism. Then the composition $\Reduce:=\alpha\circ \eta$ is also one.

\smallskip

\noindent
(\mref{it:redeq})
By $f\to_\phi g$ and Definition~\mref{def:redsys}, there exist $q\in\frakM^{\star}(Z)$, $u,v\in\frakM(Z)$ and $0\neq c\in \bfk$ such that $$g=f-cq|_{(\lc u\rc \lc v\rc-\lc B(u,v)\rc)}.$$
Then $f-g=cq|_{(\lc u\rc\lc v\rc-\lc B(u,v)\rc)}$ is in $\Id(S)$.
So $f-g+\Id(S)=0+\Id(S)$ is in $\bfk_\phi\mapm{Z}$.
Then by Lemma~\mref{lem:redur}.(\mref{it:mor}), we have
$$\Reduce(f)-\Reduce(g)=\Reduce(f-g)
=\alpha(\eta(f-g))
=\alpha(f-g+\Id(S))=0.
$$

\smallskip

\noindent
(\mref{it:redrb})
By Corollary ~\mref{co:rbcon}, $\Pi_\phi$ is convergent. Hence $\Pi_\phi$ is terminating. Then there exists $f'\in\bfk\mapm{Z}$ in RBNF such that $f\astarrow_\phi f'$.
We can assume that the reduction relation $f\astarrow_\phi f'$ is given by the finite reduction sequence $f\to_\phi f_1\to_\phi f_2\to_\phi\cdots\to_\phi f'.$  By Lemma~\mref{lem:redur} ~(\mref{it:redeq}), we have
$$ \Reduce(f)=\Reduce(f_1)=\Reduce(f_2)
=\cdots=\Reduce(f').$$
Since $f'$ is in RBNF, we have $\eta(f')=(\eta\circ\iota)(f')
=\alpha^{-1}(f')$. So
$$\Reduce(f')=(\alpha\circ \eta)(f')
=\alpha(\alpha^{-1}(f'))=f'.$$
\end{proof}

\noindent (\mref{it:freeb})
We now prove Theorem~\mref{thm:free}(\mref{it:freeb}).
Since $\alpha:\bfk_\phi\mapm{Z}\to \bfk\cdot \rbw(Z)$ is an operated $\bfk$-module isomorphism by Theorem~\mref{thm:free}(\mref{it:freea}) and Lemma~\mref{lem:redur}(\mref{it:mor}), we can transport the structure of a free $\phi$-algebra on $\bfk_\phi\mapm{Z}$ to $\kdot\rbw(Z)$. More precisely, denote the multiplication and the linear operator on the free $\phi$-algebra $\bfk_\phi\mapm{Z}$ by $\odot$ and $P_Z:=\lc\,\rc \,(\!\!\!\mod \Id(S))$ respectively. We define
$$u\,\diamondsuit'v:=\alpha(\alpha^{-1}(u)
\odot\alpha^{-1}(v))\quad
\text{and} \quad P'_r(u):=\alpha(P_Z(\alpha^{-1}(u)))\quad
\text{for any $u,v\in\rbw(Z)$}.$$
Then $(\kdot\rbw(Z),\diamondsuit',P_r')$
is a $\phi$-algebra isomorphic to $(\bfk_\phi\mapm{Z},\odot,P_Z)$, and hence is a free $\phi$-algebra on $Z$.

Since $u\in\rbw(Z)$ and $\alpha^{-1}(u)=u +\Id(S)$, we have $$P_r'(u)=\alpha(P_Z(\alpha^{-1}(u)))
=\alpha(\lc u\rc+\Id(S))=\lc u\rc=P_r(u).$$
Thus, we get $P_r'=P_r=\lc\,\rc$.

It remains to prove that $u\diamondsuit'v=u\diamondsuit v$ for any $u,v\in\rbw(Z)$. Note that $\alpha^{-1}(u)=u+\Id(S)=\eta(u)$ and $\alpha^{-1}(v)=v+\Id(S)=\eta(v)$.
By Lemma~\mref{lem:redur}(\mref{it:mor}), $\eta:\bfk\mapm{Z} \rightarrow \bfk_\phi\mapm{Z}$ is an operated $\bfk$-algebra homomorphism. Hence we have
$$\alpha^{-1}(u)\odot \alpha^{-1}(v)
=\eta(u)\odot\eta(v)=\eta(uv).$$
Thus $$u\diamondsuit'v=
\alpha(\alpha^{-1}(u)\odot \alpha^{-1}(v))=\alpha(\eta(uv))= \Reduce(uv).$$
So we just need to show that $\Reduce(uv)=u\diamondsuit v$. For any given $u,v\in\rbw(Z)$,
let $u=u_1\cdots u_{s-1}u_s$ and $v= v_1v_2\cdots v_t$ be the standard decompositions defined in {\rm Eq.~(\mref{eq:stde})}. Then $u_i$ is alternately in $ S(Z)$ or in $\lc \rbw(Z)\rc, 1\leq i\leq s$, and $v_j$ is also alternately in $ S(Z)$ or in $ \lc \rbw(Z)\rc, 1\leq j\leq t$.
First consider $s=t=1$.
If $v=1$ is the empty word in $\frakM(Z)$, then $$\Reduce(u1)=\Reduce(u)
=u.$$
Thus, we have $\Reduce(u1)=u=u\diamondsuit 1$. Similarly, $\Reduce(1u)= u=1\diamondsuit u$.
If $u\in S(Z)$ or $v\in S(Z)$, then $uv\in\rbw(Z)$. Thus we have $\Reduce(uv)=uv=u\diamondsuit v$.
If $u=\lc u^\ast\rc $ and $v=\lc v^\ast \rc$ with $u^\ast,v^\ast\in\rbw(Z)$,
then we have $uv=\lc u^\ast\rc\lc v^\ast\rc\to_\phi \lc B(u^\ast,v^\ast)\rc$. By Lemma~\mref{lem:redur}(\mref{it:redeq}), we get $\Reduce(uv)=\Reduce(\lc B(u^\ast,v^\ast)\rc)$.
By Lemma~\mref{lem:redur}(\mref{it:mor}), $\Reduce$ preserves the brackets. Then we get
\begin{equation}
\Reduce(\lc u^\ast \rc \lc v^\ast\rc)=\Reduce(\lc B(u^\ast,v^\ast)\rc)=\lc \Reduce(B(u^\ast,v^\ast))\rc=:u\diamondsuit v.
\mlabel{eq:redu}
\end{equation}
Now we consider $s>1$ or $t>1$.
If $u_s\in S(Z)$ or $v_1\in S(Z)$, then $u_sv_1$ is in $\rbw(Z)$. Thus, $uv=u_1\cdots u_{s-1}u_sv_1v_2\cdots v_t$ is in $ \rbw(Z)$. Then we have $$\Reduce(uv)=uv=u_1\cdots u_{s-1} (u_s\diamondsuit v_1)v_2\cdots v_t= u\diamondsuit v.$$
If $u_s=\lc u_s^\ast \rc$ and $v_1=\lc v_1^\ast\rc$ with $u_s^\ast,v_1^\ast\in\rbw(Z)$, then $u_{s-1}$ and $v_2$, if exist, are in $S(Z)$.
Since $uv=u_1\cdots u_{s-1} (\lc u_s^\ast\rc \lc  v_1^\ast\rc)v_2\cdots v_t$ while $u_1\cdots u_{s-1}$ and $v_2\cdots v_t$ are in $\rbw(Z)$, the rewriting system $\Pi_\phi$ can only be applied to $\lc u_s^\ast\rc\lc v_1^\ast\rc$. Since $\Pi_\phi$ is terminating, there exists $ h\in\bfk\mapm{Z}$  in RBNF such that
\begin{equation}
\lc u_s^\ast\rc\lc v_1^\ast\rc\to_\phi \lc B(u_s^\ast,v_1^\ast)\rc \astarrow_\phi \lc h\rc.
\mlabel{eq:buv}
\end{equation}
Then we have
\begin{eqnarray*}
uv&=&u_1\cdots u_{s-1} (\lc u_s^\ast\rc \lc  v_1^\ast\rc)v_2\cdots v_t\\
&\to_\phi& u_1\cdots u_{s-1} (\lc B( u_s^\ast,  v_1^\ast)\rc)v_2\cdots v_t\\
&\astarrow_\phi&
u_1\cdots u_{s-1} \lc h\rc v_2\cdots v_t.
\end{eqnarray*}
Since $\lc h\rc$ is in RBNF and $u_{s-1}$ and $v_2$ are in $S(Z)$, $u_1\cdots u_{s-1} \lc h\rc v_2\cdots v_t$ is in $\kdot\rbw(Z)$. Then by Lemma~\mref{lem:redur}(\mref{it:redrb}), we have
$$\Reduce(uv)=u_1\cdots u_{s-1} \lc h\rc v_2\cdots v_t.$$
By Eq.(\mref{eq:buv}) and Lemma~\mref{lem:redur}.(\mref{it:redeq}), we have $$\Reduce(\lc u_s^\ast\rc\lc v_1^\ast\rc)=\Reduce(\lc B( u_s^\ast, v_1^\ast)\rc)=\lc h\rc.$$
Since $\Reduce=\alpha\circ \eta$ is an operated $\bfk$-module morphism, $\Reduce(\lc B( u_s^\ast, v_1^\ast)\rc)=\lc\Reduce( B( u_s^\ast, v_1^\ast))\rc$. So $\lc h\rc=\lc\Reduce( B( u_s^\ast, v_1^\ast))\rc$.
Thus, we get
\begin{eqnarray*}
\Reduce(uv)&=&\Reduce(u_1\cdots u_{s-1}\lc u_s^\ast\rc\lc v_1^\ast\rc v_2\cdots v_t)\\
&=&u_1\cdots u_{s-1}\lc h\rc v_2\cdots v_t\\
&=&u_1\cdots u_{s-1}\lc\Reduce
(B(u_s^\ast,v_1^\ast))\rc v_2\cdots v_t\\
&=&u_1\cdots u_{s-1}(u_s\diamondsuit v_1)v_2\cdots v_t \qquad (\text{ by Eq.~(\mref{eq:redu}))}\\
&=:&u\diamondsuit v.
\end{eqnarray*}
Hence, $u\diamondsuit'v=\Reduce(uv)=u\diamondsuit v$ for any $u,v\in\rbw(Z)$. Then $(\kdot\rbw(Z), \diamondsuit, P_r)=(\kdot\rbw(Z),\diamondsuit',P_r')$ and hence is a free $\phi$-algebra.
\end{proof}

\section{Applications to Conjecture~\mref{con:rbt}}
\mlabel{sec:evi} We next construct a monomial order on $\frakM(Z)$ that is compatible with the linear operators in Conjecture~\mref{con:rbt}. This allows us to show that these operators are indeed Rota-Baxter type operators as claimed by the conjecture. At the same time this gives, in one stroke, an explicit construction of free objects in the categories of algebras with any of these operators. In the case of the Rota-Baxter operator, Nijenhuis operator or TD operator, such a construction was obtained previously by different methods~\mcite{EG1,Gub,LG,ZhC,ZG}. See~\mcite{BCQ} for the construction of free Rota-Baxter algebras by the method of Gr\"obner-Shirshov basis.

\subsection{Monomial order on $\frakM(Z)$}
\mlabel{ss:mon}

We now construct a monomial order on $\frakM(Z)$.

Let $Z$ be a set with a well order $\leq_Z$. For $u=u_1\cdots u_r\in M(Z)$ with $u_1,\cdots,u_r\in Z$, define $\deg_Z(u)=r$.  Note that $\deg_Z(1)=0$. Define the {\bf degree lexicographical order} $\leq_{\dl}$  on $M(Z)$ by taking, for  any $u=u_1\cdots u_r,v=v_1\cdots v_s \in M(Z)\backslash\{1\}$, where $u_1,\cdots,u_r,v_1,\cdots, v_s\in Z$,
$$ u\leq_{\dl}v\Leftrightarrow \left\{ \begin{array}{l}
\deg_Z(u)<\deg_Z(v),\\
\text{or}~ \deg_Z(u)=\deg_Z(v)(=r) ~ \text{and} ~ u_1\cdots u_r\leq_{lex} v_1\cdots v_r,\end{array}\right.$$
where $\leq_{lex}$ is the lexicographical order on $M(Z)$, with the convention that the empty word $1\leq_{\dl} u$ for all $u\in M(Z)$.
Then we have
\begin{lemma} \mcite{BN}
If $\leq_Z$  is a well order on $Z$, then $\leq_{\dl}$ is a well
order on
 $M(Z)$.\mlabel{le:dlexord}
\end{lemma}
\begin{defn}
Let $Y$ be a nonempty set.
\begin{enumerate}
\item
A {\bf preorder} or {\bf quasiorder} $\leq_Y$ on $Y$ is a binary relation  that is reflexive and transitive, that is, for all $x,y,z\in Y$, we have
\begin{enumerate}
\item
$x\leq_Y x;$ and
\item
if $x\leq_Y y, y\leq_Y z,$ then $x\leq_Y z.$
\end{enumerate}
We denote $x=_Y y$ if $x\leq_Y y$ and $x\geq_Y y$. If $x\leq_Y y$ but $x\neq_Y y$, we write $x<_Y y$ or $y>_Y x$.
\item
A {\bf pre-linear order} $\leq_Y$ on $Y$ is a preorder $\leq_Y$ such that either $x\leq_Y y$ or $x\geq_Y y$ for all $x,y\in Y$.
\end{enumerate}
\end{defn}
We  define the composition of two or more preorders.
\begin{defn}
\begin{enumerate}
\item
Let $k\geq 1$ and let $\leq_{\alpha_i}, 1\leq i\leq k,$ be preorders on a set $Y$. Let $u,v\in Y$. Recursively define
\begin{equation}
u\leq_{\alpha_1,\cdots,\alpha_k} v \Leftrightarrow
\left\{\begin{array}{ll} u<_{\alpha_1} v, \\
\text{or } u=_{\alpha_1} v \text{ and } u\leq_{\beta} v, \end{array} \right. \mlabel{eq:compord}
\end{equation}
where $\leq_{\beta}:=\leq_{\alpha_2,\cdots,\alpha_k}$ is defined by the induction hypothesis, with the convention that $\leq_\beta$ is the trivial relation when $k=1$, namely $u\leq_\beta v$ for all $u, v\in Y$.
\item
Let $k\geq 1$ and let $(Y_i,\leq_{Y_i}), 1\leq i\leq k$, be partially ordered sets. Define the {\bf lexicographical product order } $\leq_{\clex}$ on the cartesian product $Y_1  \times Y_2\times\cdots \times Y_k$ by  recursively  defining  \begin{equation}(x_1,x_2,\cdots,x_k) \leq_{\clex}(y_1,y_2, \cdots,y_k)\Leftrightarrow
\left\{\begin{array}{ll} x_1<_{Y_1}y_1,\\
\text{or}~ x_1= y_1 ~\text{and}~ (x_2,\cdots,x_k)\leq_{\clex}(y_2, \cdots,y_k),
\end{array}\right.\mlabel{eq:carord}
\end{equation}
where $(x_2,\cdots,x_k)\leq_{\clex} (y_2,\cdots,y_k)$ is defined by the induction hypothesis, with the convention that $\leq_{\clex}$ is the trivial relation  when $k=1$.
\item
Let $u=u_0\lc \mpu_1\rc u_1\lc \mpu_2\rc \cdots\lc  \mpu_r \rc u_r$,  $ v=v_0\lc \mpv_1\rc v_1\lc \mpv_2\rc \cdots\lc \mpv_s\rc v_s \in \frakM(Z)$, where $u_0,u_1,\cdots,u_r$, $v_0,v_1,\cdots,v_s\in M(Z)$
and $\mpu_1,\mpu_2,\cdots, \mpu_r, \mpv_1,\mpv_2,\cdots,\mpv_s\in \frakM(Z)$.
Define
\begin{equation} u\leq_{\dgp}v \Leftrightarrow \deg_P (u) \leq
\deg_P (v), \mlabel{eq:dgpord}
\end{equation}
where the {\bf $P$-degree} $\deg_P(u)$ of $u$ is the number of occurrence of $P=\lc\ \rc$ in $u$. Define
\begin{equation}
u\leq_{\brp}v \Leftrightarrow r\leq s \quad (\text{that is } |u|_P\leq |v|_P),
\mlabel{eq:brord}
\end{equation}
where $|u|_P$ is the $P$-breadth defined after Eq.~(\mref{eq:decom}).
\end{enumerate}
\end{defn}

\begin{lemma}
\begin{enumerate}
\item
Let $k\geq 1$. Let $\leq_{\alpha_1},\cdots, \leq_{\alpha_{k-1}}$ be pre-linear orders on $Z$ with descending chain condition (i.e., each decreasing chain in $(Z,\leq_{\alpha_i})$ stabilizes after finitely many steps), and $\leq_{\alpha_k}$ is a well order on $Z$. Then the order  $\leq_{\alpha_1,\cdots,\alpha_k}$ is a well order on $Z$. \mlabel{it:well1}
\item
{\bf{\cite{Ha}}} Let $\leq_{Y_i}$ be a well order on $Y_i, 1\leq i\leq k, k\geq 1$. Then the lexicographical product order ~$\leq_{\clex}$ is a well order on the cartesian product $Y_1 \times Y_2\times\cdots \times Y_k$.\mlabel{it:well2}
\item
The  pre-linear orders
$\leq_{\dgp}$ and $\leq_{\brp}$ satisfy the descending chain condition on $\frakM(Z)$.\mlabel{it:well3}
\end{enumerate}
\mlabel{le:wellord}
\end{lemma}
\begin{proof}
\noindent (\mref{it:well1}) We prove the claim by induction on $k\geq 1$. When $k=1$, $\leq_{\alpha_1,\cdots,\alpha_k}$ is a well order by the assumption. Assume that the claim holds for $k\leq n$ where $n\geq 1$ and consider the case when $k=n+1$. Denote $\leq_{\beta}=\leq_{\alpha_2,\cdots,\alpha_{n+1}}$. Then $\leq_\beta$ is a well order by the induction hypothesis. We first show that $\leq_{\alpha_1,\beta}$ is a linear order.
For all $u,v\in Z$, we have $u\leq_{\alpha_1} v$ or $ u\geq_{\alpha_1}v$ since $\leq_{\alpha_1}$ is a pre-linear order. If $u\neq_{\alpha_1} v$, then we have $u <_{\alpha_1} v$ or $u>_{\alpha_1} v$. Thus we obtain $u<_{\alpha_1,\beta}v$  or $u>_{\alpha_1,\beta} v$ and we are done.
If $u=_{\alpha_1}v$, then $u\geq_\beta v$ or $u\leq_\beta v$ since $\leq_\beta$ is a linear order. Thus we have $u\geq_{\alpha_1,\beta}v$ or $u\leq_{\alpha_1,\beta} v$. Therefore, $\leq_{\alpha_1,\beta}$ is a linear order.

Thus we just need to prove that the order $\leq_{\alpha_1,\cdots,\alpha_k}$ satisfies the descending chain condition.  Suppose that $u_1 \geq_{\alpha_1,\beta} u_2 \geq_{\alpha_1,\beta} \cdots$. Since $\leq_{\alpha_1}$ has descending chain condition, there exists $t\geq 1$ such that $u_t=_{\alpha_1}u_{t+1}=_{\alpha_1}\cdots$. Thus we must have $u_t\geq_{\beta}u_{t+1}\geq_{\beta}\cdots.$ By the induction hypothesis, $\geq_{\beta}$ is a well order and hence satisfies the descending chain condition. Thus the descending chain $u_1\geq_{\alpha_1,\beta} u_2\geq_{\alpha_1,\beta}\cdots$ stabilizes after finite steps. Therefore, $\leq_{\alpha_1,\beta}$ is a well order. This completes the induction. \noindent

(\mref{it:well2}) is proved in~ \cite[Chapter 4, Theorem~1.13]{Ha}. \noindent

(\mref{it:well3}) follows since both $\deg_P(u)$ and $|u|_P$ (defined after Eq.~(\mref{eq:decom})) take values in $\ZZ_{\geq 0}$ and hence satisfy the descending chain condition.
\end{proof}

For $m\geq0$, denote
$$\frakM^m(Z)=\{u\in \frakM(Z)~|~|u|_P=m\},
$$
Also denote $\frakM_n^m(Z)=\frakM_n(Z)\cap \frakM^m(Z)$, where $n\geq0$. For $u, v\in \frakM_n^m(Z)$, define
\begin{equation}
u\leq_{\lex_n} v \Leftrightarrow
(\mpu_1,\mpu_2,\cdots,\mpu_m,u_0,\cdots, u_{m}) \leq_{\clex} (\mpv_1,\mpv_2,\cdots,\mpv_m,v_0, \cdots,v_{m})
\end{equation}
We now define a well order on $\frakM_n(Z), n\geq 0,$ by the following recursion.
\begin{enumerate}
\item
Let  $u, v\in\frakM_0(Z)\backslash\{1\}=S(Z)$. Let $u=u_0u_1\cdots u_r$ and $ v=v_0v_1\cdots v_s$, where $u_0, u_1,\cdots, u_r,$ $v_0, v_1,$ $\cdots,v_s\in Z$. Then define $$u\leq_0 v \Leftrightarrow u \leq_{\dl} v.$$
By Lemma~\mref{le:dlexord}, $\leq_0$  is a well order on $\frakM_0(Z)$.
\item
Suppose that a well order $\leq_n$ has been defined on $\frakM_n(Z)$ for $n\geq1$. Let $u, v\in \frakM_{n+1}(Z)= M(Z\sqcup\lc\frakM_n(Z)\rc)$. Let
$u=u_0\lc \mpu_1\rc u_1\lc \mpu_2\rc \cdots\lc \mpu_r\rc u_r$ and $ v=v_0\lc \mpv_1\rc v_1\lc \mpv_2\rc \cdots\lc  \mpv_s\rc v_s$, where $u_0,u_1,\cdots,u_r$, $v_0,v_1,\cdots,v_s\in M(Z)$ and $\mpu_1,\mpu_2,\cdots, \mpu_r, \mpv_1,\mpv_2,\cdots,\mpv_s\in \frakM_n(Z)$.
First suppose that $r=s=m$ for some $m\geq 0$. Then define
\begin{equation} u\leq_{\lex_{n+1}} v\Leftrightarrow
(\mpu_1, \mpu_2 \cdots,\mpu_m,u_0,\cdots,u_m) \leq_{\clex} (\mpv_1,\mpv_2,\cdots, \mpv_m,v_0,\cdots,v_m). \mlabel{eq:clex}
\end{equation}
Since the order $\leq_n$ (resp. $\leq_{\dl}$) is a  well order on $\frakM_n(Z)$ by the induction hypothesis (resp. on M(Z)), the order $\leq_{\clex}$ is a well order on $\frakM_n(Z)^m\times M(Z)^{m+1}$ by Lemma~\mref{le:wellord}(\mref{it:well2}), and hence the order $\leq_{\lex_{n+1}}$ is a well order on $\frakM_{n+1}^m(Z)$. In general define
 \begin{equation}
  u\leq_{n+1} v~ \Leftrightarrow u\leq_{\dgp,\brp, \lex_{n+1}} v
  \Leftrightarrow \left\{
\begin{array}{l}
u<_{\dgp} v, \\
\text{or}~ u=_{\dgp} v ~\mbox{and}~ u<_{\brp}v, \\ \text{or}~
u=_{\dgp} v, u=_{\brp}v  ~\mbox{and}~ u \leq_{\lex_{n+1}} v. \\
\end{array}\right.\mlabel{eq:leq}
\end{equation}
Since the orders $\leq_{\dgp},\leq_{\brp}$ satisfy the descending chain condition and $ \leq_{\lex_{n+1}}$ is a well order, we conclude that $\leq_{n+1}$ is a well order on $\frakM_{n+1}(Z)$ from Lemma~\mref{le:wellord}(\mref{it:well1}).
\end{enumerate}

From the definition of $\leq_n$ for $n\geq0$, we see that the restriction of the order $\leq_{n+1}$ to $\frakM_n(Z)$ equals to the order $\leq_n$. Thus we can  define the order
\begin{equation}
\leq_{\db}:=\dirlim \leq_n=\bigcup\leq_n \mlabel{eq:dbord}
\end{equation}
on the direct system $\frakM(Z)=\dirlim \frakM_n(Z)$.

We note that if $\leq_{n+1}$ were defined by $\leq_{n+1}':=\leq_{\brp,\lex_{n+1}}$, then the resulting order $\leq_{\db}'=\dirlim \leq_n'$ would not be a well order. For example $\lc u\rc \lc v\rc
>_{\db}' \lc \lc u\rc \lc v\rc\rc >_{\db}' \lc\lc \lc u\rc\,\lc
v\rc\rc\rc>_{\db}' \cdots$ is an infinite decreasing chain.
\begin{lemma}
The order $\leq_{\db}$ is a well order on $\frakM(Z)$. \mlabel{le:wlord}
\end{lemma}

\begin{proof}
Since $\leq_{\db}$ is a linear order on $\frakM(Z)$ as a direct limit of linear orders $\leq_n$, we only need to verify that $\leq_{\db}$ satisfies the descending chain condition.

Let a descending chain $v_1\geq_{\db}v_2\geq_{\db}\cdots$ in $\frakM(Z)$ be given. Since $\geq_{\dgp}$ satisfies the descending chain condition, there exists  $\ell\geq0$ such that $v_{\ell}=_{\dgp}v_{\ell+1}
=_{\dgp}\cdots$. Thus $\deg_P(v_{\ell})=\deg_P(v_{\ell+1})=
\cdots=k$ for some $k\geq 0$. Then  $v_{\ell},v_{\ell+1},\cdots$ are in
$$ \frakM_{(k)}:=\{ u\in \frakM(Z)\,|\, \deg_P(u) \leq k\}.$$
Note that $\frakM_{(k)}\subseteq \frakM_k(Z)$. Since the restriction of $\leq_{\db}$ to $\frakM_k(Z)$ and hence to $\frakM_{(k)}$ is $\leq_k$ which, as shown above, satisfies the descending chain condition, the chain $v_{\ell}\geq_{\db}v_{\ell+1}
\geq_{\db}\cdots$ stabilizes after finite steps. Therefore,
$v_1\geq_{\db}v_2\geq_{\db}\cdots$ stabilizes after finite steps.
\end{proof}

\begin{defn}
A well order $\leq_\alpha$ on $\frakM(Z)$ is called {\bf bracket compatible} (resp. {\bf left (multiplication) compatible}, resp. {\bf right (multiplication) compatible}) if
$$ u\leq_\alpha v \Rightarrow \lc u\rc \leq_\alpha \lc v\rc,\
\text{(resp. } wu \leq_\alpha wv,\  \text{resp. } uw \leq_\alpha vw, \quad \text{for all } w\in \frakM(Z)).$$
\end{defn}

\begin{lemma}
A well order $\leq$ is a monomial order on $\frakM(Z)$ if and only if $\leq$ is bracket compatible,  left compatible and right compatible. \mlabel{le:eqmord}
\end{lemma}
\begin{proof}
Suppose that a well order $\leq$ is a monomial order. Let $u, v\in \frakM(Z)$ with $u\leq v$. By taking $q=\lc \star\rc, w\star$ and $\star w$ with $w\in \frakM(Z)$, we obtain $\lc u\rc \leq \lc v\rc$, $wu\leq wv$ and $uw\leq vw$ respectively, proving the bracket compatibility, left compatibility and right compatibility.

Conversely, suppose that a well order $\leq$ is bracket compatible, left compatible and right compatible. Let $u, v\in \frakM(Z)$ with $u\leq v$ and let $q\in \frakM^\star(Z)$ be given. We prove $q|_u \leq q|_v$ by induction on the depth $\dep(q)\geq 0$ of $q$. If $\dep(q)=0$, then $q\in M(Z\cup \{\star\})$ and hence $q=w_1\star w_2$ where $w_1, w_2\in M(Z)$.  Then by the left and right compatibility, we have $q|_u\leq q|_v$. Assume that $q|_u\leq q|_v$ has been proved for $q\in \frakM^\star(Z)$ with $\dep(q)\leq n$ where $n\geq 0$ and consider $q\in \frakM^\star(Z)$ with $\dep(q)=n+1$. If the $\star$ in $q$ is not in a bracket, then $q=w_1\star w_2$, where $w_1, w_2\in \frakM(Z)$. Then we have $q|_u\leq q|_v$ by the left and right compatibility. If the $\star$ in $q$ is in a bracket, then $q=w_1\lc q'\rc w_2$ with $w_1, w_2\in \frakM(Z)$ and $q'\in \frakM^\star(Z)$ with $\dep(q')=n$. Hence by the induction hypothesis, we have $q'|_u\leq q'|_v$. Then by bracket, left and right compatibility of $\leq$, we further have $q|_u\leq q|_v$,
completing the induction.
\end{proof}

\begin{theorem}
The order $\leq_{\db}$ is a monomial order on $\frakM(Z)$. \mlabel{thm:mord}
\end{theorem}

\begin{proof}
By Lemma~\mref{le:wlord}, the order $\leq_{\db}$ is a well order on $\frakM(Z)$. So we just need to prove that $\leq_{\db}$ is bracket compatible, left compatible and right compatible by Lemma~\mref{le:eqmord}.  Let $u,v\in\frakM(Z)$. Then there exists a natural number $n$ such that $u,v\in\frakM_n(Z)$. Suppose that
\begin{equation}
u=u_0\lc \mpu_1\rc u_1\lc \mpu_2\rc \cdots\lc \mpu_r\rc u_r, \  v=v_0\lc \mpv_1\rc v_1\lc \mpv_2\rc \cdots\lc \mpv_s\rc v_s, \mlabel{eq:uv}
\end{equation}
where $u_0,u_1\cdots,u_r$, $v_0,v_1\cdots,v_s\in M(Z)$ and $\mpu_1,\mpu_2,\cdots, \mpu_r, \mpv_1,\mpv_2,\cdots,\mpv_s\in \frakM_{n-1}(Z).$
First we prove that $\leq_{\db}$ is bracket compatible. By the definition of $\leq_{\db}$,
 we just need to prove
$$u\leq_n v\Rightarrow \lc u\rc \leq_{n+1} \lc v\rc.$$ Suppose that
$u\leq_n v$. By the definition of $\leq_n$, we have the following three cases.

\noindent {\bf Case 1 } $u<_{\dgp}v $. This means that $\deg_P (u) <\deg_P (v) $. Then we have $\deg_P(\lc u\rc)=\deg_P(u) +1 < \deg_P (v)  + 1=\deg_P(\lc v\rc)$. This shows that $\lc u\rc \leq_{n+1} \lc v\rc$ by the definition of $\leq_{n+1}$.

\noindent {\bf Case 2 } $u=_{\dgp} v$ and $u<_{\brp} v$. Then we have $\lc u \rc =_{\dgp} \lc v\rc$. Since the $P$-breadth of $\lc u\rc$ and $\lc v\rc$ are equal to $1$, we have $\lc u\rc=_{\brp} \lc v\rc$. Since $u\leq_n v$ and by the definition of $\lex_{n+1}$ (that is, by Eq.~(\mref{eq:clex})), $\lc u\rc  \leq_{\lex_{n+1}} \lc v\rc$. Then by Eq.$(\mref{eq:leq}_{n+1}$, we have $\lc u\rc \leq_{n+1} \lc v\rc$.

\noindent {\bf Case 3 } $u=_{\dgp} v$, $u=_{\brp} v$ and $u\leq_{\lex_n}v$. This means that $\deg_P(u) =\deg_P (v)$, $r=s$ and $(\mpu_1,\mpu_2,\cdots,\mpu_r,u_0, \cdots,u_r) \leq_{\clex} (\mpv_1,\mpv_2,\cdots, \mpv_s,v_0,\cdots,v_s)$.
Then we have $ \lc u\rc=_{\dgp}\lc v\rc $ and $\lc u\rc=_{\brp} \lc v\rc$.  Thus we have $\lc u\rc \leq_{n+1}\lc v\rc$ since $u\leq_n v$.
\medskip

Hence the order $\leq_{\db}$ is bracket compatible. Next, we prove that $\leq_{\db}$ is left compatible. For any $w\in \frakM(Z)$, take $m\geq n$ such that $wu, wv\in \frakM_m(Z)$. Denote  $$w=w_0\lc \mpw_1 \rc w_1 \lc \mpw_2 \rc\cdots \lc \mpw_t\rc w_{t},$$
with $w_0, w_1, \cdots, w_{t}\in M(Z)$ and $\mpw_1, \mpw_2,\cdots, \mpw_t \in \frakM_{m-1}(Z)$.
Then with the notation in Eq.~(\mref{eq:uv}), we have
$$w u=w_0\lc \mpw_1 \rc w_1\lc\mpw_2\rc \cdots \lc \mpw_t\rc
w_{t}u_0\lc \mpu_1\rc u_1\lc \mpu_2\rc \cdots\lc \mpu_r\rc u_r$$ and
$$w v=w_0\lc \mpw_1 \rc w_1\lc \mpw_2\rc \cdots \lc \mpw_t\rc
w_{t}v_0\lc \mpv_1\rc v_1\lc \mpv_2\rc \cdots\lc \mpv_s\rc v_s.$$

Suppose that $u\leq_n v$. To prove $w u\leq_m w v$, we only need to consider the following  three cases.

\noindent {\bf Case 1} $u<_{\dgp} v$. Then we have $\deg_P(w u)=\deg_P(w)+\deg_P (u) <\deg_P(w)+\deg_P (v) =\deg_P(w v)$, and hence $w u<_{\dgp} w v$. Thus we get $w u\leq_m w v$.
\smallskip

\noindent {\bf Case 2} $u=_{\dgp} v$ and $u<_{\brp} v$. Then we obtain $\deg_P(w u)= \deg_P(w v)$ and $t+r <t+ s$. This means that $wu=_{\dgp} wv$ and $wu<_{\brp}wv$, and hence $w u\leq_m w v$.
\smallskip

\noindent {\bf Case 3} $u=_{\dgp} v$, $u=_{\brp}v$ and $u\leq_{\lex_n}v$. Then we have $wu=_{\dgp}w v$, $wu=_{\brp}wv$
and
$$(\mpu_1,\mpu_2,
\cdots,\mpu_r,u_0,\cdots,u_r) \leq_{\clex} (\mpv_1, \mpv_2,\cdots, \mpv_s,v_0,\cdots,v_s).$$
Thus we obtain
\begin{eqnarray*}
&&(\mpw_1, \mpw_2,\cdots,\mpw_t, \mpu_1,\mpu_2,\cdots, \mpu_r,w_0,\cdots,w_{t}u_0,
\cdots,u_r) \\
&\leq_{\clex}&(\mpw_1, \mpw_2,\cdots,\mpw_t, \mpv_1,\mpv_2,\cdots, \mpv_s,w_0,\cdots,w_{t}v_0, \cdots,v_s).
\end{eqnarray*}
Hence we get $wu\leq_{\lex_m} wv$. Thus we have $w u\leq_m w v$. This completes  proof of left compatibility of order $\leq_{\db}$. The proof of the right compatibility is the same, completing the proof.
\end{proof}

\subsection{Consequences on Rota-Baxter type operators}
\mlabel{ss:crb} We now verify that all the operators listed in Conjecture~\mref{con:rbt} are Rota-Baxter type operators and obtain the free objects in the corresponding categories of algebras.

\begin{prop}
Let $\phi(x,y)=\lc x\rc \lc y\rc - \lc B(x,y)\rc$  where $B(x,y)$ is in RBNF and has total operator degree $\leq 1$. More precisely,
\begin{eqnarray}
B(x,y)&:=&a_{0}y\lc x\rc +a_{1}x\lc y \rc +b_{0}\lc y\rc x+ b_{1} \lc x \rc y +c_{0}\lc yx\rc \notag \\& &+c_{1}  \lc x y \rc +d_{0}x \lc 1
 \rc y + d_{1} x y  \lc 1 \rc   + d_{2} y\lc 1 \rc x
  \mlabel{eq:mform}\\
 & &+d_{3} y x\lc 1\rc + d_{4}\lc 1\rc x y  + d_{5} \lc 1\rc  y x
+\varepsilon_{0}
 y x+ \varepsilon_{1}x y
\notag
\end{eqnarray}
where $a_{i}, b_{j}, c_{k}, d_{\ell}, \varepsilon_{t}\in \bfk$ $(0 \leq i, j, k, t \leq 1,
0\leq \ell \leq 5)$. Then $\Pi_\phi$ is compatible with the monomial order $\leq_{\db}$ in Theorem~\mref{thm:mord}. \mlabel{it:roc2} \mlabel{pp:rocomp}
\end{prop}

\begin{proof}
For each of the monomials $g(x,y)$ in Eq.~(\mref{eq:mform}),  we have $$\deg_P(g):=\deg_P(g(x,y))\leq 1.$$
Hence for any $u, v\in \frakM(Z)$, we have
$$\deg_P(\lc g(u,v)\rc) = \deg_P(u)+\deg_P(v)+\deg_P(g) +1\leq
\deg_P(u)+\deg_P(v)+2= \deg_P(\lc u\rc \lc v\rc).$$ Further, $ \lc g(u,v)\rc <_\brp \lc u\rc\lc v\rc$. Thus $\lc g(u,v)\rc <_{\db} \lc u\rc \lc v\rc$, and hence $\overline{\lc B(u,v)\rc}<_{\db}\lc u\rc\lc v\rc$.
\end{proof}

\begin{theorem}
Let $\phi:=\lc x\rc \lc y\rc -\lc B(x,y)\rc$, where $B(x,y)$  is in the list in Conjecture~\mref{con:rbt}.
\begin{enumerate}
\item
$B(x,y)$ and the corresponding operator $P$  are of Rota-Baxter type. \mlabel{it:exam1}
\item
All the statements in Theorem~\mref{thm:gsrb} hold for $B(x,y)$. \mlabel{it:exam2}
\item
The free $\phi$-algebra on a well ordered set $Z$ has its explicit construction in Theorem~\mref{thm:free}. \mlabel{it:exam3}
\end{enumerate}
\mlabel{thm:exam}
\end{theorem}
In particular, the theorem holds for the Rota-Baxter operator, the Nijenhuis operator and the TD operator.

\begin{proof}
(\mref{it:exam1}) First note that all the 14 expressions listed in Conjecture~\mref{con:rbt} are in RBNF. Further by Proposition~\mref{pp:rocomp}, the rewriting systems $\Pi_\phi$ from these expressions are compatible with the monomial order $\leq_{\db}$. Hence the rewriting systems are terminating by Theorem~\mref{thm:roc1}. Thus we only need to check the $\phi$-reducibility.

The expressions ($e$) and ($f$) are known to have the $\phi$-reducibility~\mcite{Ca,EG1,LG}. The $\phi$-reducibility of expressions ($a$), ($b$), ($m$) and $(n)$ are easy to check.

For expression ($c$):  $B(x,y):=x\lc y\rc +y\lc x\rc$, we have
\begin{eqnarray*}
B(B(u,v),w)&=&B(u,v)\lc w\rc +w\lc B(u,v)\rc\\
&=&(u\lc v\rc +v\lc u\rc)\lc w \rc +w(\lc u\lc v\rc +v\lc u\rc
\rc)\\
&\astarrow_\phi
&u\lc v\lc w \rc\rc+  u\lc w\lc v\rc\rc + v\lc u\lc w \rc\rc +v\lc w\lc u\rc\rc+w\lc u\lc v\rc\rc +w\lc v\lc u\rc \rc.
\end{eqnarray*}
and
\begin{eqnarray*}
B(u,B(v,w))&=&u\lc B(v,w) \rc + B(v,w)\lc u \rc\\
&=&u(\lc v\lc w\rc +w\lc v\rc \rc) +(v\lc w\rc +w\lc v\rc) \lc
u\rc\\
&\astarrow_\phi &u\lc v\lc w \rc\rc+  u\lc w\lc v\rc\rc +v\lc w\lc u\rc\rc+ v\lc u\lc w \rc\rc +w\lc v\lc u\rc \rc+w\lc u\lc v\rc\rc.
\end{eqnarray*}

Thus $ B(B(u,v),w)-B(u,B(v,w))\astarrow_\phi  0$ by Theorem~\mref{thm:downarrow}.
\smallskip

The verification of expression ($d$) is similar to expression ($c$).
\smallskip

We next check expression ($k$). Then the expressions ($g$), ($h$) and ($i$) are subexpressions of expression ($k$) and can be verified similarly. For expression ($k$), we have

{\allowdisplaybreaks
\begin{eqnarray*}
B(B(u,v),w)&=&B(u,v)\lc w\rc +\lc B(u,v)\rc w- B(u,v)\lc 1\rc w- \lc
B(u,v)w\rc + \lambda B(u,v)w \\
&\astarrow_\phi & u\lc v\lc w\rc\rc+u\lc\lc v\rc w\rc-u\lc v\lc 1\rc w\rc- u\lc \lc vw\rc\rc +\lambda u\lc vw\rc +\lc u\rc v \lc w
\rc\\
&&- u\lc 1\rc v\lc w\rc-\lc uv\lc w \rc\rc+ \lc uv\lc 1\rc w \rc
+\lc\lc uvw \rc\rc-2\lambda\lc uvw\rc+\lambda uv\lc w\rc \\
&&+\lc u\lc v\rc\rc w+ \lc\lc u\rc v\rc w-\lc u\lc 1\rc v \rc
w-\lc\lc uv\rc\rc w +\lambda \lc uv \rc w  -\lc u\rc v\lc 1\rc w\\
&&+u\lc 1\rc v\lc 1\rc w -\lambda uv\lc 1\rc w -\lc u\lc v\rc w\rc
-\lc\lc u\rc vw\rc+\lc u\lc 1\rc vw\rc +\lambda\lc u\rc vw\\
&&-\lambda u\lc 1\rc vw +\lambda^{2}uvw
\end{eqnarray*}
}

and {\allowdisplaybreaks
\begin{eqnarray*}
B(u,B(v,w))&=&u\lc B(v,w) \rc + \lc u \rc B(v,w)- u\lc 1\rc
B(v,w)-\lc u B(v,w) \rc+\lambda u B(v,w)\\
&\astarrow_\phi & u\lc v\lc w\rc\rc +u\lc \lc v\rc w \rc-u\lc v\lc 1\rc w\rc-u\lc \lc vw\rc\rc +\lambda u\lc vw\rc+\lc u\rc v\lc w
\rc\\
&&+ \lc u \lc v\rc \rc w + \lc \lc u\rc v\rc w-\lc u\lc 1\rc v\rc w-
\lc \lc uv \rc \rc w+ \lambda \lc uv \rc w-\lc u\rc v \lc 1\rc w\\
&& -\lc\lc u\rc vw\rc +\lc u\lc 1\rc vw\rc +\lc \lc uvw \rc\rc-2\lambda \lc uvw \rc +\lambda \lc u\rc vw -u\lc 1\rc v\lc
w\rc\\
&&+ u\lc 1\rc v\lc 1 \rc  -\lambda u\lc 1\rc vw-\lc uv\lc w\rc\rc -
\lc u\lc v\rc w\rc +\lc uv\lc 1\rc w\rc+\lambda uv \lc w\rc\\
&& -\lambda uv \lc 1\rc w  +\lambda^{2}uvw.
\end{eqnarray*}
}
 Now  the $i$-th term in the expansion of $B(B(u,v),w)$ matches
with the
 $\sigma(i)$-th term in the expansion of $B(u,B(v,w))$. Here the
permutation $\sigma\in S_{26}$ is
 \begin{equation*}
\left ( \begin{array}{c} i\\\sigma(i)\end{array}\right) = \begin{array}{l} \left( \begin{array}{cccccccccccccccccccccccccccccccccccccc}
1&2&3&4&5&6&7&8&9&10&11&12&13&14\\
1&2&3&4&5&6&18&21&23&15&16&24&7&8
\end{array} \right ) \\
\medskip
\left( \begin{array}{ccccccccccccccccccccccccccccccccccccc}   15
&16&17&18&19&20&21&22&23&24&25&26\\
9 &10&11&12&19&25&22&13&14&17&20&26
 \end{array}  \right )
\end{array}
\end{equation*}
Thus $B(B(u,v),w)-B(u,B(v,w)) \astarrow_\phi 0.$

\smallskip
We finally verify expression ($l$). Then the remaining expression ($j$), obtained from expression ($l$) by replacing $\lc 1\rc xy$ with $xy\lc 1\rc$, is similarly verified.
Expression ($l$) is $B(x,y):={x\lc y\rc +\lc x\rc y}$ ${-x\lc 1\rc y-\lc 1\rc  xy + \lambda xy}.$ So we have {\allowdisplaybreaks
\begin{eqnarray*}
B(B(u,v ), w)&=&B(u,v)\lc w\rc +\lc B(u,v )\rc w-B(u,v)\lc 1\rc  w-
\lc 1\rc B(u,v )w+\lambda B(u,v)w\\
&\astarrow_\phi & u\lc v\lc w\rc\rc+ u\lc \lc v\rc w\rc- u\lc v\lc 1\rc w \rc-u\lc \lc 1\rc vw\rc+\lambda u\lc vw\rc+\lc u\rc v\lc
w\rc \\
&&-u\lc 1\rc v\lc w\rc-\lc 1\rc uv \lc w \rc + \lambda uv \lc w\rc+\lc u\lc v\rc \rc w+ \lc \lc u\rc v\rc w-\lc u\lc 1\rc v\rc w
\\
&&-\lc \lc 1\rc uv\rc w +\lambda \lc uv\rc w-u\lc \lc v\rc \rc w+u\lc \lc 1\rc v\rc w - \lc u\rc v\lc 1\rc
w+u\lc 1\rc v\lc 1\rc w\\
&&+\lc 1\rc uv \lc 1\rc w-\lambda uv\lc 1\rc w- \lc 1\rc u\lc v \rc w -\lc \lc u\rc\rc vw+\lc \lc 1\rc u\rc vw
+\lc 1\rc u\lc 1\rc vw\\
&& -\lambda u \lc 1\rc vw-\lambda \lc 1\rc uvw+\lambda^{2}uvw
\end{eqnarray*}
} and {\allowdisplaybreaks
\begin{eqnarray*}
B(u,B(v, w))&=&u\lc B(v,w)\rc +\lc u\rc B(v,w)- u\lc 1\rc B(v,w)-\lc
1\rc uB(v,w)+\lambda uB(v,w)\\
&\astarrow_\phi & u\lc v\lc w\rc\rc +u\lc \lc v\rc w\rc -u\lc v\lc 1\rc w\rc -u\lc\lc 1\rc vw\rc+ \lambda u\lc vw\rc +\lc u\rc
v\lc w\rc \\
&&+\lc u\lc v\rc \rc w+ \lc \lc u\rc v\rc w -\lc u\lc 1\rc v\rc
w-\lc \lc 1\rc uv\rc w+\lambda \lc uv\rc w-\lc u\rc v\lc1\rc w \\
&&-\lc \lc u\rc \rc vw+\lc \lc 1\rc u \rc vw-u\lc 1\rc v\lc w\rc -u\lc\lc v\rc \rc w +u\lc \lc 1\rc v\rc w + u\lc 1\rc v\lc 1\rc w\\
&&-\lc 1\rc uv\lc w\rc -\lc 1\rc u\lc v\rc w +\lc 1\rc uv\lc 1\rc
w+\lc 1\rc u\lc 1\rc vw -\lambda \lc 1\rc uvw +\lambda uv \lc w\rc\\
&& -\lambda uv\lc 1\rc w-\lambda u\lc 1\rc vw+\lambda ^{2} uvw.
\end{eqnarray*}
}
Note that the $i$-th term in the expansion of $B(B(u,v),w)$ matches with the $\sigma(i)$-th term in the expansion of $B(u,B(v,w))$. Here the permutation $\sigma\in S_{27}$ is defined by
 \begin{equation*}
\left ( \begin{array}{c} i\\\sigma(i)\end{array}\right) = \begin{array}{l} \left (
\begin{array}{cccccccccccccccccccccccccccccccccccccc}
1&2&3&4&5&6&7&8&9&10&11&12&13&14\\
1&2&3&4&5&6&15&19&24&7&8&9&10&11
\end{array}
\right) \\
\medskip
\left( \begin{array}{ccccccccccccccccccccccccccccccccccccc}
 15&16&17&18&19&20&21&22&23&24&25&26&27\\
 16&17&12&18&21&25&20&13&14&22&26&23&27 \end{array} \right )
 \end{array}
\end{equation*}
Thus $B(B(u,v),w)-B(u,B(v,w)) \astarrow_\phi 0.$
\smallskip

\noindent (\mref{it:exam2}) follows from Item~(\mref{it:exam1}), Corollary~\mref{co:rbcon} and Proposition~\mref{pp:rocomp}.
\smallskip

\noindent (\mref{it:exam3}) follows from Item~(\mref{it:exam2}) and Theorem~\mref{thm:free}.
\end{proof}

\noindent {\bf Acknowledgements}: This work was supported by the National Natural Science Foundation of China (Grant No. 11201201, 11371177 and 11371178), Fundamental Research Funds for the Central Universities (Grant No. lzujbky-2013-8),
the Natural Science Foundation of Gansu Province (Grant No. 1308RJZA112) and the National Science Foundation of US (Grant No. DMS~1001855).

\end{document}